\documentclass[a4paper,11pt]{article}

\usepackage[utf8]{inputenc}
\usepackage[T1]{fontenc}
\usepackage[english]{babel}
\usepackage{amsmath, amsfonts, amsthm,amssymb,here,dsfont, hyperref}
\usepackage{mathtools}
\usepackage{graphicx,color,subfigure,multirow, here}
\usepackage{natbib}
\usepackage{enumitem}
\usepackage{comment}
\newtheorem{theo}{Theorem}[section]

\newtheorem{prop}[theo]{Proposition}

\newtheorem{lemme}[theo]{Lemma}

\newtheorem{assumption}[theo]{Assumption}
\newcommand{\argmax}[1]{\underset{#1}{\operatorname{arg}\!\operatorname{max}}\;}
\newcommand{\argmin}[1]{\underset{#1}{\operatorname{arg}\!\operatorname{min}}\;}

\newcommand{\w}{\widehat}
\newcommand{\one}{\mathds{1}}
\newcommand{\E}{\mathbb{E}}

\renewcommand{\P}{\mathbb{P}}
\newcommand{\mfp}{\mathfrak{p}}
\newcommand{\R}{\mathbb{R}}
\newcommand{\cR}{\mathcal{R}}
\newcommand{\cS}{\mathcal{S}}
\newcommand{\cY}{\mathcal{Y}}

\usepackage[textwidth=1.8cm, textsize=tiny]{todonotes}

\title{Nonparametric plug-in classifier for multiclass classification of S.D.E. paths}

\author{Christophe Denis$^{(1)}$, Charlotte Dion-Blanc$^{(2)}$,  Eddy Ella-Mintsa$^{(1)}$, Viet Chi Tran$^{(1,3)}$
}

\begin{document}

\maketitle
\begin{center}
${(1)}$ LAMA, Université Gustave Eiffel\\ ${(2)}$ Sorbonne Université, CNRS, Laboratoire de Probabilités, Statistique et Modélisation, F-75013 Paris, France \\
${(3)}$ CRM-CNRS, Université de Montréal.
\end{center}

\date{}

%%%%%%%%%%%%%%%%%%%%%%%%%%%%%%%
\begin{abstract}
We study the multiclass classification problem where the features come from a mixture of time-homogeneous diffusions.
Specifically, the classes are discriminated by their drift functions while the diffusion coefficient is common to all classes and unknown.
In this framework, we build a plug-in classifier which relies on nonparamateric estimators of the drift and diffusion functions.
We first establish the consistency of our classification procedure under mild assumptions and then provide rates of convergence under different set
of assumptions. Finally, a numerical study supports our theoretical findings.  
\end{abstract}

\vspace*{0.25cm} 

\noindent {\bf Keywords}: Supervised learning; Multiclass classification; Nonparametric estimation; Plug-in classifier; Diffusion process\\

\noindent MSC: 62G05; 62M05; 62H30\\

%\textcolor{blue}{R-package: \href{https://github.com/Eddymichelella/classifsde.git}{PlugInSDE}}

%%%%%%%%%%%%%%%%%%%%%%%%%%
%%%%%%%%%%%%%%%%%%%%%%%%%%
\section{Introduction}
\label{sec:intro}

The massive collection of functional data has found many applications in recent years for the modeling of the joint (time)-evolution of agents -- individuals, species, particles -- that are represented by some sets of features -- time-varying variables such as geographical positions, population sizes, portfolio values etc. Examples can be found in mathematical finance \citep[see \emph{e.g.}][]{el1997backward}, biology \citep[see \emph{e.g.}][]{erban2009stochastic}, or physics \citep[see \emph{e.g.}][]{domingo2020properties}. 
This gave rise to an abundant literature on statistical methods for functional data, \citep[see \emph{e.g.}][for a review]{ramsay2005fitting, reviewfunctional}. Within this context, the study of efficient supervised classification procedures that are designed to handle temporal data is a major challenge. Indeed, usual learning algorithms such as random forests, kernel methods or neural networks are not directly tailored to take into account the temporal dependency of the data. Recently, this question has drawn a lot of attention, see 
\citet{villa2008,baillo2011classification,wang2020deep,de2021depth,neuralsde} any references therein.

In the present paper, we tackle the multiclass classification problem where the features belong to a particular family of functional data, namely trajectories, whose  temporal dynamic is modelled by stochastic differential equation.
In this framework, we propose a nonparametric plug-in type procedure for such data generated by diffusion processes observed at discrete time. Hence, our work takes place in the high frequency setup.
Let us denote by $(X,Y)$ a random couple built on a probability space $\left(\Omega,\mathcal{F},\P_{\textcolor{black}{(X,Y)}}\right)$. The feature $X=(X_t)_{t\in[0,1]}$ is a real-valued diffusion process whose \textcolor{black}{drift coefficient depends} on its associated label $Y$ taking values in $\cY=\{1,\cdots,K\}$, with $K\geq 2$.
More precisely, for each $i\in\cY$, $X$ is a solution of a stochastic differential equation whose drift function, denoted by $b^*_i$, depends on the class $i$. The marginal distribution of $X$ is hence a mixture of distributions of time-homogeneous diffusion processes.
We assume that a learning sample $\mathcal{D}_N = \{((X^{i}_t)_{t\in [0,1]},Y_i), i= 1, \ldots, N\}$ is provided, composed of $N$ \emph{i.i.d.} random \textcolor{black}{couples} with distribution $\P_{(X,Y)}$. Additionally, in this paper, the diffusions $X_i$ are observed on a subdivision $\{0, 1/n, \cdots , 1\}$ of the time interval $[0,1]$, for a positive integer $n$. 
Since we deal with multiclass classification setting, the statistical goal is then to build, based on $\mathcal{D}_N$, a classifier $\w{g}$, 
such that $\w{g}(X)$ is a prediction of the associated label $Y$ of a new path $X$.
Besides, we expect that the empirical classifier mimics the optimal Bayes classifier $g^*$
characterized as
\begin{equation*}
\textcolor{black}{g^*(X) \in \argmin{g} \P_{(X,Y)}\left(g(X) \neq Y\right)}.
\end{equation*}
Specifically, we propose a classification procedure based on the plug-in principle. In particular, the construction of our empirical classifier relies on estimators of both drift and diffusion coefficients. The performance of a predictor $\w{g}$ is assessed through its excess risk $\textcolor{black}{\P(\w{g}(X)\neq Y)- \P\left(g^*(X) \neq Y\right)}$. In the finite dimensional classification setup (\emph{e.g.} $X \in \mathbb{R}^d$), rates of convergence for plug-in rules are usually obtained under the strong density assumption ($X$ admits a density which is lower bounded) as in~\citet{audibert2007fast, Gadat_Klein_Marteau16}. 
However, theoretical properties of plug-in rules in supervised classification of trajectories
are much less studied.

\paragraph*{Related works.}
Up to our knowledge, the work of~\cite{cadre2013supervised} is the first one that tackles the problem of supervised classification in the stochastic differential equation framework. More precisely, the authors consider the model where $X=(X_{t})_{t\in[0,1]}$ is a mixture of two diffusion processes and provide a classifier based on the empirical risk minimization strategy for which
they establish rates of convergence. However, the proposed method is not implementable since it involves the minimization of a non-convex criterion.
More recently, ~\cite{gadat2020optimal}, and~\cite{denis2020classif} study plug-in classifiers for classification of diffusion paths.
In~\cite{gadat2020optimal} the authors propose a plug-in rule
for the binary classification problem where the trajectories are generated by Gaussian processes, solutions of the white noise model.
In this model, the drift function depends on time and on the label $Y$, also, the diffusion coefficient is supposed to be constant and known.
Within this framework, \cite{gadat2020optimal} establish the optimality of their classification procedure which reaches the minimax rate of convergence of order $N^{-s/(2s+1)}$, where the drift function is assumed to belong to a Sobolev space of regularity $s\geq 1$. Under an additional margin type assumption, they also derive faster rates of convergence. 
Closest to our framework, \cite{denis2020classif} also consider the challenging multiclass problem where the drift functions are space-dependent. However, the authors consider drift functions modeled under parametric assumptions, keeping the diffusion coefficient known and constant.
They propose a plug-in classifier for which only consistency is established.

In the present work, we consider a plug-in classifier that relies on nonparametric estimators of the drift and diffusion coefficients. 
The literature on this topic is extensive. Usually, the construction of estimators of drift and diffusion functions relies on the observation
of a single path. For instance, \cite{hoffmann1999lp} stud\textcolor{black}{ies} minimax rate of convergence for the estimation of the diffusion coefficient on a compact interval. For the inference of the drift coefficient, the main references using penalized contrasts can be found for long time observation with high frequency data in~\cite{hoffmann1999adaptive,comte2007penalized, comte2021drift}.
However, since we deal with the multiclass classification framework, the construction of estimators of both drift and diffusion coefficients is based on the learning sample $\mathcal{D}_N$ which is composed of repeated observations of the process on the fixed time-interval $[0,1]$. Recently,~\cite{comte2020nonparametric,marie2021nadaraya, della2022nonparametric} consider nonparametric procedures for the estimation of the drift function for continuous observations in the context of \emph{i.i.d.} observations when the horizon time is fixed. Furthermore, towards high-frequency data,~\cite{denis2020ridge} study minimum contrast estimator under a $l_2$ constraint.

\paragraph*{Main contributions. }

In this paper, we extend the results of~\cite{denis2020classif} and~\cite{gadat2020optimal} in several directions. In particular, one of the major contribution is to provide, up to our knowledge, the first study of rates of convergence for plug-in classifier in the mixture model of time-homogeneous diffusion. Importantly, we highlight that extending the results of~\cite{gadat2020optimal} to diffusion models in which the drift functions are space-dependent and the diffusion coefficient is either unknown or non-constant add many difficulties.
Besides, contrary to~\cite{denis2020classif}, we consider the nonparametric mixture model where both drift {\it and} diffusion functions are unknown as well as the weights of the mixture.
Specifically, we build a plug-in classifier that relies on the Girsanov's theorem and involves nonparametric estimators of the drift functions $b_i^*, i \in \cY$, and the diffusion coefficient.
The construction of our estimators is inspired of the ridge estimators provided in~\cite{denis2020ridge}, and consists in the minimization of a least-squares type contrast over a finite dimensional subspace under a $l_2$-constraint. The considered space of approximation is then spanned by the $B$-spline basis~\cite{deboor78}.

One of the main difficulty of the study of statistical properties of the plug-in classifiers in our context is that
it requires deriving rates of convergence for the drift and diffusion coefficients on a non-compact interval. It hence implies that the strong density assumption does not hold, although, we consider assumptions that ensure existence of transition density. Notably, our results embed generalization of the results provided in~\cite{denis2020ridge} for the estimation of non-compactly supported drift functions for $B$-spline based estimators, but also exhibit the first result for the estimation of the diffusion coefficient in the \emph{i.i.d.} framework. A salient point of our theoretical findings is obtained when the diffusion coefficient is constant and known. In this case, by leveraging the results of~\cite{comte2020nonparametric}, we show that 
optimal rates for drift estimation can only be achieved on intervals included in $[-C\sqrt{\log(N)}, C\sqrt{\log(N)}]$, with $C > 0$.

To sum up our results, a first part is dedicated to the consistency of our plug-in classifier which is obtained under very mild assumptions. In a second part, convergence rates are established in three particular cases. 
\begin{enumerate}[label=(\roman*)]
    \item When the drift functions are bounded and Lipschitz, and the diffusion coefficient is unknown and possibly non-constant,
we obtain a rate of convergence of order $N^{-1/5}$ for the plug-in classifier (up to a factor of order $\exp(\sqrt{c\log(N)}), \ c>0$). 
    \item When the diffusion coefficient is known and constant, and when the drift functions are bounded and belongs to some H\"older space with regularity $\beta$, using some arguments developed in \cite{comte2020regression} and \cite{comte2021drift} for the estimation of non-compactly supported drift functions, together with approximations of the transition density of $X$ (as they are intractable), we then prove that the plug-in classifier reaches rate of order $N^{-\beta/(2\beta+1)}$ (up to a factor of order $\exp(\sqrt{c\log(N)}), \ c>0$).
    \item When the drifts are unbounded but re-entrant and H\"older continuous with regularity $\beta$, we obtain a rate of convergence of order $N^{-3\beta/(4(2\beta+1))}$. Notice that when $\beta=1$ and $d=1$, it corresponds to the rate found in \cite{Gadat_Klein_Marteau16}. 
\end{enumerate}

\textcolor{black}{The proposed plug-in classifier is implemented in the \texttt{R}-package \texttt{SDEclassif}  available on \href{https://github.com/Eddymichelella/SDEclassif.git}{github}}.

\paragraph*{Outline of the paper. }
Section~\ref{sec:statSetting} is dedicated to presentation of the mathematical framework for the classification task. Then, the construction of the plug-in classifier is described in Section~\ref{sec:classProc} and its consistency is established in Section~\ref{subsec:consistency}.
In Sections\ref{subsec:consistency}~and~\ref{sec:ratesKnownSigma} we provide rates of convergence of our plug-in procedure under different assumptions. We perform a numerical experiment that supports our theoretical results in Section~\ref{sec:NumStudy}. Finally, We provide a discussion in Section~\ref{sec:conclusion} and the proofs of our results are postponed to Section~\ref{sec:proofs}.

%%%%%%%%%%%%%%%%%%%%%%%%%%%%%%%%%%%%%%%%%%%%%%%%%%%%%%%%%
%%%%%%%%%%%%%%%%%%%%%%%%%%%%%%%%%%%%%%%%%%%%%%%%%%%%%%%%%
\section{Statistical setting}
\label{sec:statSetting}
%%%%%%%%%%%%%%%%%%%%%%%%%%%%%%%%%%%%%%%%%%%%%%%%%%%%%%%%%
%%%%%%%%%%%%%%%%%%%%%%%%%%%%%%%%%%%%%%%%%%%%%%%%%%%%%%%%%

We consider the multiclass classification problem, where the feature $X$ comes from a mixture of Brownian diffusions with drift. More precisely, the generic data-structure
is a couple $(X,Y)$ where the label $Y$ takes its values in the set $\cY := \{1,\ldots,K\}$ with distribution denoted by ${\bf \mfp}^*=\left(\mfp_1^*,\cdots,\mfp_K^*\right)$, and where the process $X = (X_t)_{t \in [0,1]}$ is defined as the solution of the following stochastic differential equation 
%(S.D.E.)
\begin{equation}
    \label{eq:model}
    dX_t=b^{*}_{Y}(X_t)dt+\sigma^{*}(X_t)dW_t, \; \; X_0=0,
\end{equation} 
where $\left(W_{t}\right)_{t\geq 0}$ is a standard Brownian motion independent of $Y$. In the following, we denote by ${\bf b}^* =(b_1^*, \ldots, b_K^*)$ the vector of drift functions. 
The real-valued functions $b^*_i(.)$, $i \in \cY$, and the diffusion coefficient $\sigma^*(.)$ are assumed to be unknown. We also assume that $0 < \mfp_0^* = \min_{i \in \cY} \mfp_i^*$.

%Note that the assumption $x_0 = 0$ is not limiting. \sout{Indeed, for $x_0 \neq 0$, it is sufficient to consider the process $X^{\prime}=\left(X_t-x_0\right)_{t\in[0,1]}$ solution of the following equation,
%\begin{equation*}
%dX^{\prime}_{t}=b^*_{Y}\left(X^{\prime}_{t}+x_0\right)dt+\sigma^{*}\left(X^{\prime}_{t}+x_0\right)dW_t
%\end{equation*}
%}

In this framework, the objective is to build a classifier $g$, \textit{i.e.} a measurable function such that the value $g(X)$ is a prediction of the associated label $Y$ of $X$. The accuracy of such classifier $g$ is then assessed through its misclassification risk, denoted by
\begin{equation*}
\mathcal{R}(g) := \mathbb{P}_{\textcolor{black}{(X,Y)}}\left(g(X) \neq Y\right).
\end{equation*}
In the following, the set of all classifiers is denoted by $\mathcal{G}$.

The main assumptions considered throughout the paper are presented in Section~\ref{subsec:ass}.
The definition and characterization of the optimal classifier {\it w.r.t.} the misclassification risk, namely the {\it Bayes classifier}, is provided in Section~\ref{subsec:BayesClassifier}

%%%%%%%%%%%%%%%%%%%%%%%%%%%%%%%%%%%%%%%%%%%%%%%%%%%%%%
\subsection{Assumptions}
\label{subsec:ass}
%%%%%%%%%%%%%%%%%%%%%%%%%%%%%%%%%%%%%%%%%%%%%%%%%%%%%%%

The following assumptions ensure that Equation~\eqref{eq:model} admits a unique strong solution \citep[see][Theorem~2.9]{karatzas2014brownian}, and that the diffusion process $X$ admits a transition density 
$$\textcolor{black}{p_{X} : (t,x) \in ([0,1]\times \R)\mapsto p_{X}(t,x)}$$ 
\citep[see for example][]{gobet2002lan}. 
\begin{assumption}(Ellipticity and regularity)
\label{ass:RegEll}
\begin{enumerate}[label=(\roman*)]
\item There exists $L_0>0$ such that the functions $b_i^*, i = 1, \ldots, K$ and $\sigma^*$ are $L_0$-Lipschitz:
\begin{align*}
        \underset{i\in\mathcal{Y}}{\sup}{\left|b^{*}_{i}(x)-b^{*}_{i}(y)\right|}+\left|\sigma^{*}(x)-\sigma^{*}(y)\right|\leq L_0|x-y|, \ \forall (x,y)\in\mathbb{R}^{2}.
    \end{align*}
\item There exist real constants $\sigma^{*}_{0},\sigma^{*}_{1}$ such that
\begin{equation*}
0<\sigma^{*}_{0}\leq\sigma^{*}(x)\leq\sigma^{*}_{1}, \ \ \forall x\in\mathbb{R}.
\end{equation*}
\item $\sigma^{*}\in\mathcal{C}^{2}\left(\mathbb{R}\right)$ and there exist $\gamma\geq 0$ \textcolor{black}{and $c > 0$} such that : $\left|\sigma^{*\prime}(x)\right|+\left|\sigma^{*\prime\prime}(x)\right|\leq \textcolor{black}{c}\left(1+|x|^{\gamma}\right), \ \ \forall x\in\mathbb{R}$.
\end{enumerate}
\end{assumption} 
Assumption~\ref{ass:RegEll} insures that for any integer $q \geq 1$, there exists $C_q > 0$ such that
$$
\E\left[\underset{t\in[0,1]}{\sup}{\left|X_t\right|^q}\right] \leq C_q.
$$
We also assume that the following Novikov's criterion is fulfilled \cite[Prop. (1.15) p. 308]{revuzyor1999} .
\begin{assumption}(Novikov's condition)
\label{ass:Novikov}
For all $i \in \mathcal{Y}$, we have
\begin{align*}
    \E\left[\exp\left(\frac{1}{2}\int_{0}^{1}{\frac{b^{*2}_{i}}{\sigma^{*2}}(X_s)ds}\right)\right]<+\infty.
\end{align*}
\end{assumption}
In particular, this assumption allows to apply Girsanov's theorem that is a key ingredient to derive a characterization of the Bayes classifier in the next section.

%%%%%%%%%%%%%%%%%%%%%%%%%%%%%%%%%%%%%%%%%%%%%%%%%%%%%
\subsection{Bayes Classifier}
\label{subsec:BayesClassifier}
%%%%%%%%%%%%%%%%%%%%%%%%%%%%%%%%%%%%%%%%%%%%%%%%%%%%%

The Bayes classifier $g^*$ is  a minimizer of the misclassification risk over $\mathcal{G}$
\begin{equation*}
g^* \in \argmin{g \in \mathcal{G}} \mathcal{R}(g),
\end{equation*}
and is expressed as
\begin{equation*}
g^*(X) \in \argmax{i \in \mathcal{Y}} \pi^{*}_i(X), \;\; {\rm with} \;\; \pi_i^*(X) := \P\left(Y = i|X\right).
\end{equation*}
The following result of \cite{denis2020classif} provides a closed form of the conditional probabilities $\pi^{*}_i$, $i \in \mathcal{Y}$.
%It relies on the application of the Girsanov's theorem.
\begin{prop}\citep{denis2020classif}
\label{prop:probCond}
Under Assumptions \ref{ass:RegEll}, \ref{ass:Novikov},
for all $i \in \mathcal{Y}$, we define
\begin{equation*}
F^*_i(X):=\int_{0}^{1}{\frac{b_i^{*}}{\sigma^{*2}}(X_s)dX_s}-\frac{1}{2}\int_{0}^{1}{\frac{b^{*2}_{i}}{\sigma^{*2}}(X_s)ds}.
\end{equation*}
Under Assumptions~\ref{ass:RegEll},\ref{ass:Novikov}, for each $i \in \mathcal{Y}$, the conditional probability $\pi^{*}_i$ is given as follows:
\begin{equation*}
    \pi^*_i(X)= \textcolor{black}{\phi^*_i}\left(\mathbf{F}^*(X)\right), 
\end{equation*}
where ${\bf F}^* = \left(F_1^*, \ldots, F_k^*\right)$, and $\phi_i^* : (x_1,\cdots,x_K)\mapsto\frac{\mfp^*_i\mathrm{e}^{x_i}}{\sum_{k=1}^{K}{\mfp^*_k\mathrm{e}^{x_k}}}$ are the softmax functions.  
\end{prop}
The above proposition provides an explicit dependency of the Bayes classifier on the unknown parameters $\mathbf{b}^*$,
$\sigma^*$, and $\mathbf{\mfp}^*$. Hence, it  naturally suggests to build {\it plug-in} type estimators $\w{g}$ of the Bayes classifier $g^*$, relying on estimators of the unknown parameters. In this way, we aim at building an empirical classifier whose misclassification risk is closed to the minimum risk which is reached by the Bayes classifier.
The following section is devoted to the presentation of the classification procedure.

%%%%%%%%%%%%%%%%%%%%%%%%%%%%%%%%%%%%%%%%%%%%%%%%%%%%%%
%%%%%%%%%%%%%%%%%%%%%%%%%%%%%%%%%%%%%%%%%%%%%%%%%%%%%%
\section{Classification procedure: a plug-in approach}
\label{sec:classProc}
%%%%%%%%%%%%%%%%%%%%%%%%%%%%%%%%%%%%%%%%%%%%%%%%%%%%%%
%%%%%%%%%%%%%%%%%%%%%%%%%%%%%%%%%%%%%%%%%%%%%%%%%%%%%%

Let $n \geq 1$ be an integer, and $\Delta_n = 1/n$ the time step which defines the regular grid of the observation time interval $[0,1]$. Let us assume now that an observation is a couple $(\bar{X},Y)$, with $\bar{X} := (X_{k \Delta_n})_{0\leq k \leq n}$ a high frequency sample path coming from $(X_t)_{t \in [0,1]}$ a solution of Equation~\eqref{eq:model}, and $Y$ its associated label. We also introduce, for $N\geq 1$, a learning dataset $\mathcal{D}_N = \{(\bar{X}^j, Y_j),\ j\in \{1,\dots N\}\}$ which \textcolor{black}{consists} of $N$ independent copies of $(\bar{X},Y)$.
The asymptotic framework is such that $N$ and $n$ tend to infinity.

Based on $\mathcal{D}_N$ we build a classification procedure that relies on the result of Proposition~\ref{prop:probCond}. Our classifier uses the knowledge of the class $Y_j$ for the path $X^j$, placing our work in the frame of supervised learning.
The procedure is formally described in Section~\ref{subsec:proc_def} and Section~\ref{subsec:estimators} while its statistical properties are provided in Section~\ref{subsec:consistency}. 

%Both parameters $n,N$ will be considered to go to infinity, whereas the horizon time of observation is fixed).

%%%%%%%%%%%%%%%%%%%%%%%%%%%%%%%%%%%%%%%%%%%%%%%%%%%%%%%%%%%
\subsection{Classifier and excess risk}
\label{subsec:proc_def}
%%%%%%%%%%%%%%%%%%%%%%%%%%%%%%%%%%%%%%%%%%%%%%%%%%%%%%%%%%%

As suggested by Proposition~\ref{prop:probCond}, 
based on $\mathcal{D}_N$, we first build estimators $\boldsymbol{{\w{b}}} = (\w{b}_1, \ldots, \w{b}_K)$, and $\textcolor{black}{\w{\sigma}^{2}}$ of ${\bf b}^*$ and $\textcolor{black}{\sigma^{*2}}$ respectively. Besides, we consider the empirical estimators of $\mfp_i^*$, $i = 1, \ldots, K$:
\begin{equation}\label{def:hatpi}
\w{\mfp}_i = \dfrac{1}{N} \sum_{j=1}^N \one_{\{Y_j=i\}}.
\end{equation} 
Then, in a second step, we introduce the discretized estimator of ${\bf F}^*$ 
\begin{equation}\label{def:Fhat}
\boldsymbol{\w{F}} = (\w{F}_1, \ldots, \w{F}_K), \;\; {\rm with} \;\; \w{F}_{i}(X)=\sum_{k=0}^{n-1}{\left(\frac{\w{b}_{i}}{\w{\sigma}^{2}}(X_{k\Delta})\left(X_{(k+1)\Delta}-X_{k\Delta}\right)-\frac{\Delta}{2}\frac{\w{b}^{2}_{i}}{\w{\sigma}^{2}}(X_{k\Delta})\right)}.
\end{equation}
Finally,  considering the functions $\w{\phi}_i  : (x_1,\cdots,x_K)\mapsto\frac{\w{\mfp}_i\mathrm{e}^{x_i}}{\sum_{k=1}^{K}{\w{\mfp}_k\mathrm{e}^{x_k}}}$, we naturally define the resulting plug-in classifier
$\w{g}$ as 
\begin{equation}\label{def:hatg}
\w{g}(X) \in \argmax{i \in \cY} \w{\pi}_i(X), \;\; {\rm with} \;\; \w{\pi}_i(X) = \w{\phi}_i(\boldsymbol{\w{F}}(X)).
\end{equation}
Hereafter, we establish that the consistency of the plug-in classifier $\w{g}$ can be obtained through an empirical distance between estimators $\boldsymbol{\w{b}}$, and $\textcolor{black}{\w{\sigma}^{2}}$ and the true functions ${\bf b}^*$, and $\textcolor{black}{\sigma^{*2}}$ respectively. This distance relies on the  empirical norm defined for  $h : \mathbb{R} \rightarrow \mathbb{R}$ as
\begin{equation*}\label{eq:normei}
\|h\|_{n,i}^2 := \E_{X|Y=i}\left[\dfrac{1}{n} \sum_{k=0}^{n-1} h^2(X_{k\Delta})\right].
\end{equation*}
We also introduce the general empirical norm $\|.\|_{n}$ which, for any function $h$, is 
\begin{equation*}\label{eq:normn}
%\displaystyle
\|h\|_{n}^2 := \E_{X}\left[\dfrac{1}{n} \sum_{k=0}^{n-1} h^2(X_{k\Delta})\right].
\end{equation*}
\textcolor{black}{Let us begin with a result which provides a closed formula of the excess risk in multiclass classification.
\begin{prop}
\label{prop:excessRiskClass}
Let $g$ be a classifier. The following holds
\begin{equation*}
\mathcal{R}(g) - \mathcal{R}(g^*) =
\E\left[\sum_{i=1}^K\sum_{j\neq i}
\left|\pi_i^*(X)-\pi_j^*(X)\right|\one_{\{g(X) = j, g^*(X) = i\}}\right].
\end{equation*}
\end{prop}
The proof of this result is omitted and can be found for instance in~\cite{denis2020classif}. From the result of Proposition~\ref{prop:excessRiskClass}, and upper-bounding the indicator function by $1$, we take advantage of the Lipschitz property of the softmax functions $(\phi^{*}_{i})_{i=1,\ldots, K}$ that define the probabilities $(\pi_i^{*}(X))_{i=1, \ldots, K}$ to bound the excess risk of an empirical classifier $\w{g}$ based on $\w{\mathbf{b}} = \left(\w{b}_1, \ldots, \w{b}_K\right)$ and $\w{\sigma}^{2}$ by the respective risks of estimation of estimators $\w{b}_i$ and $\w{\sigma}^{2}$.
}
Let us now announce the main result on the excess risk of a plug-in type classifier. 

\begin{theo}
\label{thm:comparisonInequality}
Assume $N$ and $n$ fixed (and large). Grant Assumptions \ref{ass:RegEll}, \ref{ass:Novikov}.
Assume that there exists ${b}_{\rm max}, {\sigma}^2_0 > 0$ such that for all $x \in \mathbb{R}$
\begin{equation}\label{hyp:borne-bchapeau}
\max_{i \in \mathcal{Y}} |\w{b}_i(x)| \leq {b}_{\rm max} \;\; {\rm and} \;\; \w{\sigma}^{2}(x) \geq {\sigma}^2_0. 
\end{equation}
Then the classifier $\w{g}$ defined in Equation \eqref{def:hatg} satisfies
%following holds
\begin{equation*}
\E\left[\mathcal{R}(\w{g}) - \mathcal{R}(g^*)\right] \leq C \left(\sqrt{\Delta_n} + \frac{1}{\mfp_0^* \sqrt{N}} 
+ \E\left[{b}_{\rm max}{\sigma}^{-2}_0  \sum_{i =1}^K \|\w{b}_i -b_i^*\|_n\right] + \E\left[{\sigma}^{-2}_0 \|\w{\sigma}^2-\sigma^{*2}\|_n\right]\right),
\end{equation*}
where $C >0$ is a constant which depends on $b^*$,  $\sigma^*$, and $K$.
\end{theo}

Theorem~\ref{thm:comparisonInequality}
highlights that the excess risk of the plug-in classifier depends on the discretization error which is of order $\Delta_n^{-1/2}$, the $L_2$ error of $\w{\mfp}$ which is of order $N^{-1/2}$, and the estimation error of $\w{\bf b}$ and $\w{\sigma}^2$ assessed through the empirical norm $\|.\|_{n}$.
Therefore, a straightforward consequence of Theorem~\ref{thm:comparisonInequality}
is that consistent estimators of ${\bf b}^*$, and $\sigma^{*2}$ yield the consistency of plug-in classifier $\w{g}$.
Notice that the additional assumption \eqref{hyp:borne-bchapeau} does not require that the true functions $b_i^*$'s are bounded, only their estimators should be. For the difference between $b_i^*$ and $\widehat{b}_i$ to remain controlled in the norm $\|\cdot \|_n$, it is necessary that the process $X$ rests with high probability in a compact region of $\R$. 
%We will see in the next section that this region naturally depends on $n$ and $N$, and so will $b_{\max}$.
The next section is devoted to the construction of consistent estimators of both 
drift and diffusion \textcolor{black}{coefficients}.

%%%%%%%%%%%%%%%%%%%%%%%%%%%%%%%%%%%%%%%%%%%%%%%%%%%%%%%%%%%
\subsection{Estimators of drift and diffusion coefficients}
\label{subsec:estimators}
%%%%%%%%%%%%%%%%%%%%%%%%%%%%%%%%%%%%%%%%%%%%%%%%%%%%%%%%%%%

In this section, we provide consistent estimators $\boldsymbol{\w{b}}$, and $\w{\sigma}^{2}$, implying the consistency of the associated plug-in classifier.
These estimators are defined as minimum contrast estimators under an $l_2$-constraint 
on a finite dimensional vector space spanned by the $B$-spline basis, but other families of nonparametric estimators could have been chosen as well. 
In particular, to ensure statistical guarantees on $\mathbb{R}$, 
the considered estimators are built on a large intervals parameterized by the number $N$ of sample paths, and that tends to the whole real line as $N$ goes to infinity.

\subsubsection{Spaces of approximation}

Let $\textcolor{black}{A, K^{*} > 0}$, and $M \geq 1$.
Let \textcolor{black}{${\bf u} = (u_{-M}, \ldots, u_{K^{*}+M})$}, a sequence of knots of the compact interval \textcolor{black}{$[-A, A]$} such that 
\begin{align*}
    &u_{-M}=\cdots=u_{-1}=u_0= \textcolor{black}{-A}, \ \ \mathrm{and} \ \ \textcolor{black}{u_{K^{*}}=u_{K^{*}+1}=\cdots = u_{K^{*}+M} = A}. \\ 
    &\textcolor{black}{\forall \ell\in[\![0,K^{*}]\!], \ \ u_{\ell}= -A + \dfrac{2 \ell A}{K^{*}}.}
\end{align*}
Let us consider the $B$-spline basis $\left(B_{-M}, \ldots, B_{\textcolor{black}{K^{*}}+M}\right)$ of order $M$ defined by the knots sequence ${\bf u}$.
For the construction of the $B$-spline and its properties, we refer for instance to~\citep{gyorfi2006distribution}.
Let us mention that the considered $B$-spline functions are \textcolor{black}{nonnegative and} $M$-1 continuously differentiable on $\textcolor{black}{[-A,A]}$ and are zero outside \textcolor{black}{$[-A, A]$}.
Besides, for all \textcolor{black}{$x \in [-A,A]$}, we have that $\sum_{\ell=-M}^{\textcolor{black}{K^{*}-1}} B_{\ell}(x) = 1$.
Now, we introduce the space of approximation \textcolor{black}{$\mathcal{S}_{K^{*},M}$} defined as
\begin{equation}\label{eq:SKNM}
\textcolor{black}{\cS_{K^{*},M}: = \left\{\sum_{\ell=-M}^{K^{*} - 1} a_{\ell} B_{\ell}, \;\; \|{\bf a}\|^{2}_2 \leq (K^{*}+M)A^2\log(N) \right\}},
\end{equation} 
where \textcolor{black}{$\|{\bf a}\|^{2}_2=\sum_{\ell=-M}^{K^{*}-1} a_\ell^2$} is the usual $\ell_2$-norm. \textcolor{black}{Note that $A$ can depend on the size $N$ of the learning sample and tend to infinity as $N \rightarrow \infty$.}
The introduction of the constraint space \textcolor{black}{$\mathcal{S}_{K^{*},M}$} is motivated by two facts. The first one is the following important property of spline approximations, inspired by the related properties for the H\"older functions (see \cite{gyorfi2006distribution}):
\begin{prop}
\label{prop:approx}
Let $h$ be a $L$-lipschitz function. Then there exists \textcolor{black}{$\tilde{h} \in \mathcal{S}_{K^{*},M}$}, such that
\begin{equation*}
\textcolor{black}{|\tilde{h}(x)-h(x)| \leq C \frac{A}{K^{*}}, \;\; \forall x \in [-A,A]},
\end{equation*}
where $C >0$ depends on $L$, and $M$.
\end{prop}
The second one is that the set of functions \textcolor{black}{$\mathcal{S}_{K^{*},M}$} is a {\it totally bounded class}, in the following sense \cite[Chapter 28]{devroye2013probabilistic}. 
According to~\cite{denis2020ridge}, \textcolor{black}{for each $\varepsilon \in (0,1)$ and for $N$ large enough}, there exists an $\varepsilon$-net $\tilde{\mathcal{S}}_{\varepsilon}$ of \textcolor{black}{$\mathcal{S}_{K^{*},M}$} {\it w.r.t.} to the supremum norm $\|.\|_{\infty}$  such that
\begin{equation*}
\textcolor{black}{\log\left({\rm card}(\tilde{\mathcal{S}}_{\varepsilon})\right) \leq C_MK^{*}
\log\left(\frac{K^{*}}{\varepsilon}\right).}
\end{equation*}
It shows that the complexity of \textcolor{black}{$\mathcal{S}_{K^{*},M}$} given in Equation \eqref{eq:SKNM} is parametric which is particularly appealing in order to apply concentration inequalities.

\subsubsection{Minimum contrast estimators}

In this section, we propose two estimators of $\boldsymbol{b}^*$, and $\sigma^{*2}$ which lead to a plug-in classifier that exhibits appealing properties. 
The construction of the estimators $\boldsymbol{\w{b}}$, and $\w{\sigma}^{2}$ relies on the minimization of a least squares contrast function over the space \textcolor{black}{$\mathcal{S}_{K^{*},M}$}. They are both based on the observed increments of the process $X$.

\paragraph{Estimator of the drift functions. }
Let $i \in \mathcal{Y}$ and $N_i:= \sum_{j=1}^N \one_{\{Y_j = i\}}$ a random variable of Binomial distribution with parameters $(N,\mfp^*_i)$. We define the random set $\mathcal{I}_{i} := \{j, \; Y_j= i\} \textcolor{black}{= \{i_1, \ldots, i_{N_i}\}}$ and consider the dataset $\left\{\bar{X}^j, \; j \in \mathcal{I}_i\right\}$ \textcolor{black}{of size $N_i$} composed of the observations 
of the class $i$. \textcolor{black}{Herealter, we work conditional on $(\one_{\{Y_1=i\}}, \ldots, \one_{\{Y_N=i\}})$, on the event $\{N_i > 1\}$. Hence, $N_i$ is viewed as a deterministic variable such that $N_i > 1$ . In this context, we set for all $i \in \mathcal{Y}, ~ A = A_{N_i} > 0$ and $K^{*} = K_{N_i} > 0$ where $(A_{N_i})$ and $(K_{N_i})$ are increasing sequences of $N_i$}.
The first estimator $\widetilde{b}_i$ of $b_i^*$ is defined as
\begin{equation}
\label{eq:eqEstimator}
\tilde{b}_i \in \argmin{h \in \mathcal{S}_{\textcolor{black}{K_{N_i}},M}} \dfrac{1}{nN_i}\sum_{j \in \mathcal{I}_{i}}\sum_{k=0}^{n-1} \left(Z_{k\Delta_n}^j-h(\textcolor{black}{X}^j_{k\Delta_n})\right)^2\one_{N_i>0}, \;\; {\rm with} \;\; Z_{k\Delta_n}^j := \dfrac{(\textcolor{black}{X}^j_{(k+1)\Delta_n}-\textcolor{black}{X}^j_{k\Delta_n})}{\Delta_n}.
\end{equation} 
Then, to fit the assumption of Theorem~\ref{thm:comparisonInequality}, rather than $\tilde{b}_i$, we consider its thresholded counterpart 
\begin{equation}\label{eq:boundedbi}
\w{b}_i(x) := \tilde{b}_i(x) \one_{\{|\tilde{b}_i(x)|\leq \textcolor{black}{A_{N_i}\log^{1/2}(N)}\}} + {\rm sgn}(\tilde{b}_i(x))\textcolor{black}{A_{N_i}\log^{1/2}(N)}  \one_{\{|\tilde{b}_i(x)| > \textcolor{black}{A_{N_i}\log^{1/2}(N)}\}}.
\end{equation}
Note that the value of the threshold $\textcolor{black}{A_{N_i}\log^{1/2}(N)}$ corresponds to the bound $b_{\max}$ in \eqref{hyp:borne-bchapeau}. Although this bound depends on $N$, Theorem \ref{thm:comparisonInequality} can be applied, but to ensure the consistency of the classifier, we now have to prove that the estimation rate for $\widehat{b}_i$ decreases sufficiently fast.

\paragraph{Estimator of the diffusion coefficient.}

The construction of the estimator of \textcolor{black}{$\sigma^{*2}$} follows the same lines. However, since the diffusion coefficient is the same for all classes, we can use the whole dataset $\mathcal{D}_N$ to build its estimator \textcolor{black}{with $A = \tilde{A}_N, ~ K^{*} = \tilde{K}_N$ and $(\tilde{A}_N), ~ (\tilde{K}_N)$ are increasing sequences of $N$}. More precisely,
we define 
\begin{equation}\label{eq:least squares contrast - sigma}
\tilde{\sigma}^2  \in \argmin{h\in \mathcal{S}_{\textcolor{black}{\tilde{K}_N},M}} \dfrac{1}{nN} \sum_{j=1}^N\sum_{k=0}^{n-1} \left(U^j_{k\Delta_n}-h(\textcolor{black}{X}^j_{k\Delta_n})\right)^2, \;\; {\rm with} \;\; U^j_{k\Delta_n} = \dfrac{(\textcolor{black}{X}^j_{(k+1)\Delta_n}-\textcolor{black}{X}^j_{k\Delta_n})^2}{\Delta_n}
\end{equation}
Finally, as for the drift estimator we consider the truncated version $\w{\sigma}^2$ as 
\begin{equation}\label{eq:boundedsigma}
\w{\sigma}^2(x) := \tilde{\sigma}^2(x) \one_{\{\frac{1}{\log(N)} \leq \tilde{\sigma}^2(x)\leq \textcolor{black}{\tilde{A}_N\log^{1/2}(N)}\}} + \textcolor{black}{\tilde{A}_N\log^{1/2}(N)} \one_{\{\tilde{\sigma}^2(x) > \textcolor{black}{\tilde{A}_N\log^{1/2}(N)}\}} + \frac{1}{\log(N)} \one_{\{\tilde{\sigma}^2(x) \leq \frac{1}{\log(N)} \}}.
\end{equation}
Although this constraint does not appear in Theorem~\ref{thm:comparisonInequality}, it remains natural in view of 
Assumption~\ref{ass:RegEll} (ii). We will impose that $\w{\sigma}^{2}$ is bounded by $\textcolor{black}{\tilde{A}_N\log^{1/2}(N)}$ to derive its consistency.

%Indeed, the estimator of $\sigma^2$ need to not be too small because it appears in the denominator if Equation \eqref{def:Fhat}, and doing so, the assumption of Theorem \ref{thm:comparisonInequality} are satisfied. 

%%%%%%%%%%%%%%%%%%%%%%%%%%%%%%%%%%%%%%%%%%%%%%%%%%%%%%%%%%%
\subsection{A general consistency result}
\label{subsec:consistency}
%%%%%%%%%%%%%%%%%%%%%%%%%%%%%%%%%%%%%%%%%%%%%%%%%%%%%%%%%%%

In this section, we establish the consistency of the empirical classifier based on the estimators presented in the previous section.
We first provide rates of convergence for \textcolor{black}{the estimators of both the drift and the diffusion coefficients}. 
\begin{theo}
\label{thm:cveDriftSig} 
Let $i \in \mathcal{Y}$, \textcolor{black}{and set $A_{N_i} = \log^{2}(N_i)$ conditional on the event $\{N_i > 1\}$, and $\tilde{A}_N = \log^{2}(N)$}. Assume that  Assumptions \ref{ass:RegEll}, \ref{ass:Novikov} are satisfied.
Considering the estimator $\w{b}_i$ of $b^*_i$ \eqref{eq:boundedbi} and the estimator $\w{\sigma}^2$ of $\sigma^{*2}$ \eqref{eq:boundedsigma}, \textcolor{black}{set $K_{N_i} \propto (N_i\log(N_i))^{1/5}$ for $\w{b}_i$, and $\tilde{K}_{N} \propto (N\log(N))^{1/5}$ for $\w{\sigma}^2$}. \textcolor{black}{For $N$, $N_i$ then $n$ large enough, such that $\Delta_n  = O(1/N)$, we have}
\begin{equation*}
\E\left[\|\w{b}_i -b_i^*\|_{n,i}\right] \leq C_1 \left(\dfrac{\log^{4}(N)}{N}\right)^{1/5}, \;\; {\rm and} \;\; \E\left[\|\w{\sigma}^2 -\sigma^{*2}\|_{n}\right] \leq C_2 \left(\dfrac{\log^{4}(N)}{N}\right)^{1/5},
\end{equation*}
where $C_1, C_2 >0$ are constants which depend on $L_0$, $\mfp_0$, and $K$.
\end{theo} 
\textcolor{black}{Regarding the estimation of the drift functions $b^{*}_{i}$, 
the control of the integrated risk $\displaystyle \E \left[ \|\w{b}_i-b^*_i\|^2_{n,i} \right]$ is deduced from the control of the empirical risk $\displaystyle \E \left[ \|\w{b}_i-b^*_i\|^2_{n,N_i} \right]$, defined as}
\begin{equation*}
    \textcolor{black}{\E \left[ \|\w{b}_i-b^*_i\|^2_{n,N_i} \right]=\E \left[ \frac{1}{nN} \sum_{j=j_1}^{j_{N_i}}\sum_{k=0}^{n-1} (\w{b}_i-{b}^*_i)^2(X^{(j)}_{k\Delta_n})\right].}
\end{equation*}
%
%\textcolor{red}{using %Equation~\eqref{eq:eqEstimator}. 
%The properties of the constrained subspace $\mathcal{S}_{K_{N_i},M}$ together with that of functions of the spline basis (see proof of Theorem~\ref{thm:cveDriftSig} for more details). 
%\textcolor{red}{Then, the control of the following risk}
%\begin{equation*}
 %   \textcolor{red}{\E \left[ \|\w{b}_i-b^*_i\|^2_{n,i} \right]:=\E \left[ \sum_{k=0}^{n-1} (\w{b}_i-{b}^*_i)^2(X_{k\Delta_n}) \right]}
%\end{equation*}
%
\textcolor{black}{The link between the two risks is done using concentration arguments.
}
%is obtained from the previous risk $\E \left[ \|\w{b}_i-b^*_i\|^2_{n,N_i} \right]$ using, for all function $h$ such that $\|h\|_{\infty} = \mathrm{O}(L_N)$ with $(L_N)$ an increasing sequence of $N$ , the following result}
%\begin{equation*}
 %  \textcolor{red}{\E\left[\|h\|^{2}_{n}\right] - 2\E\left[\|h\|^{2}_{n,N}\right] = \mathrm{O}\left(\frac{K^{*}L_N\log(N)}{N}\right)}
%\end{equation*}
%
%\textcolor{red}{proved in \cite{denis2020ridge}, \textit{Lemma A.2}}.

Several comments can be made about Theorem \ref{thm:cveDriftSig}. First, we obtain a general rate of convergence for the estimation on $\R$ for both drift and diffusion coefficient functions under mild assumptions. This rate is, up to a logarithmic factor, of order $N^{-1/5}$. Hence, it extends the result of Theorem 3.3 in~\cite{denis2020ridge}, where only consistency of drift estimators is obtained. In particular, a difficulty in establishing  the convergence rate on $\R$ is to control the exit probabilities from \textcolor{black}{the intervals $(-A_{N_i},A_{N_i})$ and $(-\tilde{A}_N, \tilde{A}_N)$, which are} provided here by careful estimates for the transition densities following \cite{gobet2002lan}.

This result together with Theorem~\ref{thm:comparisonInequality} yields the consistency of the plug-in classifier 
\begin{equation}\label{eq:classifierPI}
\w{g}:=\w{g}_{\w{\mfp},\w{\bf b}, \w{\sigma}^2}
\end{equation} 
where the unknown parameters are replaced by their estimators in Equation~\eqref{def:hatg}.
However, application of Theorem~\ref{thm:comparisonInequality} requires the consistency of the estimator $\w{b}_i$ in terms of empirical norm $\|.\|_n$ and not in terms of norm $\|.\|_{n,i}$. To circumvent this issue, we can use a change of probability to get rid of the conditioning on $Y=i$. For this purpose, we take advantage of Lemma~\ref{lem:controleSortiCompact} and~\ref{lem:boundDensity} to derive precise control of the transition density of the process $X$ conditioned on $Y=i$, and then to establish the consistency of the plug-in classifier. 
\begin{theo}
\label{thm:consistency}
Grant Assumptions \ref{ass:RegEll}, \ref{ass:Novikov}. \textcolor{black}{For $N$ large enough, set $\Delta_n  = O(1/N), ~ \tilde{A}_N = \log(N)$ and $\tilde{K}_N = (N\log(N))^{1/5}$. Moreover, for each $i \in \mathcal{Y}$, on the event $\{N_i > 1\},~ A_{N_i} = \log(N_i)$ and $K_{N_i} \propto (N_i\log(N_i))^{1/5}$}. \textcolor{black}{Then}, the classifier $\w{g}$
%$= \w{g}_{\w{\mfp},\w{\bf b}, \w{\sigma}^2}$ given in Equation \eqref{def:hatg} 
satisfies 
\begin{equation*}
\E\left[\mathcal{R}(\w{g})-\mathcal{R}(g^*)\right] \underset{N \rightarrow \infty}\longrightarrow 0.
\end{equation*}
\end{theo}  
The consistency of our classification procedure is obtained under very mild assumptions. 
The study of the rates of convergence  requires more structural assumptions.
In the following section, we obtain rates of convergence of the plug-in classifier under different kind of assumptions. 

%%%%%%%%%%%%%%%%%%%%%%%%%%%%%%%%%%%%%%%%%%%%%%%%%%%%%%%%%%%%%%
%\section{Rates of convergence}
%\label{sec:ratesOfCve}
%%%%%%%%%%%%%%%%%%%%%%%%%%%%%%%%%%%%%%%%%%%%%%%%%%%%%%%%%%%%%%

%In this section, we study the rates of convergence of the proposed method described in Section~\ref{subsec:proc_def}. A general rate of convergence is first provided in Section~\ref{subsec:genRate} under the additional assumption that the drift functions of the considered mixture model are bounded, no additional assumptions being made on the diffusion coefficient. In Section~\ref{subsec:ratesKnownSigma}, we consider the case where the diffusion coefficient is known and assumed to be constant. In this case, the procedure achieves faster rates of convergence.

\subsection{General rate of convergence for bounded drift function}
\label{subsec:genRate}

\textcolor{black}{In this section, we study the general rate of convergence of the proposed method described in Section~\ref{subsec:proc_def} under the additional assumption that the drift functions of the considered mixture model are bounded. Note that no additional assumption is made on the diffusion coefficient.} 

Let us consider the following assumption.
\begin{assumption}
\label{ass:boundedDrift}
There exists $C_{\bf b^*}$ such that
\begin{equation*}
\max_{i \in \mathcal{Y}}\|b_i^*\|_{\infty} \leq C_{{\bf b}^*}.    
\end{equation*}
\end{assumption}
Let $i,j \in \cY^2$ with $i \neq j$. The following property allows to upper bound the expectation conditional on $\{Y=i\}$ by the expectation conditional on $\{Y = j\}$. This happens to be the cornerstone to derive rates of convergence for our procedure.
\begin{prop}
\label{prop:eqNorm}
Under Assumptions~\ref{ass:RegEll}, ~\ref{ass:Novikov}, and~\ref{ass:boundedDrift}, we have for all $i,j \in \cY^2$ such that $i\neq j$, and $N$ large enough
\begin{equation*}
\textcolor{black}{\left\|\w{b}_i - b^{*}_{i}\right\|^{2}_{n,j} \leq C\exp\left(\sqrt{c\log(N)}\right)\left\|\w{b}_i - b^{*}_{i}\right\|^{2}_{n,i}+C\frac{A^2\log(N)}{N}},
\end{equation*}
where $C, c >0$ depend on $C_{{\bf b}^*}, \sigma_1$, and $\sigma_0$.
\end{prop}

A crucial consequence of this result is that in particular the empirical norms $\left\|.\right\|_{n,i}$, $i \in \mathcal{Y},$ are now equivalent up to a factor of order $\exp\left(\sqrt{c\log(N)}\right)$. Notice that for all $r_1, r_2>0$,
\begin{equation}
\label{eq:ordre-exp-sqrt-log}
\log^{r_1}(N)=o\big(\exp\left(\sqrt{c\log(N)}\right)\big),\quad \mbox{ and }\quad \exp\left(\sqrt{c\log(N)}\right)=o\big(N^{r_2}\big). 
\end{equation}
In particular, the factor $\exp\left(\sqrt{c\log(N)}\right)$ is negligible with respect to any power of $N$. 
Therefore, combining Theorem~\ref{thm:comparisonInequality},~\ref{thm:cveDriftSig}, and Proposition~\ref{prop:eqNorm},
we are able to give the rate of convergence for our procedure (when the drift coefficients are globally Lipschitz and bounded).
\begin{theo}
\label{thm:boundedDrift}
Grant Assumptions~\ref{ass:RegEll}, \ref{ass:Novikov}, and~\ref{ass:boundedDrift}. \textcolor{black}{Set $\tilde{A}_N = \log(N)$ and $\tilde{K}_N \propto (N\log(N))^{1/5}$. Moreover, for each class $i \in \mathcal{Y}$, on the event $\{N_i > 1\}$, $A_{N_i} = \log(N_i)$ and $K_{N_i} \propto (N_i\log(N_i))^{1/5}$. The} plug-in classifier $\w{g}$ given in Equation \eqref{eq:classifierPI}, provided that $\Delta_n = O\left(N^{-1}\right)$ and $N$ large enough, satisfies
\begin{equation*}
\E\left[\mathcal{R}(\w{g})-\mathcal{R}(g^*)\right] \leq C \exp\left(\sqrt{c\log(N)}\right)N^{-1/5},
\end{equation*}
where $C >0$ depends on $C_{{\bf b}^*}, \sigma_1$, and $\sigma_0$.
\end{theo}
Leveraging the result of Theorem~\ref{thm:cveDriftSig} and Proposition~\ref{prop:eqNorm}, we obtain a rate of convergence which is of order $N^{-1/5}$ up to the extra factor  
$\exp\left(\sqrt{c\log(N)}\right)$. Note that the optimal rate of convergence obtained when the estimation of drift function is done over on a compact set is of order $N^{-1/3}$ {\it w.r.t.}  $\|.\|_{n}$ rather than $N^{-1/5}$~\citep[see][]{denis2020ridge}.
Here, this slower rate is mainly due to the fact that our procedure
requires a control of the drift estimators over $\R$. 

In the next section, we show that when $\sigma^*$ is constant and assumed to be known, we derive faster rates of convergence. In particular, under Assumption~\ref{ass:boundedDrift}, we show that our plug-in procedure achieves a rate of convergence of order $N^{-1/3}$. 
Lastly, note that Theorem~\ref{thm:boundedDrift} can be easily extended to higher order of regularity for the drift functions ({\it e.g.} H\"older with regularity $\beta >1$). In this case, the obtained rate of convergence is of order
$N^{-\beta/(2\beta +3)}$.

\section{Classifier's rate of convergence with known diffusion coefficient}
\label{sec:ratesKnownSigma}

In this section, we consider that the diffusion coefficient is known and constant, \textcolor{black}{and we derive faster rates of convergence of the classification procedure}. For sake of simplicity, we choose $\sigma^* = 1$. In this case, our plug-in procedure only involves
the estimation of the drift function $\boldsymbol{\w{b}}$. Hence, the plug-in classifier now writes as $\w{g} = \w{g}_{\mfp,\boldsymbol{\w{b}}, 1}$.

In order to derive a general rate of convergence as a function of the drift regularity, we consider the following smoothness assumption~\citep{tsybakov2008introduction}, which is a subset of Lipschitz functions.
%We also assume that
\begin{assumption}%(Smoothness assumption)
\label{ass:smoothnessDrift}
For all $i \in \mathcal{Y}$, $b^*_i$ is H\"older with regularity parameter $\beta \geq 1$.
\end{assumption}

\subsection{Rates of convergence for drift estimators}

Let $i \in \cY$.
The study of the rates of convergence of the estimator $\w{b}_i$ relies on the properties of
the matrix $\Psi_{K_{N_i}} \in \R^{(K_{N_i}+M)^2}$ defined by
\begin{equation}
\label{eq:psimatrix}
\Psi_{K_{N_i}}:=\left(\frac{1}{n}\sum_{k=0}^{n-1}{\E_{X|Y=i}\left[B_{\ell}(X^{i}_{k\Delta})B_{\ell^{\prime}}(X^{i}_{k\Delta})\right]}\right)_{\ell,\ell^{\prime}\in[-M,K_{N_i}-1]}.
\end{equation}
Note that for $t \in S_{K_{N_i}, M},~ t=\sum_{i=-M}^{K_{N_i}-1} a_i B_{i,M,{\bf u}}$, we have the relation
\begin{equation*}
\|t\|_{n,i}^2 = {\bf a}^{\prime} \Psi_{K_{N_i}} {\bf a}, \ \ \mathrm{with} \ \ {\bf a}=\left(a_{-M},\cdots,a_{K_{N_i}-1}\right)^\prime.
\end{equation*}
Let us remind the reader that for a matrix \textcolor{black}{$P$},
the operator norm \textcolor{black}{$\|P\|_{\mathrm{op}}$} is defined as the square root of the largest eigenvalue of the matrix \textcolor{black}{$P^\prime P$}. Besides, if \textcolor{black}{$P$} is symmetric, its norm is equal to its largest eigenvalue.
The matrix $\Psi_{K_{N_i}}$ satisfies the following property.
%coincided with the largest absolute value of the eigenvalues of $A$.
%
%
\begin{lemme}
\label{lm:MinEigenValue}
Conditional on $(\one_{\{Y_1=i\}}, \ldots, \one_{\{Y_N=i\}})$, on the event $\{N_i > 1\}$,
the matrix $\Psi_{K_{N_i}}$ given in Equation~\eqref{eq:psimatrix} satisfies
\begin{enumerate}[label=(\roman*)]
\item if $K_{N_i}\geq 1$,  $\Psi_{K_{N_i}}$ is invertible,
\item
under Assumption~\ref{ass:RegEll}, for $N$ large enough, if $K_{N_i} \leq \sqrt{N_i},$ there exists two constants $C,c>0$ such that
\begin{equation*}
 c\frac{K_{N_i}}{A_{N_i}}\exp\left(\frac{A^{2}_{N_i}}{6}\right) \leq    \|\Psi_{K_{N_i}}^{-1}\|_{\mathrm{op}} \leq C\frac{K_{N_i}\log(N_i)}{A_{N_i}}\exp\left(\frac{2}{3}A^{2}_{N_i}\right).
\end{equation*}
\end{enumerate}
\end{lemme}
A major consequence of Lemma~\ref{lm:MinEigenValue} is to give the order of $A_{N_i}$ {\it w.r.t.} $N_i$ to obtain optimal rates of convergence for the estimation of the drift function $b_i^*$. \textcolor{black}{Similar conditions are considered in \cite{comte2020nonparametric}.}

\textcolor{black}{For fixed $n$ and $N_i$ in $\mathbb{N}^{*}$, let us denote}
\begin{equation*}
    \textcolor{black}{\Omega_{n,N_i,K_{N_i}} := \underset{h \in \mathcal{S}_{K_{N_i},M}\setminus\{0\}}{\bigcap}{\left\{\left|\frac{\|h\|^{2}_{n,N_i}}{\|h\|^{2}_{n,i}} - 1\right| \leq \frac{1}{2}\right\}}}.
\end{equation*}
\textcolor{black}{The empirical norms $\|h\|_{n,N_i}$ and $\|h\|_{n,i}$ of any function $h \in \mathcal{S}_{K_{N_i},M}\setminus\{0\}$ are equivalent on the random set $\Omega_{n,N_i,K_{N_i}}$. More precisely, on $\Omega_{n,N_i,K_{N_i}}$, for all $h \in \mathcal{S}_{K_{N_i},M}\setminus\{0\}$, we have
$$ \frac{1}{2}\|h\|^{2}_{n,i} \leq \|h\|^{2}_{n,N_i} \leq \frac{3}{2}\|h\|^{2}_{n,i}. $$
}
\textcolor{black}{On $\Omega_{n,N_i,K_{N_i}}$ we are able to derive faster rate of convergence of the risk $\E\left[\|\w{b}_i - b^{*}_{A_{N_i},i}\|^{2}_{n,N_i}\right]$ while we control the probability $\P\left(\Omega^{c}_{n,N_i,K_{N_i}}\right)$. 
%of the estimator $\w{b}_i$ of the restriction $b^{*}_{A_{N_i},i}$ of the drift function $b^{*}_{i}$ on the interval $[-A_{N_i}, A_{N_i}]$, with $A_{N_i} > 0$ on the event $\{N_i > 1\}$. We have for $N$ large enough
More precisely, we have the bound}
\begin{align*}
    \textcolor{black}{\E\left[\|\w{b}_i - b^{*}_{A_{N_i},i}\|^{2}_{n,N_i}\right]} & \textcolor{black}{= \E\left[\|\w{b}_i - b^{*}_{A_{N_i},i}\|^{2}_{n,N_i}\one_{\Omega_{n,N_i,K_{N_i}}}\right] + \E\left[\|\w{b}_i - b^{*}_{A_{N_i},i}\|^{2}_{n,N_i}\one_{\Omega^{c}_{n,N_i,K_{N_i}}}\right]} \\
    & \textcolor{black}{\leq \E\left[\|\w{b}_i - b^{*}_{A_{N_i},i}\|^{2}_{n,N_i}\one_{\Omega_{n,N_i,K_{N_i}}}\right] + 4A^{2}_{N}\log(N)\P\left(\Omega^{c}_{n,N_i,K_{N_i}}\right),}
\end{align*}
\textcolor{black}{and the probability $\P\left(\Omega^{c}_{n,N_i,K_{N_i}}\right)$ satisfies}
\begin{equation}
    \label{eq:proba-omega-comp1}
    \textcolor{black}{\P\left(\Omega^{c}_{n,N_i,K_{N_i}}\right) \leq 2(K_{N_i}+M)\exp\left(-C \frac{N_i}{A_{N_i}\left\|\Psi^{-1}_{K_{N_i}}\right\|_{\mathrm{op}}}\right)}
\end{equation}
\textcolor{black}{(the proof of the lemma in Section~\ref{proof:lemmaprobacom}, follows the ideas of \cite{comte2020regression}).
%Equation~\eqref{eq:OmegaComp-bound2-bis} in Appendix). 
From Equation~\eqref{eq:proba-omega-comp1} and Lemma~\ref{lm:MinEigenValue}, we obtain
}
\begin{equation}
 \label{eq:proba-omega-comp2}
    \textcolor{black}{\P\left(\Omega^{c}_{n,N_i,K_{N_i}}\right) \leq 2(K_{N_i}+M)\exp\left(-C \frac{N_i}{K_{N_i}\log(N_i)}\exp\left(-\frac{2}{3}A^{2}_{N_i}\right)\right)}
\end{equation}
%
%
%Indeed, \cite{comte2021drift}
%show that the rates of convergence for $\w{b}_i$ and $\w{\sigma}$ in Theorem \ref{thm:cveDriftSig} can be established
%if the following constraint is satisfied
%depends on
%the interplay between parameters $K_{N_i}$ and $A_{N_i}$
%\begin{equation}
%\label{eq:eqConditionPsiMat}
%\|\Psi_{K_{N_i}}^{-1}\|_{\mathrm{op}} \leq C \dfrac{N_i}{\log^{2}(N_i)},
%\end{equation}
%\Cha{reprendre}
%with a constant $C>0$ depending on $\beta$.

Notably, conditional on $(\one_{\{Y_1=i\}}, \ldots, \one_{\{Y_N=i\}})$ \textcolor{black}{and on the event $\{N_i > 1\}$}, if $K_{N_i}$ is of order $N_{i}^{1/(2\beta+1)}$ (up to some extra logarithmic factors), and $A_{N_i}$ is chosen such that \textcolor{black}{the upper-bound of $\P\left(\Omega^{c}_{n,N_i,K_{N_i}}\right)$ is dominated by $K_{N_i}/N_i$ as $N$ tends to infinity}, then
the drift estimator converges as $N_i^{-2\beta/(2\beta+1)}$ {\it w.r.t.} $\|.\|^2_{n,i}$.
%Indeed, following~\cite{comte2021drift}, we can show that up to some extra logarithmic factors if $K_{N_i}$ is of order $N_{i}^{1/(2\beta+1)}$
%and $A_{N_i}$ is chosen such that
%\begin{equation}
%\label{eq:eqConditionPsiMat}
%\|\Psi_{K_{N_i}}^{-1}\|_{\mathrm{op}} \leq C \dfrac{N_i}{\log^{2}(N_i)},
%\end{equation}
%with $C>0$ a constant depending on $\beta$, then we can show that conditional on $(\one_{\{Y_1=i\}}, \ldots, \one_\{\{Y_N=i\})$
%the drift estimators achieves rates of convergence of order $N_i^{-2\beta/(2\beta+1)}$ {\it w.r.t.} $\|.\|^2_{n,i}$.
Interestingly, this is the same rate of convergence obtained in \citep{denis2020ridge} when the estimation of the drift function is performed over a fixed compact interval.
%which does not depend on $N$.
%\Cha{sens ?}
%\Chr{Remonter le lemme avant Equation~\eqref{eq:eqConditionPsiMat}}
%According to the Lemma, we obtain the value of $A_N$ for which this rate of convergence can be reached.
From this remark, 
%Lemma~\ref{lm:MinEigenValue} teaches us that 
if $K_{N_i}$ is of order $ \log^{-5/2}(N_i)N^{1/(2\beta+1)}_{i}$, %Equation~\eqref{eq:eqConditionPsiMat} is satisfied for 
\textcolor{black}{and} $A_{N_i} \leq \sqrt{\frac{3\beta}{2\beta+1} \log(\textcolor{black}{N_i})}$, \textcolor{black}{we deduce from Equation~\eqref{eq:proba-omega-comp2} that there exists a constant $C > 0$ such that} 
\begin{equation*}
    \textcolor{black}{\P\left(\Omega^{c}_{n,N_i,K_{N_i}}\right) \leq 2(K_{N_i}+M)\exp\left(-C\log^{3/2}(N_i)\right)}
\end{equation*}
\textcolor{black}{and the desired result is obtained since $N_i \rightarrow \infty$ a.s. as $N \rightarrow \infty$.} Furthermore, the lemma shows that the order of $A_{N_i}$ is tight. Indeed, for another choice of $A_{N_i}$ such that 
\begin{equation*}
\dfrac{A_{N_i}}{\sqrt{\log(N_i)}} \longrightarrow +\infty \;\; {\rm as} \;\; N \rightarrow +\infty,
\end{equation*}
\textcolor{black}{then, from Equation~\eqref{eq:proba-omega-comp2}, the upper-bound of $\P\left(\Omega^{c}_{n,N_i,K_{N_i}}\right)$ is of order $K_{N_i}$ since}
\begin{equation*}
    \textcolor{black}{\exp\left(-C \frac{N_i}{K_{N_i}\log(N_i)}\exp\left(-\frac{2}{3}A^{2}_{N_i}\right)\right) \longrightarrow 1 ~~ a.s. ~~ \mathrm{as} ~ N \rightarrow \infty}
\end{equation*}
\textcolor{black}{and the convergence of $\P\left(\Omega^{c}_{n,N_i,K_{N_i}}\right)$ to $0$ is no longer guaranteed.}

Based on this observation, the next result 
 establishes the rates of convergence for our proposed drift estimator on the event $\{N_i > 1\}$.

\begin{theo}
    \label{thm:LossErrorDrift}
    Let Assumptions~\ref{ass:RegEll}, ~\ref{ass:Novikov} and~\ref{ass:smoothnessDrift} be satisfied. Let $b^{*}_{A_{N_i},i} = b^*_i\one_{[-A_{N_i},A_{N_i}]}$ 
    defined on the event $\{N_i > 1\}$.
    If
     $A_{N_i} \leq \sqrt{\frac{3\beta}{2\beta+1}\log(N_i)}$, $K_{N_i} \propto \left(\log^{-5/2}(N_i)N^{1/(2\beta+1)}_{i}\right)$, and $\Delta_n = O\left(N^{-1}\right)$. Then for all $i \in  \mathcal{Y}$
    \begin{equation*}
     \E\left[\left\|\w{b}_i-b^{*}_{A_{N_i},i}\right\|^{2}_{n,i}\one_{N_i>1}\right]\leq C\log^{6\beta}(N)N^{-2\beta/(2\beta+1)},
    \end{equation*}
    where $C$ is a constant which depends on ${\bf b}^*$.
\end{theo}
The above result shows that for a proper choice of $A_{N_i}$ the drift estimators $\w{b}_i$ achieves, up to a logarithmic factor, the minimax rates of convergence {\it w.r.t.} $\|.\|_{n,i}$ (see Theorem~4.7 in~\citep{denis2020ridge}). Notably, Theorem~\ref{thm:LossErrorDrift} extends results obtained in~\citep{denis2020ridge} to the estimation of the drift function on an interval which depends on $N$.

In Section~\ref{subsubsec:ratesBoundedDriftOptimal} and Section~\ref{subsubsec:rateReentrentDrift}, we exploit this result to derive rates of convergence for the plug-in classifier $\w{g}$ defined as follows. 
On the event $\{\min_{i \in \cY} N_i > 1\}$, we consider the estimators $\w{\bf b}$ presented in Section~\ref{subsec:estimators}, and define the plug-in classifier $\w{g} = \w{g}_{\w{\mfp},\w{\bf b},1}$. On the complementary event $\{\min_{i \in \cY} N_i \leq 1\}$, we simply set 
$\w{g} = 1$.

%\begin{remark}
%Note that if we assume that the diffusion coefficient is unknown or non-constant, then we cannot prove that the resulting plug-in classifier reaches an optimal rate. In fact, since the transition density $p_{X,i}$ of the diffusion process $X$ given $Y=i$ for each $i\in\cY$ is not lower bounded on $\R$, or since we cannot approximate $p_{X,i}, \ i\in\cY$ such that for all $i,j\in\cY$ such $i\neq j$, the ratio $p_{X,i}/p_{X,j}$ is upper bounded on $\R$, it is impossible in this case to prove that the empirical norms $\|.\|_{n,i}$ and $\|.\|_{n,j}$ with $j\neq i$ are equivalent nor prove that the Gram matrix given in Equation~\eqref{eq:psimatrix}~ satisfies Equation~\eqref{eq:eqConditionPsiMat}~ necessary to establish the optimal rate of convergence of nonparametric estimators $\w{b}_i$ of drift functions $b^{*}_{i}$. Fortunately, the first condition is satisfied (up to a factor of order $\exp(\sqrt{c\log(N)})$) when $\sigma^{*}=1$ and the drift functions are bounded, and the second condition is satisfied assuming that $\sigma^{*}=1$ and imposing an additional assumption on the estimation intervals. 
%\end{remark}

\subsection{Rates of convergence: bounded drift functions}
\label{subsubsec:ratesBoundedDriftOptimal}

In this section, we assume that, additionally to $\sigma^*=1$, Assumption~\ref{ass:boundedDrift} is fulfilled (the drift function is bounded). 
Hence, we can use Proposition~\ref{prop:eqNorm}, and apply Theorem~\ref{thm:LossErrorDrift} to derive rates of convergence for plug-in estimator $\hat{g}$.

\begin{theo}
\label{thm:AlmostOptimalRate}
 Grant Assumptions~\ref{ass:RegEll}, ~\ref{ass:Novikov}, ~\ref{ass:boundedDrift}, ~\ref{ass:smoothnessDrift}.
Assume that for all $i\in\cY$, on the event $\{N_i>1\}$, \textcolor{black}{$A_{N_i}=\sqrt{\frac{6\beta}{2\beta+1}\log(N_i)}$} and $K_{N_i}\propto\left(\log^{-5/2}(N_i)N^{1/(2\beta+1)}_{i}\right)$, and $\Delta_n = O\left(N^{-1}\right)$.
Then the plug-in classifier $\w{g}= \w{g}_{\w{\bf b}, 1}$ satisfies
\begin{equation*}
\E\left[\mathcal{R}(\w{g})-\mathcal{R}(g^*)\right] \leq C\exp\left(\sqrt{c\log(N)}\right)N^{-\beta/(2\beta+1)}
\end{equation*}
where $C,c>0$ are constants depending on ${\bf b}^*$, $\beta, K$ and $\mfp_0$.
\end{theo}
The above theorem shows that the plug-in classifier $\w{g}$ achieves faster rates of convergence than in the case where $\sigma^*$ is unknown (see Theorem~\ref{thm:boundedDrift}). Notably,  the obtained rate is of the same order, up to a factor of order $\exp\left(\sqrt{c\log(N)}\right)$, than the
rates of convergence provided in~\cite{gadat2020optimal} in the framework of 
binary classification of functional data where the observation are assumed to come from  a white noise model. In their setting, $\sigma^* = 1$ and the drift functions depend only on the observation time interval, which is also assumed to be $[0,1]$. Therefore, our specific setup is more challenging since the drift functions are space-dependent, which involves to deal with estimation of function on a non-compact interval. Finally, it is worth noting that, up to $\exp\left(\sqrt{c\log(N)}\right)$ factor, the rate of convergence provided in Theorem~\ref{thm:AlmostOptimalRate} is the same as the minimax rates in the classical classification framework where the feature vector $X$ belongs to $\mathbb{R}$ and that $X$ admits a lower bounded density~\citep{Yang99, audibert2007fast}.

\subsection{Rates of convergence: when the drift functions are re-entrant}
\label{subsubsec:rateReentrentDrift}
 
In this section, we study performance of the plug-in classifier when the drift functions are not necessarily bounded. In this context, rates of convergence are obtained under the following assumption.
\begin{assumption}(re-entrant drift function)
    \label{ass:Re-entrant.Drift}
    For each label $i\in\cY$, there exists $c_0>4$ and $K_0\in\R$ such that
    $$\forall x\in\R, \ \ b^{*}_{i}(x)x\leq -c_0x^2+K_0.$$
\end{assumption}
An important consequence of this assumption is that there exists $C > 0$ (see Proposition~1.1 in~\citep{gobet2002lan}) such that
\begin{equation}
\label{eq:eqExpMoment}
\E\left[\exp(4|X_t|^2)\right] \leq C,    
\end{equation}
which yields a better bound on the tail probability $\P\left(|X_t| \geq A\right)$ for $A > 0$. 
It worth noting that under Assumption~\ref{ass:Re-entrant.Drift}, the drift functions are not bounded. Hence, we can not take advantage of Proposition~\ref{prop:eqNorm} to derive rates of convergence. Nonetheless, we obtain the following result.
\begin{theo}
    \label{thm:RateClass.Re-entrant.Drift}
    Grant Assumptions~\ref{ass:RegEll}, ~\ref{ass:Novikov}, ~\ref{ass:smoothnessDrift},~\ref{ass:Re-entrant.Drift}.
    Assume that for all $i\in\cY$, on the event $\{N_i>1\}$, $A_{N_i}=\sqrt{\frac{3\beta}{2\beta+1}\log(N_i)}$ and $K_{N_i}\propto\left(\log^{-5/2}(N_i)N^{1/(2\beta+1)}_{i}\right)$, and $\Delta_n = O\left(N^{-1}\right)$. Then, the plug-in classifier $\w{g}$ satisfies
    \begin{equation*}
\E\left[\mathcal{R}(\w{g})-\mathcal{R}(g^*)\right] \leq C\log^{3\beta+1}(N)N^{-3\beta/4(2\beta+1)}.
\end{equation*}
\end{theo}
The above theorem shows that the rate of convergence
of the plug-in classifier is, up to a logarithmic factor,
of order $N^{-3\beta/4(2\beta+1)}$. Therefore, this rate of convergence is slightly slower than the one provided in Theorem~\ref{thm:AlmostOptimalRate}.
It is mainly due to the fact that under Assumption~\ref{ass:Re-entrant.Drift}, Proposition~\ref{prop:eqNorm} does not apply and then, in view of considered assumptions in Theorem~\ref{thm:RateClass.Re-entrant.Drift},  we only manage to obtain the following bound, 
\begin{equation*}
    \forall i,j\in\cY: \ i\neq j, \ \textcolor{black}{\E\left[\left\|\w{b}_i - b^{*}_{i}\right\|^{2}_{n,j}\right] \leq C N^{\beta/ 4(2\beta+1)}\E\left[\left\|\w{b}_i - b^{*}_{i}\right\|^{2}_{n,i}\right]},
\end{equation*}
which is clearly worse than the one obtain\textcolor{black}{ed} in Proposition~\ref{prop:eqNorm}.
Interestingly, for $\beta = 1$, we can note that the rates obtained in Theorem~\ref{thm:RateClass.Re-entrant.Drift} are of the same order as the rates of convergence established in~\cite{Gadat_Klein_Marteau16} in the classification setup where the input vector lies in $\R$ under the assumption 
that $X$ does not fulfil the strong density assumption ({\it e.g.} the density of $X$ is not lower bounded).

%The rate obtained in Theorem~\ref{thm:RateClass.Re-entrant.Drift}~ is not optimal for various reasons. First, the transition density of the diffusion process is not lower bounded on $\R$, second,  according to Assumption~\ref{ass:Re-entrant.Drift}, the drift functions are not bounded, then we cannot apply Proposition~\ref{prop:eqNorm}~ for the change of probability and rather use the following relation
%\begin{equation*}
%    \forall i,j\in\cY: \ i\neq j, \ \E_{X|Y=i}[Z]\leq C\exp\left(\frac{2}{3}A^{2}_{N}\right) \E_{X|Y=j}[Z]
%\end{equation*}

%where the additional factor $\exp(2A^{2}_{N}/3)$ has a polynomial growth in $N$ since $A_N\propto\sqrt{\log(N)}$. Nonetheless, when $\beta=1$, this rate is of the same order (up to a log factor) than the rate $N^{-1/(3+d)}$ (for $d=1$) established in \cite{gadat2014classification} in the context of binary classification of $d-$dimensional data where the marginal density of $X\in\R^{d}$ is not lower bounded and we omit the margin assumption on the regression function ($\alpha=0$). Thus, we have a faster rate for all $\beta>1$.

%Next, we sum up in a short section some very convincing numerical results on the classifier.

%%%%%%
\section{Simulation study}
\label{sec:NumStudy}

This section is devoted to numerical experiments that support our theoretical findings. A first part is dedicated to the study of the performance of the plug-in classifier in a setting which meets the assumptions of Section~\ref{subsec:genRate}. The considered model is presented in Section~\ref{subsec:DiffModel}.
The implementation of the proposed procedure is discussed in Section~\ref{subsec:ImplClass} while the performances of the plug-in classifier are given in Section~\ref{subsec:results}.
Finally, several features of the problem are investigated in Section~\ref{subsec:influenceOfsigma}. In particular,
we consider the classical Ornstein-Uhlenbeck model, for which assumptions of Section~\ref{subsec:genRate} are not fulfilled.

\subsection{Models and simulation setting}
\label{subsec:DiffModel}

We fix $K = 3$ classes in the following. Note that, we do not consider larger value of $K$ since the evaluation of the impact of $K$ on the procedure is beyond the scope of this paper. To illustrate the accuracy of the presented plug-in classifier, 
we investigate the model described in Table~\ref{tab:DiffModels}.
\begin{table}[!ht]
\begin{center}
\renewcommand{\arraystretch}{1.25}
\begin{tabular}{c|c}
    \hline 
      $b^{*}_{1}(x)$ & $1/4+(3/4)\cos^2 x$ \\ 
    
     $b^{*}_{2}(x)$ & $\theta[1/4+(3/4)\cos^2 x]$ \\ 
    
     $b^{*}_{3}(x)$ & $-\theta[1/4+(3/4)\cos^2 x]$ \\ 
    
     $\sigma^{*}(x)$ & $0.1+0.9/\sqrt{1+x^2}$ \\
    \hline
\end{tabular}
\caption{{\small \textit{Drift and diffusion coefficients, depending on $\theta\in\Theta=\{1/2,3/4,(4+\alpha)/4, \alpha\in[\![1,12]\!]$.}}}\label{tab:DiffModels}
\end{center}
\end{table}
This toy model, described in Table~\ref{tab:DiffModels}, fulfills the assumptions of Section~\ref{subsec:genRate}. Interestingly, this model allows evaluating the influence of the distance between the drift functions of each of the three classes, on the classification problem, through the parameter $\theta$.
Indeed, 
$$\underset{i,j=1,2,3}{\min}{\|b^{*}_{i}-b^{*}_{j}\|_{\infty}}=\theta, \ \ \mathrm{where} \ \ \theta\in\Theta=\{1/2,3/4,(4+\alpha)/4, \alpha\in[\![1,12]\!]\}.$$
We investigate the consistency of the empirical classifier using learning samples of size $N\in\{100,1000\}$ with $n \in \{100, 500\}$ (and thus with $\Delta_n=1/n$).
We use the \texttt{R}-package \texttt{sde}  \citep[see][]{iacus2009simulation} to simulate the solution of the stochastic differential equation corresponding to the chosen model.

Figure~\ref{fig:path-model5} displays simulated trajectories from the proposed model. On the left panel (right panel respectively)  the observed learning sample comes from the model with parameter $\theta=1/2$ ($\theta = 4$ respectively) and each class is represented by one color. We can see from Figure~\ref{fig:path-model5} that the distance between the drift functions strongly impacts the dispersion of the trajectories and leads to a more difficult classification task. 
\begin{figure}[hbtp]
\begin{minipage}[t]{.45\textwidth}
\raggedright
  \centering
  \includegraphics[width=0.9\linewidth, height=0.2\textheight]{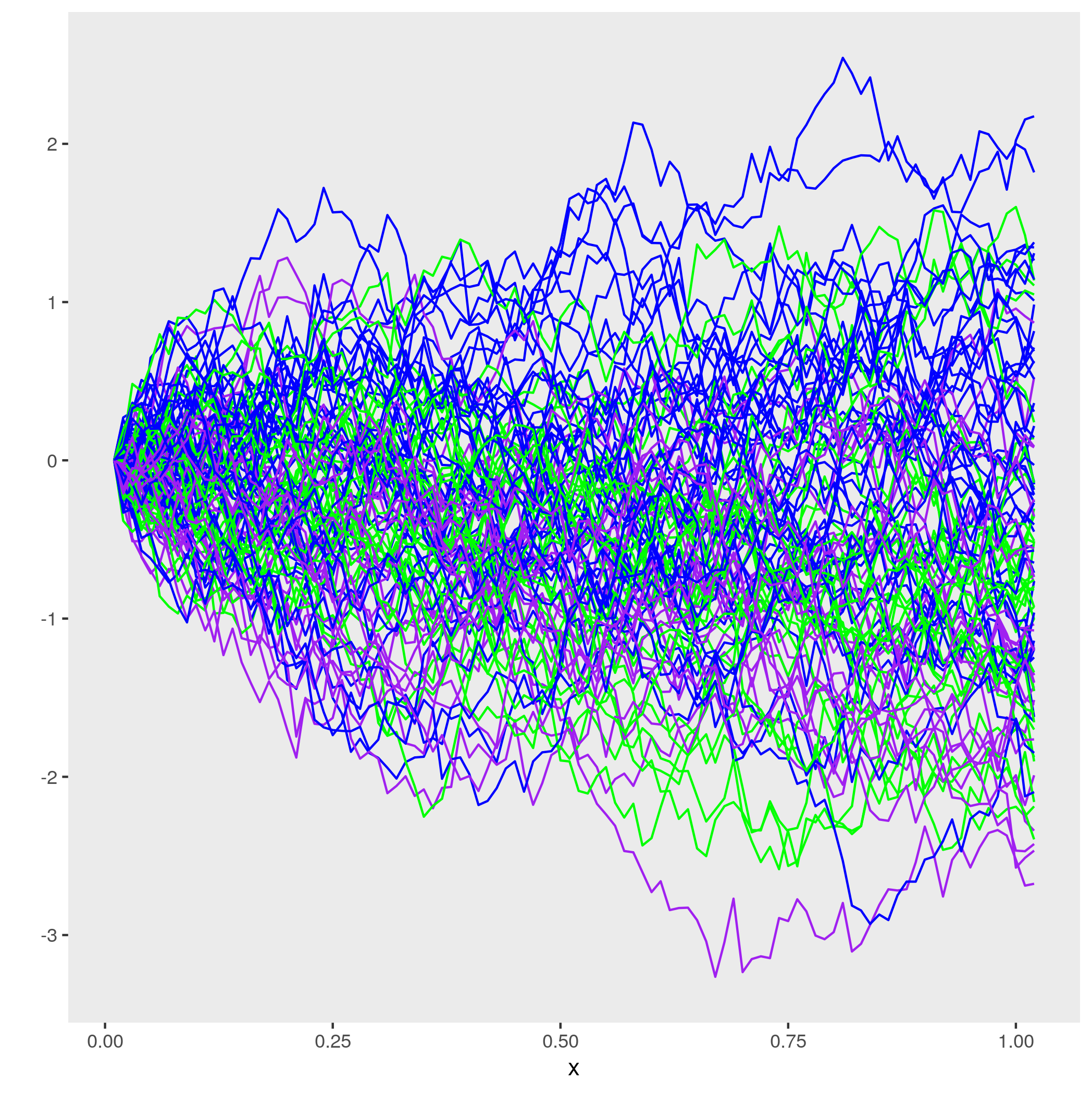}
\end{minipage}
\hfill
\noindent
\begin{minipage}[t]{.55\textwidth}
\raggedleft
  \centering
  \includegraphics[width=0.8\linewidth, height=0.2\textheight]{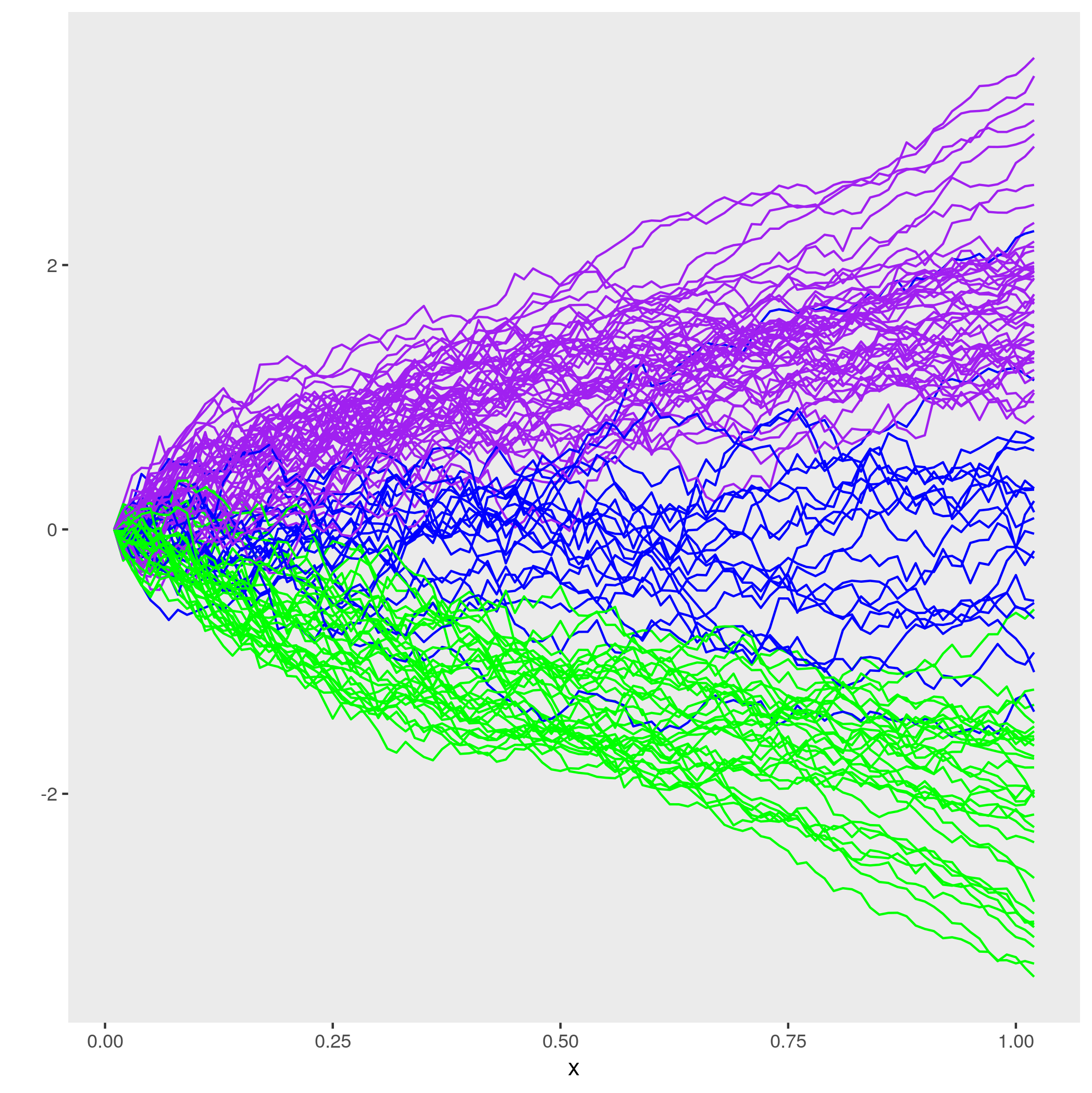}
\end{minipage}
\caption{{\small \textit{Dispersion of diffusion paths from model given in \textcolor{black}{Table}~\ref{tab:DiffModels}. Left: $\theta=1/2$, right: $\theta=4$ (blue lines $K=1$, purple lines $K=2$, green lines $K=3$); with $N=100$ and $n=100$.}}}
\label{fig:path-model5}
\end{figure}

\paragraph{Performance of the Bayes classifier.}
We evaluate the performance of the Bayes classifier $g^{*}$ with respect to 
four values of parameter $\theta$ ($\theta  \in \{1/2, 3/2, 5/2, 4\}$). 
To this end, we compute its average error rate 
over $100$ repetitions of the following steps
\begin{enumerate}[label=(\roman*)]
\item simulate $\mathcal{D}_M$ of size $M =4000$ with $n=500$; 
\item based on $\mathcal{D}_M$ compute the misclassification error rate
of the discrete counterpart of $g^{*}$.
\end{enumerate}
Table~\ref{tab:RiskBayes} provides the mean and standard deviation of the misclassification risk. 
The obtained results highlight the significant impact of the minimum distance $\theta$, between the drift functions of each class, on the performance of $g^{*}$. Indeed, as expected, the Bayes classifier is more accurate on our model when parameter $\theta$ is large, especially in the case of separable data ($\theta=4$). On the contrary, the worst case corresponds to $\theta = 0.5$. In this model, the data are highly ambiguous. 
\begin{table}[hbtp]
\centering
\renewcommand{\arraystretch}{1.5}
\begin{tabular}{c|c|c|c|c}
\hline
 & $\theta = 1/2$ & $ \theta = 3/2$ & $ \theta = 5/2$ & $\theta = 4$ \\
\hline
 $\w{\mathcal{R}}(g^{*})$ & 0.49 \ (0.01) & 0.36 \ (0.01)  & 0.22 \ (0.01) &  0.11 \ (0.01)    \\
\hline
\end{tabular}
\caption{{\small \textit{Classification risks of the Bayes classifier $g^{*}$ w.r.t parameter $\theta$ from learning samples of size $N=4000$ with $n=500$.}}}
\label{tab:RiskBayes}
\end{table}
%%%%%%%%%%%%%%%%%%%%%%%%%%%%%%%%

\subsection{Implementation of the plug-in classifier}
\label{subsec:ImplClass}
Hereafter, we briefly describe the implementation of the proposed plug-in
classifier. We first estimate the drift functions $b_i^*$, $i =1,2,3$.
For each $i \in \{1,2,3\}$, the estimator $\w{b}_i$ is built on the interval
$[-A_{N_i},A_{N_i}]$. Since the drifts (and the diffusion) coefficients are bounded, we can use the construction considered in Section~\ref{sec:classProc}. Therefore, we fix $A_{N_i} =\log(N)$, $M=3$, and divide the learning sample $\mathcal{D}_N$ into sub-samples $\mathcal{D}^{i}_{N}$ of size $N_i$ that contains all diffusion paths belonging to the class $i$. From the sub-sample $\mathcal{D}_N^i$, we build estimators $\w{b}_i$, $i=1,2,3$.

For the construction of the estimator $\w{b}_i$, we  have to choose the dimension parameter $K_{N_i}$. We follow~\cite{denis2020ridge}, and consider an adaptive choice denoted by $\w{K}_{N_i}$. 

Let us remind the reader that in~\cite{denis2020ridge}, the adaptive dimension $\w{K}_{N_i}$  is selected such that $\w{K}_{N_i}$ is the minimizer of the following penalized contrast
\begin{equation}
\label{eq:SelectDimDrift}
\w{K}_{N_i}:={\argmin{K \in\mathcal{K}}}{\left\{\frac{1}{Nn}\sum_{j=1}^{N}{\sum_{k=0}^{n-1}{(\w{b}_{i,K}-Z^{j}_{k\Delta_n})^2}}+\mathrm{pen}_b(K)\right\}},
\end{equation}
where $\mathcal{K} \in \left\{2^q, \ q\in[\![0,5]\!]\right\}$, and  $\w{b}_{i,K}$ is the drift estimator built on the approximation subspace $\cS_{K,M}$. Besides, $\mathrm{pen}_b(K)=\kappa(K+M)\log^{3}(N)/N$ is the penalty function with $\kappa>0$. We fix the parameter $\kappa=0.1$ as recommended in~\cite{denis2020ridge}.
%
%of the approximation subspace $\cS_{K_{N_i},M}$
%For our numerical study, we implement an adaptive choice of the dimension $K_N$ among the collection. The chosen value is denoted by $\w{K}_N$.
%named $\w{K}_N$ in order to choose one estimator of each function $b_i^*,\sigma^2$ among the collection, when $\mathcal{K}_{N_i}=\left\{2^q, \ q\in[\![0,5]\!]\right\}$ is the set of possible values of $K_{N_i}$.
%in which the selection of an optimal dimension $\w{K}_N+M$ is  made. 
%To do so, we follow the steps initiated in \cite{denis2020ridge}. 
%Let us remind the reader that in this work, the best dimension $\w{K}_N+M$ of the approximation subspace $\cS_{K_N,M}$ is selected such that $\w{K}_N$ is the solution of the following penalized contrast
%\begin{equation}\label{eq:SelectDimDrift}
%\w{K}_N:=\underset{K_N\in\mathcal{K}_N}{\argmin{}}{\left\{\frac{1}{Nn}\sum_{j=1}^{N}{\sum_{k=0}^{n-1}{(\w{b}_{i,K_N}-Z^{j}_{k\Delta_n})^2}}+\mathrm{pen}_b(K_N)\right\}},
%\end{equation}
%where $\mathrm{pen}_b(K_N)=\kappa(K_N+M)\log^{4}(N)/N$ is the penalty function with $\kappa>0$. We fix the parameter $\kappa=0.1$ as recommended in~\cite{denis2020ridge}. 

For the estimation of $\sigma^2$, we consider the whole sample $\mathcal{D}_N$ and apply the methodology described in Section~\ref{sec:classProc} with $M=3$. 
We follow the same lines to build an adaptive estimator of $\sigma^{2*}$,
and choose $\w{K}_N$ as the minimizer over $\mathcal{K}$ of the following penalized contrast
\begin{equation}
    \label{eq:SelectDimSigma}
\w{K}_N:={\argmin{K\in\mathcal{K}}}{\left\{\frac{1}{Nn}\sum_{j=1}^{N}{\sum_{k=0}^{n-1}{(\w{\sigma}^{2}_{K}-U^{j}_{k\Delta_n})^2}}+\mathrm{pen}_{\sigma}(K)\right\}},
\end{equation}
where $\hat{\sigma}^2_K$ is the estimator built on $\mathcal{S}_{K,M}$, and $\mathrm{pen}_{\sigma}(K):=\kappa(K+M)\log^{3}(N)/Nn$ is the penalty function, with $\kappa>0$.
The value of the tuning parameter $\kappa$ is calibrated through an intensive simulation study and chosen equal to $\kappa=5$.

\textcolor{black}{The function \texttt{SDEclassif} of the \texttt{R}-package \href{https://github.com/Eddymichelella/SDEclassif.git}{SDEclassif}, available on github, implements the resulting plug-in  classifier.}

\subsection{Simulation results}
\label{subsec:results}

%We now focus on the performance of the plug-in classification procedure. 
The performance of the plug-in classifier $\w{g}$ is evaluated by repeating $100$ times the following steps
\begin{enumerate}
\item Simulate learning samples $\mathcal{D}_{N}$ and $\mathcal{D}_{N^{\prime}}$ with $N\in\{100,1000\}, \ N^{\prime}=1000$, and $n\in\{100,500\}$;
\item for each  $i\in\{1, 2 ,3\}$, from the sub-sample $D^{i}_{N}=\{\bar{X}^{j}, j\in\mathcal{I}_i\}$, select $\w{K}_N$ minimizing  \eqref{eq:SelectDimDrift} and compute the estimator $\w{b}_{i,\w{K}_N}$ of $b^{*}_{i}$ given in Equation~\eqref{eq:boundedbi};
\item from $\mathcal{D}_N$ select $\w{K}_N$ using Equation \eqref{eq:SelectDimSigma} and compute the estimator $\w{\sigma}^{2}_{\w{K}_N}$ of $\sigma^{*2}$ given in \eqref{eq:boundedsigma};
\item based on $\mathcal{D}_N$ compute $\w{\mfp}=\left(\frac{1}{N}\sum_{j=1}^{N}{\one_{Y_j=1}}, \frac{1}{N}\sum_{j=1}^{N}{\one_{Y_j=2}} ,\frac{1}{N}\sum_{j=1}^{N}{\one_{Y_j=3}}\right)$;
\item based on $\mathcal{D}_{N^{\prime}}$, compute the error rate of
the plug-in classifier $\w{g}$ where $\w{\mathbf{b}}=\left(\w{b}_{1,\w{K}_N}, \w{b}_{2,\w{K}_N} ,\w{b}_{3,\w{K}_N}\right)$ and $\w{\sigma}^{2}=\w{\sigma}^{2}_{\w{K}_N}$, and $\w{\mfp}$.
\end{enumerate}
From these repetitions, we compute the empirical mean and standard deviation of the error rate of $\w{g}$.
The results are given in Table~\ref{tab:RiskModel1} and Figure~\ref{fig:plugin}. 
As expected, from Table~\ref{tab:RiskModel1} and Table~\ref{tab:RiskBayes}, we can see that the error rate of the plug-in classifier $\w{g}$ is closed to the error rate of the Bayes classifier. In particular, for $N=1000$, it performs as well as the Bayes classifier. Note that the length of the paths $n$ does not significantly impact the performance of $\w{g}$.
Moreover, from Figure~\ref{fig:plugin}, we can make similar comments as for the Bayes classifier (see Table~\ref{tab:RiskBayes}), in particular, the accuracy of $\w{g}$ decreases as parameter $\theta$ increases.
\begin{table}[hbtp]
\centering
\renewcommand{\arraystretch}{1.5}
\begin{tabular}{c|c|c|c|c}
\hline
\multirow{2}{*}{$\w{\mathcal{R}}(\w{g})$} & \multicolumn{2}{c}{$n=100$} & \multicolumn{2}{c}{$n=500$} \\
& $N=100$ & $N=1000$ & $N=100$ & $N=1000$\\
\hline
  $\theta =1/2$ & 0.53 \ (0.05) & 0.50 \ (0.05)  & 0.53 \ (0.05) &  0.49 \ (0.05)    \\
$ \theta =3/2$ & 0.39 \ (0.06) & 0.37 \ (0.05) & 0.39 \ (0.05) & 0.36 \ (0.05)  \\            
$ \theta =5/2$ & 0.24 \ (0.05) & 0.22 \ (0.04) & 0.25 \ (0.04) & 0.22 \ (0.04)  \\
$\theta = 4$ & 0.12 \ (0.03) & 0.10 \ (0.03) & 0.11 \ (0.03) & 0.10 \ (0.03)  \\
\hline
\end{tabular}
\caption{{\small \textit{Risks of the plug-in classifier $\w{g}$ w.r.t. values of parameter $\theta$}}}
\label{tab:RiskModel1}
\end{table}
\begin{figure}[hbtp]
    \centering
\includegraphics[width=0.7\linewidth, height=0.3\textheight]{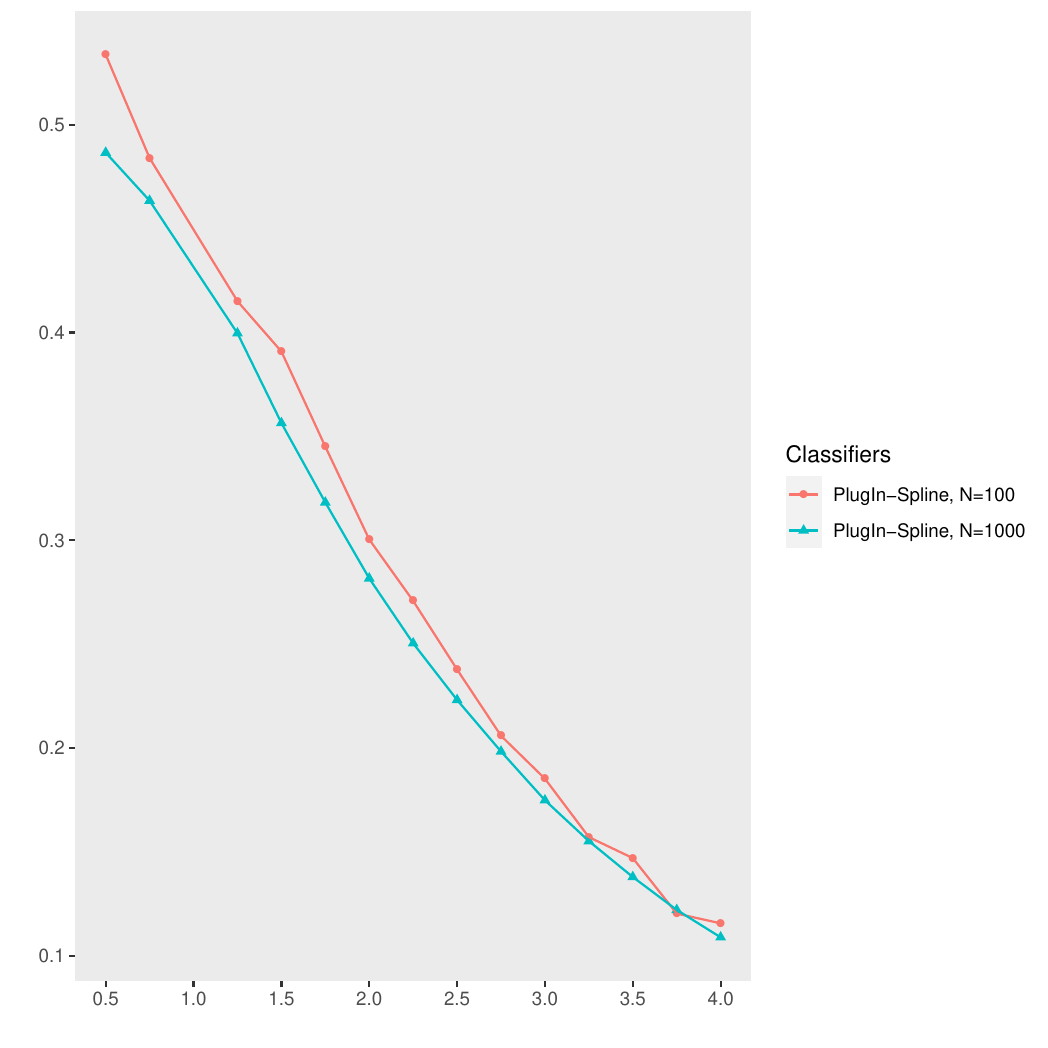}
    \caption{{\small \textit{Risks of the plug-in classifier w.r.t values of the minimum gap $\theta$ between the drift functions}}}
    \label{fig:plugin}
\end{figure}

\subsection{Ornstein-Uhlenbeck model}
\label{subsec:influenceOfsigma}

In this section, we focus on the influence of the diffusion coefficient $\sigma^{*}$ on the performance of our plug-in procedure. To this end, we consider the Ornstein-Uhlenbeck diffusion model given in Table~\ref{tab:OU} where the diffusion coefficient $\sigma^*$ is constant. Let us notice also that in this model the drift functions are unbounded. 
\begin{table}[hbtp]
\begin{center}
\renewcommand{\arraystretch}{1.25}
\begin{tabular}{c|c}
    \hline 
 $b^{*}_{1}(x)$ & $1-x$ \\ 
 $b^{*}_{2}(x)$ & $-1-x$ \\ 
$b^{*}_{3}(x)$ & $-x$\\
$\sigma^{*}(x)$ & $\sigma$ \\
\hline
\end{tabular}
\caption{{\small \textit{Ornstein-Uhlenbeck mixture model with $K=3$}}}
\end{center}
\label{tab:OU}
\end{table}
We investigate the performance of the plug-in classifier $\w{g}$ \emph{w.r.t.} the level of noise $\sigma^*$. This study is motivated by the fact that, inherently, the diffusion coefficient impacts the dispersion of the trajectories. Therefore, it can lead to separable data when $\sigma^*$ is close to zero, and ambiguous data for large values of $\sigma^*$. Thus, we evaluate the performance of $\w{g}$ for $\sigma^*=1/2$ which is close enough to zero, and for larger value $\sigma^* \in\{1,3/2\}$.  
We first consider the case where $\sigma^*$ is unknown. The results are given in Table~\ref{tab:Risk-OU} and confirm our intuition. The error rate of the plug-in classifiers decreases as $\sigma^*$ decreases.

In a second step, we investigate the influence of estimating the coefficient $\sigma^*$ in the procedure. To evaluate this point, we assess the error rate of the plug-in classifier when $\sigma^{*}=1$ is known. In this case, we only estimate the drift functions and the weights of mixture $\mfp{p}^*$ to build our predictor. The results are given in Table~\ref{tab:Risk-OU2}. First, we can notice that by comparison with results provided in Table~\ref{tab:Risk-OU}, there is almost no impact on the performance of the plug-in classifier when we assume the diffusion coefficient $\sigma^{*}$ in the Ornstein-Uhlenbeck model to be known or not.

Finally, we also study the influence of parameter $A_N$ on the estimation procedure. Indeed, our theoretical results indicates that $A_N$ should be of order 
$\sqrt{\log(N)}$ when $\sigma^*$ is constant and known, while  $A_N=\log(N)$
is recommended when $\sigma^*$ is unknown. To this end, we evaluate the error rate of our procedure for these choices. The results are also provided in Table~\ref{tab:Risk-OU2} and show that the performance are almost the same in the two cases.

\begin{table}[hbtp]
\centering
\renewcommand{\arraystretch}{1.5}
\begin{tabular}{l|c|c}
\hline
& $\w{\cR}(\w{g})$ & $\w{\cR}(g^{*})$ \\
\hline
\multirow{1}{*}{  $\sigma^*=1/2$} & 0.23 (0.04) & 0.21 (0.01)  \\
%\hline
\multirow{1}{*}{$\sigma^*=1$} & 0.44 (0.05) & 0.41 (0.01) \\
%\hline
\multirow{1}{*}{$\sigma^*=3/2$} & 0.52 (0.05) & 0.49 (0.01) \\
\hline
\end{tabular}
\caption{{\small \textit{Evolution of the performance of the plug-in classifier $\w{g}$ and of $g^*$ w.r.t values of the constant diffusion coefficient $\sigma^{*}$ for $N=100$ and $n=100$. }}}
%Comparison with the performance of the Bayes classifier $g^{*}$
\label{tab:Risk-OU}
\end{table}
\begin{table}[hbtp]
\centering
\renewcommand{\arraystretch}{1.5}
\begin{tabular}{c|c|c}
\hline
& $N=100$ & $N=1000$\\
\hline
$A_N=\sqrt{\log(N)}$ & 0.44 (0.05) & 0.41 (0.05)  \\
 $A_N=\log(N)$ & 0.43 (0.05) & 0.43 (0.05) \\
\hline
\end{tabular}
\caption{{\small \textit{Risk classification of $\w{g}$ when the diffusion $\sigma^{*}=1$ is known, and $n=100$.}}}
\label{tab:Risk-OU2}
\end{table}

\section{Conclusion and discussion}
\label{sec:conclusion}
In this paper, we propose a plug-in classifier for the multiclass classification of trajectories generated by a mixture of diffusion processes whose drift functions $b^{*}_{i}, \ i\in\cY$ and diffusion coefficient $\sigma^{*}$ are assumed to be unknown.
In the considered model, each class $i$ is characterized by a drift function, 
$b^{*}_{i}$ whereas the diffusion coefficient $\sigma^{*}$ is common for all classes. This work extends to the nonparametric case, the multiclass classification procedure provided in~\cite{denis2020classif} where $\sigma^{*}=1$ and the drift functions depend on an unknown parameter $\theta \in \R^d$.
Our proposed procedure relies on consistent projection estimators $\w{b}_i, i \in \cY$ and $\w{\sigma}^{2}$ of the drift and diffusion coefficients on a constrained approximation subspace spanned by the spline basis. We establish the consistency, \emph{w.r.t.} the excess risk, of our procedure and then studied its rate of convergence under different kind of assumptions.
In particular, we show that the proposed plug-in classifier reaches a rate of convergence of order $N^{-1/5}$ (up to a factor of order $\exp(\sqrt{c\log(N)})$) when ${\bf b}^{*}, \ \sigma^{*}$, and $\mfp^{*}$ are unknown.
Besides, a numerical study illustrates the performance of our classification procedure.

In the case where $\sigma^{*}=1$, we manage to derive faster rates of convergence. In particular, when the drift functions are bounded and H\"older with regularity $\beta \geq 1$, we obtained a rate of order $N^{-\beta/(2\beta+1)}$ (up to a factor of order $\exp(\sqrt{c\log(N)})$).
Interestingly, this result can be viewed as an extension of the one obtained in~\cite{gadat2020optimal} to the multiclass mixture model, where the drift 
functions are time-dependent. Furthermore, up to $\exp(\sqrt{c\log(N)})$ factor, our rate of convergence matches the optimal rates of convergence 
obtained in the univariate setting (\emph{e.g.} $X \in \R$), in~\cite{audibert2007fast}.
Finally, for the case of unbounded drift functions, we assume that the drift functions are the re-entrant. Taking advantage of this property, we establish that our plug-in classifier achieves a rate of convergence of order $N^{-3\beta/4(2\beta+1)}$. For $\beta =1$, this rate of convergence is of the same order as the one obtained in~\cite{Gadat_Klein_Marteau16} for plug-in classifier in the univariate classification setting, when the feature $X$ does not satisfy the strong density assumption.

A question that can be tackled for future research is the study of the optimality in the minimax sense of our plug-in procedure. In particular, the adaptivity of estimators of the drift and diffusion coefficients should be investigated. Furthermore, it might be interesting to consider the margin type assumption as in~\cite{gadat2020optimal} to derive faster rates of convergence. 
Also, following~\cite{denis2020classif}, it is natural to derive theoretical properties for empirical risk minimization procedure based on convex losses.
Finally, the extension to the  high-dimensional setting would require further work. In particular, the control of the transition densities is different in this setting.

%%%%%%%%%%%%%%%%%%%%%%%%%%%%%%%%%%%%%%%%%%%%%%%%%%%%%%%%%%%%
\section{Proofs}\label{sec:proofs}
The section is devoted to the proofs of our main results.
In order to simplify the notation, we write $\Delta_n= \Delta$. 
Besides, $C>0$ is a constant which may change from one line to another. When the dependency on a parameter $\theta$ needs to be highlighted, we write $C_\theta$.

\subsection{Technical results on the process $X$}

\begin{lemme}
\label{lem:discrete}
\textcolor{black}{Under Assumption~\ref{ass:RegEll} and for} all integer $q\geq 1$, there exists $C^{*}>0$ depending on $q$ such that for all $0\leq s<t\leq 1$,
\begin{equation*}
\E\left|X_t-X_s\right|^{2q}\leq C^{*}(t-s)^{q}.
\end{equation*}
\end{lemme}
\textcolor{black}{The proof of Lemma~\ref{lem:discrete} is provided in Appendix.} 

For each $t \in [0,1]$ and $x \in \mathbb{R}$,
we denote by \textcolor{black}{$p_X(t,x)$} the transition density of the underlying process $X_t$ given the starting point $X_0=0$. We also denote by \textcolor{black}{$p_{i,X}(t,\cdot)$} the transition density of the process driven by the drift function $b_i^*$. Note that Assumption~\ref{ass:RegEll} ensures the existence of the transition densities. The rest of this section is dedicated to some results on the transition densities \textcolor{black}{$p_{i,X}$} for $i = 1, \ldots, K$. Nonetheless, since the transition \textcolor{black}{$p_X$}  of the process $X$ writes as
\begin{equation*}
\textcolor{black}{p_X = \sum_{i = 1}^K \mfp^{*}_i p_{i,X}},
\end{equation*}
all these results apply also for \textcolor{black}{$p_X$}.
The following proposition is provided in~\citep{gobet2002lan} (Proposition 1.2).
\begin{prop}
\label{prop:densityTransition}
Under Assumptions~\ref{ass:RegEll} and~\ref{ass:Novikov}, there exist constants $c >1$, $\textcolor{black}{C > 1}$ such that for all $t \in (0,1]$, $x \in \mathbb{R}$, and $i=1,\ldots,K$
\begin{equation*}
\frac{1}{\textcolor{black}{C}\sqrt{t}} \exp\left(-c\frac{x^2}{t}\right) \leq  \textcolor{black}{p_{i,X}(t,x)} \leq \dfrac{\textcolor{black}{C}}{\sqrt{t}} \exp\left(-\frac{x^2}{ct}\right).
\end{equation*}
\end{prop}
From this result, we can deduce an evaluation of the probability of the process to exit a compact set. This is the purpose of the next result.
\begin{lemme}
\label{lem:controleSortiCompact}
Under Assumption~\ref{ass:RegEll} and \ref{ass:Novikov}, there exist $C_1,C_2 >0$ such that for all $A >0$
\begin{equation*}
\sup_{t \in [0,1]}\P\left(\left|X_t\right|\geq A\right) 
\leq \frac{C_1}{A} \exp(-C_2A^2).
\end{equation*}
\end{lemme}
\begin{proof}
Let $A >0$, we have for $t \in (0,1]$,
\begin{equation*}
\P\left(\left|X_t\right| \geq A\right) = \textcolor{black}{\int_{A}^{+\infty} p_X(t,x) \mathrm{d}x + \int_{A}^{+\infty} p_X(t,-x) \mathrm{d}x}.   
\end{equation*}
From Proposition~\ref{prop:densityTransition}, we then deduce that
\begin{eqnarray*}
\P\left(\left|X_t\right| \geq A\right)  
 \leq   C \dfrac{\sqrt{t}}{A} \int_{A}^{+\infty} \textcolor{black}{c\frac{2x}{t}}\exp\left(-c\frac{x^2}{t}\right)\mathrm{d}x
 \leq  \frac{C\sqrt{t}}{A} \exp\left(-\frac{cA^2}{t}\right).
\end{eqnarray*}
From the above inequality, and using that $t\in (0,1]$, we deduce the result.
\end{proof}
\begin{lemme}
\label{lem:boundDensity}
\textcolor{black}{Under Assumption~\ref{ass:RegEll}, there} exist $C_0, C_1$, and $C_2$, such that for $i =1,\ldots,K$, for $x \in [-A,A]$, we have
\begin{equation*}
C_1 \exp\left(-C_2 \textcolor{black}{A^2}\right)\leq \dfrac{1}{n} \sum_{k=1}^{n-1} \textcolor{black}{p_{i,X}(k\Delta,x)} \leq    C_0.
\end{equation*}
\end{lemme}
\begin{proof}[\textbf{Proof of Lemma \ref{lem:boundDensity}}]
For $i \in \{1, \ldots, K\}$, for all $x\in \R $, we have from Proposition~\ref{prop:densityTransition}, 
\begin{equation}
\label{eq:eqLemmeBoundDensity1}
    \frac{1}{n}\sum_{k=1}^{n}{\textcolor{black}{p_{i,X}(k\Delta,x)}} \leq\frac{C}{n}\sum_{k=1}^{n}{\frac{1}{\sqrt{k\Delta}}}
    = \frac{C}{\sqrt{n}} \sum_{k=1}^n \frac{1}{\sqrt{k}} \leq \frac{2C}{\sqrt{n}} \sum_{k=1}^n \frac{1}{\sqrt{k+1}}.
\end{equation}
Since the function $x\mapsto\frac{1}{\sqrt{x}}$ is decreasing over $[1,+\infty[$, we deduce 
from Equation~\eqref{eq:eqLemmeBoundDensity1} that
\begin{equation*}
\frac{1}{n}\sum_{k=1}^{n}{\textcolor{black}{p_{i,X}(k\Delta,x)}} \leq \dfrac{4C\sqrt{n+1}}{\sqrt{n}} \leq C_0,  
\end{equation*}
which gives the upper bound.
For the lower bound, we observe from Proposition~\ref{prop:densityTransition} that for $k \in [\![1,n-1]\!]$, and $x \in \R$, 
\begin{equation}
\label{eq:eqLemmeBoundDensity2}
 C \exp\left(-\frac{cx^2}{k\Delta}\right) \leq \frac{C}{\sqrt{k\Delta}}\exp\left(-\frac{cx^2}{k\Delta}\right) \leq  p_i(k\Delta,x).
\end{equation}
Since $g : (s,x)\mapsto\exp\left(-\frac{cx^2}{s}\right)$ is strictly increasing in $s$ over $(0,1]$, we obtain for $k \in [\![1,n-1]\!]$,
\begin{equation*}
\int_{\frac{1}{n}}^{\frac{n-1}{n}}{g(s,x)ds} \leq 
\sum_{k=2}^{n-1}{\int_{(k-1)\Delta}^{k\Delta}{\left(g(k\Delta,x)+g(s,x)-g(k\Delta,x)\right)ds}} \leq \frac{1}{n}\sum_{k=1}^{n-1} \exp\left({-\frac{cx^2}{k\Delta}}\right).
\end{equation*}
Hence, we deduce that for $n \geq 3$, and $x \in [-A,A]$
\begin{equation*}
 \dfrac{1}{6}\exp(-2cA^2) \leq \int_{\frac{1}{2}}^{\dfrac{n-1}{n}} g(s,x) \mathrm{d}s  \leq \frac{1}{n}\sum_{k=1}^{n-1} \exp\left({-\frac{cx^2}{k\Delta}}\right).
\end{equation*}For the first lower bound, we use that $g(s,x)\geq e^{-2cA^2}$ for $x\in [-A,A]$ and $s\geq 1/2$, and that the length of $[1/2,(n-1)/n]$ is larger than $1/6$ for $n\geq 3$. \textcolor{black}{This explains our choice of integration interval in the middle term of the above inequalities.}
Finally, gathering this bound with Equation~\eqref{eq:eqLemmeBoundDensity2}, leads to 
$$\frac{1}{6}  \exp(-2cA^2)\leq \frac{1}{n} \sum_{k=1}^{n-1}  \exp\left(-\frac{cx^2}{k\Delta}\right) \leq \frac{1}{n}\sum_{k=1}^{n-1} \textcolor{black}{p_{i,X}(k\Delta,x)}.
$$
\end{proof}

\begin{lemme}
\label{lm:DensityConstSigma}
Suppose that $\sigma^*$ is a constant. \textcolor{black}{Under Assumption~\ref{ass:RegEll}, and for} all $q>1$, there exists $K_q>1$ such that for all $(t,x)\in(0,1]\times [-A,A]$,
$$\frac{1}{K_q \sqrt{t}} \exp \left(-\frac{2 q-1}{2 q \sigma^{* 2} t} x^2\right) \leq \textcolor{black}{p_{X}(t,x)} \leq \frac{K_q}{\sqrt{t}} \exp \left(-\frac{x^2}{2 q \sigma^{* 2} t}\right).$$
\end{lemme}
\begin{proof}[\textbf{Proof of Lemma \ref{lm:DensityConstSigma}}]
The transition density \textcolor{black}{$p^{0}_{X}$} of the process $\left(0+\sigma^{*}W_t\right)_{t\in]0,1]}$ (with a constant diffusion coefficient $\sigma^{*}$) is given by 
\begin{equation}\label{eq:p0}
\textcolor{black}{p^{0}_{X}(t,x)}:=\frac{1}{\sqrt{2\pi\sigma^{*2}t}}\exp\left(-\frac{1}{2\sigma^{*2}t}\left|0-x\right|^2\right).
\end{equation}
We are going to demonstrate the inequality for \textcolor{black}{$p_{i,X}$}, which is the transition density of $X$ in class number $i$. Indeed, then it will be true for all $i \in \cY$ and thus for $p_X=p_{Y,X}$.
We follow here the arguments given in the \textit{proof of (1.6)} in \cite{gobet2002lan}. 
Let us denote, 
$$Z_{i,t}=\exp\left(\int_{0}^{t}{\frac{b^{*}_{i}(X_s)}{\sigma^{*}}dW_s}-\int_{0}^{t}{\frac{b^{*2}_{i}(X_s)}{\sigma^{*2}}ds}\right).$$
We have $\forall (t,x)\in]0,1]\times \R$, 
\begin{align*}
 \textcolor{black}{p_{i,X}(t,x)=p^{0}_{X}(t,x)\E^{0}\left[Z_{i,t}|X_t=x\right]}, 
\end{align*}
and
\begin{equation}\label{eq:pip0}
\textcolor{black}{\frac{1}{p_{i,X}(t,x)}\leq\frac{1}{p^{0}_{X}(t,x)}\E^{0}\left[Z_{i,t}^{-1}|X_t=x\right]}.\end{equation}
Then, 
$$
\E^{0}\left[Z_{i,t}|X_t=x\right]=1+\frac{1}{\textcolor{black}{p^{0}_{X}(t,x)}}\int_{0}^{t}{\E^{0}\left[Z_{i,t}b^{*}_{i}(X_s)\frac{X_s-x}{\sigma^{*2}(t-s)}\textcolor{black}{p^{0}_{X}(t-s,x)}\right]ds}
$$
and
$$
\E^{0}\left[Z_{i,t}^{-1}|X_t=x\right]=1+\frac{1}{\textcolor{black}{p^{0}_{X}(t,x)}}\int_{0}^{t}{\E^{0}\left[Z_{i,s}^{-1}b^{*}_{i}(X_s)\frac{X_s-x}{\sigma^{*2}(t-s)}\textcolor{black}{p^{0}_{X}(t-s,x)}\right]ds}.
$$
For all $(t,x)\in]0,1]\times \R$, one has :
\begin{align*}
\E^{0}\left[Z_{i,t}|X_t=x\right]&=1+\frac{1}{\textcolor{black}{p^{0}_{X}(t,x)}}\int_{0}^{t}{\E^{0}\left[Z_{i,s}b^{*}_{i}(X_s)\frac{X_s-x}{\sigma^{*2}(t-s)}\textcolor{black}{p^{0}_{X}(t-s,x)}\right]ds}\\
&\leq 1+\frac{C}{\textcolor{black}{p^{0}_{X}(t,x)}}\int_{0}^{t}{\E^{0}\left[Z_{i,s}\left|b^{*}_{i}(X_s)\right|\frac{|X_s-x|}{(t-s)^{3/2}}\exp\left(-\frac{(X_s-x)^{2}}{2\sigma^{*2}(t-s)}\right)\right]ds}\\
&\leq 1+\frac{C}{\textcolor{black}{p^{0}_{X}(t,x)}}\int_{0}^{t}{\E^{0}\left[Z_{i,s}\left|b^{*}_{i}(X_s)\right|\frac{1}{\varepsilon(t-s)}\exp\left(-\frac{(1-\varepsilon)(X_s-x)^{2}}{2\sigma^{*2}(t-s)}\right)\right]ds}
\end{align*}
using that $y \varepsilon \exp(- \varepsilon y^2/2) \leq 1$ for $0<\varepsilon<1$.
Let $q,q^{\prime}>1$ be two real numbers such that $\frac{1}{q}+\frac{1}{q^{\prime}}=1$.
Using H\"older's inequality, and the Lipschitz property of $b^*$, one has:
\begin{equation}\label{eq:E0}
\E^{0}\left[Z_{i,t}|X_t=x\right]\leq 1+\frac{C \varepsilon^{-1}}{\textcolor{black}{p^{0}_{X}(t,x)}}\int_{0}^{t}{\left(\E^{0}\left[\frac{Z^{q}_{i,s}{q}\left(1+|X_s|\right)^{q}}{(t-s)^{q}}\right]\right)^{\frac{1}{q}}\left(\E^{0}\left[\exp\left(-\frac{\left(X_s-x\right)^{2}}{2\sigma^{*2}(t-s)}\right)\right]\right)^{\frac{1}{q^{\prime}}}ds}
\end{equation}
with $\varepsilon = 1-1/q^{\prime}$. 
According to \textit{Lemma A.1 in \cite{gobet2002lan}}, one has:
\begin{align*}
\forall q>1, \ \ \E^{0}\left[Z^{q}_{i,s}\left(1+|X_s|\right)^{q}\right]+\E^{0}\left[Z^{-q}_{i,s}\left(1+|X_s|\right)^{q}\right]\leq C_{1}
\end{align*}
where $C_1>0$ is a constant. Thus, it remains to upper bound $\E^{0}\left[\exp\left(-\frac{\left(X_s-x\right)^{2}}{2\sigma^{*2}(t-s)}\right)\right]$ and then deduce an upper bound of $\E^{0}\left[Z_{i,t}|X_t=x\right]$. For all $s<t$, we have:
\begin{align*}
\sqrt{2\pi\sigma^{*2}s}\E^{0}\left[\exp\left(-\frac{\left(X_s-x\right)^{2}}{2\sigma^{*2}(t-s)}\right)\right] &= \int_{\R}{\exp\left(-\frac{1}{2\sigma^{*2}(t-s)}\left(z-x\right)^{2}\right)\exp\left(-\frac{1}{2\sigma^{*2}s}z^{2}\right)dz}
%\\
%&= \int_{\mathbb{R}}{\exp\left(-\frac{1}{2\sigma^{*2}(t-s)}\left(z^2-2xz+x^2\right)-\frac{1}{2\sigma^{*2}s}z^{2}\right)dz}\\
%&= \int_{\mathbb{R}}{\exp\left(-\frac{t}{2\sigma^{*2}(t-s)s}z^2+\frac{sx}{\sigma^{*2}(t-s)s}z-\frac{x^2}{2\sigma^{*2}(t-s)}\right)dz}.
\end{align*}
It follows that,
\begin{align*}
\E^{0}\left[\exp\left(-\frac{\left(X_s-x\right)^{2}}{2\sigma^{*2}(t-s)}\right)\right]&=\sqrt{\frac{t-s}{t}}\exp\left(-\frac{x^2}{2\sigma^{*2}t}\right).
%\\
%&\leq\frac{1}{t^{1/2}(t-s)^{-1/2}}\exp\left(-\frac{x^2}{2\sigma^{*2}t}\right).
\end{align*}
Thus, from Equation \eqref{eq:E0}, we obtain:
\begin{align*}
\E^{0}\left[Z_{i,t}|X_t=x\right]&\leq 1+\frac{C \varepsilon^{-1}}{\textcolor{black}{p^{0}_{X}(t,x)}}\int_{0}^{t}{\frac{(t-s)^{\frac{1}{2q^{\prime}}-1}}{t^{\frac{1}{2q^{\prime}}}}\exp\left(-\frac{x^2}{2q^{\prime}\sigma^{*2}t}\right)ds}\\
&\textcolor{black}{\leq 1+\frac{C \varepsilon^{-1} t^{-1/2q^{\prime}}}{p^{0}_{X}(t,x)}\exp\left(-\frac{x^2}{2q^{\prime}\sigma^{*2}t}\right)\left[-2q^{\prime}(t-s)^{1/2q^{\prime}}\right]^{t}_{0}}\\
&\leq 1+\frac{C \varepsilon^{-1}}{\textcolor{black}{p^{0}_{X}(t,x)}\sqrt{t}}\exp\left(-\frac{x^2}{2q^{\prime}\sigma^{*2}t}\right),
\end{align*}
%by noticing that the integral of $(t-s)^{\frac{1}{2q^{\prime}}-1}$ is smaller than 1 (since $0< s<t\leq 1$) and that $t^{-\frac{1}{2q^{\prime}}}\geq 1/\sqrt{t}$.
%
From the definition of function \textcolor{black}{$p^{0}_{X}$} given in Equation \eqref{eq:p0} together with relation \eqref{eq:pip0}, we obtain that
%$(t,x)=\frac{1}{\sqrt{2\pi\sigma^{*2}t}}\exp\left(-\frac{x^2}{2\sigma^{*2}t}\right)$. 
$$
\textcolor{black}{p_{i,X}(t,x) \leq 
%\frac{1}{\sqrt{1\pi \sigma^{*2}t} \exp(-\frac{x^2}{2\sigma^{*2}t}) 
p^{0}_{X}(t,x) \left(1+\frac{C \varepsilon^{-1}}{p^{0}_{X}(t,x)\sqrt{t}}\exp\left(-\frac{x^2}{2q^{\prime}\sigma^{*2}t}\right)\right)}.
$$
Thus, there exists a constant $K_{q}>1$ (as $\varepsilon=1-1/q'$ and $1/q+1/q'=1$) such that,
%\Chi{Dans la suite, $K_q=K_*$ ?}
\begin{equation}
\label{upperbound density}
\forall (t,x)\in]0,1]\times\R, \ \ \textcolor{black}{p_{i,X}(t,x)}\leq\frac{K_{q}}{\sqrt{t}}\exp\left(-\frac{x^2}{2q^{\prime}\sigma^{*2}t}\right), \ \ \forall q^{\prime}>1.
\end{equation} 
Following the same lines,  one has
$$
\E^{0}\left[Z^{-1}_{i,t}|X_t=x\right] \leq 1+\frac{C^{\mathrm{te}}}{\textcolor{black}{p^{0}_{X}(t,x)}\sqrt{t}}\exp\left(-\frac{x^2}{2q^{\prime}\sigma^{*2}t}\right).
$$
Also, there exists a constant $K_{q}>1$, such that,
\begin{equation}
\label{lowerbound density}
\forall (t,x)\in]0,1]\times \R \ \ \textcolor{black}{p_{i,X}(t,x)}\geq\frac{1}{K_{q}\sqrt{t}}\exp\left(-\frac{2q^{\prime}-1}{2q^{\prime}\sigma^{*2}t}x^{2}\right), \ \ \forall q^{\prime}>1.
\end{equation}
The final result is deduced from \eqref{upperbound density} and \eqref{lowerbound density}.
\end{proof}

%%%%%%%%%%%%%%%%%%%%%%%%%%%%
\subsection{Proofs of Section~\ref{sec:classProc}}

%Let us begin this section with a proposition which establishes a closed formula of the excess risk in multiclass classification.
%\begin{prop}
%\label{prop:excessRiskClass}
%Let $g$ a classifier. The following holds
%\begin{equation*}
%\mathcal{R}(g) - \mathcal{R}(g^*) =
%\E\left[\sum_{i=1}^K\sum_{j\neq i}
%\left|\pi_i^*(X)-\pi_j^*(X)\right|\one_{\{g(X) = j, g^*(X) = i\}}\right]
%\end{equation*}
%\end{prop}
%The proof of this result is omitted and can be found instance in~\cite{denis2020classif}.
%
\textcolor{black}{We begin by providing} the proof of Theorem~\ref{thm:comparisonInequality} that relies in part on Proposition~\ref{prop:excessRiskClass}. 

\begin{proof}[\textbf{Proof of Theorem~\ref{thm:comparisonInequality}}]

From Proposition~\ref{prop:excessRiskClass}, we have the following inequality
\begin{equation}
\label{eq:eqDecompExcessRisk0}
\E\left[\mathcal{R}(\w{g}) - \mathcal{R}(g^*)\right]    
\leq 2 \sum_{i=1}^K \E \left[ \left|\w{\pi}_i(X)-\pi^*_i(X)\right|\right].
\end{equation}    
We define  $\bar{{\bf F}}$  the discretized version
of  ${\bf F}^*$,
\begin{equation*}
{\bar{\bf F}} = (\bar{F}_1, \ldots, \bar{F}_K), \;\; {\rm with} \;\; \bar{F}_{i}(X)=\sum_{k=0}^{n-1}{\left(\frac{b^*_{i}}{\sigma^{*2}}(X_{k\Delta})\left(X_{(k+1)\Delta}-X_{k\Delta}\right)-\frac{\Delta}{2}\frac{b^{*2}_{i}}{\sigma^{*2}}(X_{k\Delta})\right)},  
\end{equation*}
and for each $i \in \mathcal{Y}$, $\bar{\pi}^*_i = \phi_i\left(\bar{\bf F}\right)$ the discretized version of $\pi_i^*$, {and $\bar{\pi}_i=\phi_i(\boldsymbol{\w{F}})$}.
From Equation~\eqref{eq:eqDecompExcessRisk0},
we deduce
\begin{eqnarray}
\E\left[\mathcal{R}(\w{g}) - \mathcal{R}(g^*)\right]   
&\leq & 2 \left(\sum_{i=1}^K \E \left[ \left|\w{\pi}_i(X)-{\bar{\pi}_i(X)}\right|\right]+ {\E \left[ \left|\bar{\pi}_i(X)-\bar{\pi}^*_i(X)\right|\right]}\right.\nonumber\\
&& + \left.\sum_{i=1}^K {\E \left[ \left|\bar{\pi}^*_i(X)-\pi^*_i(X)\right|\right]}\right)\nonumber\\
 &\leq & 2 \sum_{i=1}^K \E \left[ \left|\w{\phi}_i(\boldsymbol{\w{F}}(X))-\phi_i(\boldsymbol{\w{F}}(X))\right|\right]+2 \sum_{i=1}^K \E \left[ \left|\phi_i(\boldsymbol{\w{F}}(X))-\phi_i(\bar{\bf F}(X))\right|\right]\nonumber\\
 &&+2\sum_{i=1}^K \E \left[ \left|\phi_i(\bar{\bf F}(X))-\phi_i({\bf F}^*(X))\right|\right].\label{eq:eqDecompExcessRisk}
\end{eqnarray}
For the first term of the {\it r.h.s.} of the above inequality, we observe that
for $(x_1, \ldots, x_K) \in \R^K$, and $(i,j) \in \cY^2$ we have
\begin{equation*}
\left|\dfrac{\partial}{\partial_{\mfp^{*}_j}}
\frac{\mfp^{*}_i\exp(x_i)}{\sum_{k=1}^K \mfp^{*}_k\exp(x_k)}\right| \leq \frac{1}{\mfp^{*}_0}.
\end{equation*}
Therefore, 
%we deduce that
\begin{equation}
\label{eq:eqDecompExcessRisk1}
 \sum_{i=1}^K \E \left[ \left|\w{\phi}_i(\boldsymbol{\w{F}}(X))-\phi_i(\boldsymbol{\w{F}}(X))\right|\right] \leq C_{K,\mfp^{*}_0} \sum_{k=1}^K \E\left[\left|\w{\mfp}_k-\mfp_k\right| \right] \leq \frac{C_{K,\mfp^{*}_0}}{\sqrt{N}}. 
\end{equation}
For the second term of Equation \eqref{eq:eqDecompExcessRisk}, since the softmax function is $1$-Lipschitz, we have for $j \in \mathcal{Y}$
\begin{equation*}
\E \left|\phi_j(\boldsymbol{\w{F}}(X))-\phi_j(\bar{\bf F}(X))\right| \leq 
\sum_{i=1}^K\E\left[\left|\w{F}_i(X)-\bar{F}_i(X) \right|\right].  
\end{equation*}
We set  
$\xi(s):=
k\Delta, $ if
$s\in[k\Delta,(k+1)\Delta)
$, for $ k\in[\![0,n-1]\!]$.
We then deduce that
\begin{multline*}
\left|\w{F}_i(X)-\bar{F}_i(X) \right| \leq 
\int_{0}^{1}{\left|\left(\frac{\w{b}_i}{\w{\sigma}^{2}}-\frac{b^{*}_{i}}{\sigma^{*2}}\right)\left(X_{\xi(s)}\right)b^{*}_{Y}(X_s)\right|ds}
+\frac{1}{2}\int_{0}^{1}{\left|\left(\frac{\w{b}^{2}_{i}}{\w{\sigma}^{2}}-\frac{b^{*2}_{i}}{\sigma^{*2}}\right)\left(X_{\xi(s)}\right)\right|ds}\\
+\left|\int_{0}^{1}{\left(\frac{\w{b}_i}{\w{\sigma}^{2}}-\frac{b^{*}_{i}}{\sigma^{*2}}\right)\left(X_{\xi(s)}\right)\sigma^{*}(X_s)dW_s}\right|,
\end{multline*}
which implies
\begin{eqnarray*}
\E\left[\left|\w{F}_i(X)-\bar{F}_i(X) \right| \right] &\leq &
%\E\left[\int_{0}^{1}{\left|\left(\frac{b_i}{\w{\sigma}^{2}}-\frac{b^{*}_{i}}{\sigma^{*2}}\right)\left(X_{\xi(s)}\right)b^{*}_{Y}(X_s)\right|ds}\right]
%+\frac{1}{2}\E\left[\int_{0}^{1}{\left|\left(\frac{b^{2}_{i}}{\w{\sigma}^{2}}-\frac{b^{*2}_{i}}{\sigma^{*2}}\right)\left(X_{\xi(s)}\right)\right|ds}\right]\\
%&&+\E\left[\left(\int_{0}^{1}{\left(\frac{b_i}{\w{\sigma}^{2}}-\frac{b^{*}_{i}}{\sigma^{*2}}\right)\left(X_{\xi(s)}\right)\sigma^{*}(X_s)dW_s}\right)^2\right]
%\\
%&\leq &
\E\left[\int_{0}^{1}{\left|\left(\frac{\w{b}_i}{\w{\sigma}^{2}}-\frac{b^{*}_{i}}{\sigma^{*2}}\right)\left(X_{\xi(s)}\right)b^{*}_{Y}(X_s)\right|ds}\right]
+\frac{1}{2}\E\left[\int_{0}^{1}{\left|\left(\frac{\w{b}^{2}_{i}}{\w{\sigma}^{2}}-\frac{b^{*2}_{i}}{\sigma^{*2}}\right)\left(X_{\xi(s)}\right)\right|ds}\right]\\
&&+\E\left[\int_{0}^{1}{\left(\frac{\w{b}_i}{\w{\sigma}^{2}}-\frac{b^{*}_{i}}{\sigma^{*2}}\right)^2\left(X_{\xi(s)}\right)\sigma^{*2}(X_s)ds}\right].
\end{eqnarray*}
Since for all $x$, $\sigma^*(x) \geq \sigma_0^*$, and
$\w{\sigma} \geq {\sigma}_0$, we get 
\begin{equation}\label{eq:majratio}
\begin{cases}
\left|\frac{\w{b}_i}{\w{\sigma}^{2}}(x)-\frac{b^{*}_{i}}{\sigma^{*2}}(x)\right| \leq  {\sigma}_{0}^{-2}\left|\w{b}_i(x)-b^{*}_{i}(x)\right|+{\sigma}_{0}^{-2}\sigma^{*-2}_{0}\left|b^{*}_i(x)\right|\left|\w{\sigma}^{2}(x)-\sigma^{*2}(x)\right|,\\ \\
\left|\frac{\w{b}^{2}_i}{\w{\sigma}^{2}}(x)-\frac{b^{*2}_{i}}{\sigma^{*2}}(x)\right|  \leq 
{\sigma}_{0}^{-2} \left|\w{b}_i(x)+b^{*}_{i}(x)\right|\left|\w{b}_i(x)-b^{*}_{i}(x)\right|+{\sigma}_{0}^{-2}\sigma^{*-2}_{0}\left|b^{*}_i(x)\right|^2\left|\w{\sigma}^{2}(x)-\sigma^{*2}(x)\right|.
\end{cases}
\end{equation}
Hence, as $\w{b_i}(x) \leq b_{\rm max}$, and 
$\E\left[\sup_{t \in [0,1]}\left| b_i^*(X_t)\right|\right] \leq C_1 $, the above inequalities and the Cauchy-Schwarz inequality yield 
\begin{equation*}
\E\left|\w{F}_i(X)-\bar{F}_i(X) \right| 
\leq C_{\sigma_0^*}{\sigma}_0^{-2} 
\left(b_{\rm max}\E\left\|\w{b}_i-b^*_i\right\|_{n}+\E\left\|\w{\sigma}^2-\sigma^2\right\|_{n}\right). \end{equation*}
Therefore, we have, 
\begin{equation}
\label{eq:eqDecompExcessRisk2}
\sum_{i=1}^K \E \left[ \left|\phi_i(\boldsymbol{\w{F}}(X))-\phi_i(\bar{\bf F}(X))\right|\right] \leq C_{K, \sigma_0^*} {\sigma}_0^{-2} 
\sum_{i=1}^K\left(b_{\rm max}\E\left\|\w{b}_i-b^*_i\right\|_{n}+\E\left\|\w{\sigma}^2-\sigma^2\right\|_{n}\right).
\end{equation}
Finally, the last term is bounded as follows. We first observe that for all $i \in  \mathcal{Y}$
\begin{multline*}
\E\left[\left|\bar{F}_i(X)-{F}^*_i(X)\right|^2\right] \leq 3\E\int_{0}^{1}{\left(\frac{b^{*}_{i}\left(X_{\xi(s)}\right)}{\sigma^{*2}\left(X_{\xi(s)}\right)}-\frac{b^{*}_{i}\left(X_s\right)}{\sigma^{*2}(X_s)}\right)^2b^{*2}_{Y}(X_s)ds}\\
+3\E\int_{0}^{1}{\left(\frac{b^{*2}_{i}\left(X_{\xi(s)}\right)}{\sigma^{*2}\left(X_{\xi(s)}\right)}-\frac{b^{*2}_{i}\left(X_s\right)}{\sigma^{*2}(X_s)}\right)^2ds}
+3\E\int_{0}^{1}{\left(\frac{b^{*}_{i}\left(X_{\xi(s)}\right)}{\sigma^{*2}\left(X_{\xi(s)}\right)}-\frac{b^{*}_{i}\left(X_s\right)}{\sigma^{*2}(X_s)}\right)^2\sigma^{*2}(X_s)ds}.   
\end{multline*}
Using again that $\sigma^*(\cdot) \geq \sigma^*_0$, and
$ \E\left[\sup_{t \in [0,1]}\left|b_i^*(X_t\right)|^q\right] \leq C$ for $q\geq 1$ (by Assumption \ref{ass:RegEll}), the Cauchy-Schwarz inequality implies
\begin{multline*}
\E\left[\left|\bar{F}_i(X)-{F}^*_i(X)\right|^2\right] \leq C_{\sigma_0^*} 
 \left(\int_{0}^{1}{\E\left[\left|b^{*}_{i}\left(X_{\xi(s)}\right)-b^{*}_{i}(X_s)\right|^{2}\right]ds} \right.\\
 +\left. \int_{0}^{1}{\sqrt{\E\left[\left|b^{*}_{i}\left(X_{\xi(s)}\right)-b^{*}_{i}(X_s)\right|^{4}\right]}ds} + \int_{0}^{1}{\sqrt{\E\left[\left|\sigma^ {*2}\left(X_{\xi(s)}\right)-\sigma^{*2}(X_s)\right|^{4}\right]}ds}\right). 
\end{multline*}
Finally, since the functions $b_i^*$, and $\sigma^{*}$ are Lipschitz, we deduce from Lemma~\ref{lem:discrete}
that
\begin{equation*}
\E\left[\left|\bar{F}_i(X)-{F}^*_i(X)\right|^2\right] \leq C_{\sigma_0^*} \Delta,   
\end{equation*}
which implies together with the fact that the sofmax function is $1$-Lipschitz and the Jensen inequality that
\begin{equation}
\label{eq:eqDecompExcessRisk3}
\sum_{i=1}^K \E \left[ \left|\phi_i(\bar{\bf F}(X))-\phi_i({\bf F}^*(X))\right|\right] \leq C_{K,\sigma_0^*} \sqrt{\Delta}.
\end{equation}
In view of Equation~\ref{eq:eqDecompExcessRisk}, the combination of Equations~\eqref{eq:eqDecompExcessRisk1}, and~\eqref{eq:eqDecompExcessRisk2}, and~\eqref{eq:eqDecompExcessRisk3} yields the desired result.
\end{proof}

\begin{proof}[\textbf{Proof of Proposition~\ref{prop:approx}}]
We consider $h$ a $L$-Lipschitz function. We define the spline-approximation $\tilde{h}$ of $h$ by
\begin{equation*}
\tilde{h}(x) := \sum_{\ell=-M}^{\textcolor{black}{K^{*}}-1} h(u_\ell) B_{\ell}(x), \;\; \forall x\in \R.    
\end{equation*}
First, we note that $\tilde{h} \in \textcolor{black}{\mathcal{S}_{K^{*},M}}$. Indeed, since $h$ is $L$-Lipschitz, there exists $C_L > 0$ such that \textcolor{black}{for $A$ large enough},
\begin{equation*}
\textcolor{black}{\left|h(x)\right| \leq C_L(1 + |x|) \leq C A, \;\; \forall x \in [-A, A].} 
\end{equation*}
Therefore, for $N$ large enough, we have
\begin{equation*}
\left|h(x)\right| \leq \textcolor{black}{A\log^{1/2}(N)}.  
\end{equation*}
Then, we deduce
\begin{equation*}
\textcolor{black}{\sum_{\ell=-M}^{K^{*}-1} h^2(u_\ell) \leq (K^{*}+M) A^2\log(N).} 
\end{equation*}
For \textcolor{black}{$x \in [-A, A]$}, there exists $0 \leq \ell_0 \leq \textcolor{black}{K^{*}}-1$ such that $x \in [u_{\ell_0}, u_{\ell_0+1})$.
We use the following property of the $B$-spline basis 
\begin{equation*}
B_\ell(x) = 0, \;\;  {\rm if} \;\; x \notin [u_\ell,u_{\ell+M+1}), \;\; \ell =-M, \ldots, K_N+M.    
\end{equation*}
Hence, for $x \in [u_{\ell_0}, u_{\ell_0+1})$, we have $B_\ell(x) =0$ for  $\ell \leq \ell_0-M-1$, and $\ell \geq  \ell_0+M$.
Thus,
\begin{eqnarray*}
\left|\tilde{h}(x)-h(x)\right| & \leq & 
\sum_{\ell=-M}^{\textcolor{black}{K^{*}}-1} \left|h(u_\ell) -h(x)\right| B_\ell(x)\\
& = & \sum_{\ell=\ell_0-M}^{\ell_0} \left|h(u_\ell) -h(x)\right| B_l(x)\\
& \leq & \max_{\ell =\ell_0-M, \ldots, \ell_0} \left|h(u_\ell) -h(x)\right|\\
&\leq & L(u_{\ell_0+1} -u_{\ell_0-M} )\leq \textcolor{black}{\frac{2L(M+1)A}{K^{*}}}, 
\end{eqnarray*}
which concludes the proof.
\end{proof}

\begin{proof}[\textbf{Proof of Theorem~\ref{thm:cveDriftSig}}]

The proof is divided in two parts. The first part establishes the rates of convergence of the drift estimators, and the second part is devoted to the study of the rates of convergence of the diffusion coefficient estimator.

\paragraph*{Rates of convergence for drift estimator.}
 
Let $i \in\{1, \ldots, K\}$. 
We introduce, \textcolor{black}{on the random event $\{N_i > 1\}$, and with $A_{N_i} = \log(N_i)$}, the function, $$\bar{b}_{i} := b_i^*\textcolor{black}{\one_{(-A_{N_i}, A_{N_i})}}.$$
We recall that \textcolor{black}{$N_i = \sum_{j=1}^N \one_{\{Y_j=i\}}$ is the random number of paths in the class number $i$}.
For a function $h$, we introduce the empirical norm of class $i$ \textcolor{black}{on the event $\{N_i > 1\}$} as
\begin{equation*}
 \left\|h\right\|^2_{n,N_i}  := \dfrac{1}{nN_i}\sum_{j \in \mathcal{I}_j}\sum_{k = 0}^{n-1} h^2(X^j_{k\Delta}) 
\end{equation*}
We first observe that 
\begin{equation*}
\E\left[\left\|\w{b}_i-b_i^*\right\|_n\right]  =   \E\left[\left\|\w{b}_i-b_i^*\right\|_n \textcolor{black}{\one_{\{N_i > 1\}}}\right] + \E\left[\left\|\w{b}_i-b_i^*\right\|_n\textcolor{black}{\one_{\{N_i \leq 1\}}}\right]. 
\end{equation*}
Let us work at first on the event \textcolor{black}{$\{N_i > 1\}$}. For all $i\in\cY$, we define the following conditional expectation 
\begin{equation*}
    \E_i[.] = \E[.|\one_{\{Y_1=i\}}, \ldots, \one_{\{Y_N=i\}}].
\end{equation*}
We apply Proposition~\ref{prop:approx}, and Proposition~3.2 
of~\cite{denis2020ridge} on the event \textcolor{black}{$\{N_i > 1\}$} and deduce that
\begin{equation}
\label{eq:eqb1}
\E_i\left[\left\|\w{b}_i-\bar{b}_i\right\|^2_{n,N_i}\right] %&\leq& \left\|\tilde{b}_i-\bar{b}_i\right\|_n+C\left(\sqrt{\dfrac{(K_N+M)\log^3(N)}{N}}+\Delta\right)\\
\leq C\left(\textcolor{black}{\dfrac{A^{2}_{N_i}}{K^2_{N_i}}} + \sqrt{\dfrac{\textcolor{black}{K_{N_i}A^{2}_{N_i}\log(N)}}{N_i}}+\Delta\right).
\end{equation}
Now, for all $i\in\cY$, let us write
\begin{equation}
\label{eq:eqb2}
\E_i\left[\left\|\w{b}_i-\bar{b}_i\right\|^2_{n,i}\right]
= \E_i\left[\left\|\w{b}_i-\bar{b}_i\right\|^2_{n,i}\right]-2
\E_i\left[\left\|\w{b}_i-\bar{b}_i\right\|^2_{n,N_i}\right]+2\E_i\left[\left\|\w{b}_i-\bar{b}_i\right\|^2_{n,N_i}\right].
\end{equation}
For $h \in \textcolor{black}{\mathcal{S}_{K_{N_i},M}}$, we denote by $\bar{h}$ its thresholded counterpart
\begin{equation*}
\bar{h}(\cdot) :=h(\cdot) \one_{\{|h(\cdot)|\leq \textcolor{black}{A_{N_i}\log^{1/2}(N)}\}} + {\rm sgn}(h(\cdot))\textcolor{black}{A_{N_i}\log^{1/2}(N)}  \one_{\{|h(\cdot)| > \textcolor{black}{A_{N_i}\log^{1/2}(N)}\}}.
\end{equation*}
We also denote \textcolor{black}{$\mathcal{H}_{K_{N_i},M}: = \{\bar{h}, ~ h \in \mathcal{S}_{K_{N_i},M}\}$}.
Then, \textcolor{black}{on the event $\{N_i > 1\}$}, we have that
\begin{eqnarray*}
\E_i\left[\left\|\w{b}_i-\bar{b}_i\right\|^2_{n,i}\right]-2
\E_i\left[\left\|\w{b}_i-\bar{b}_i\right\|^2_{n,N_i}\right] &\leq & \E_i\left[\sup_{\bar{h} \in \textcolor{black}{\mathcal{H}_{K_{N_i},M}}} \left\|\bar{h}-\bar{b}_i\right\|^2_{n,i}-2
\left\|\bar{h}-\bar{b}_i\right\|^2_{n,N_i}\right]\\
&\leq &\E_i\left[\sup_{g \in \textcolor{black}{\mathcal{G}_{K_{N_i},M}}}\E_{X|Y=i}\left[g(\bar{X})-\frac{2}{N_i}\sum_{i \in \mathcal{I}} g(\bar{X}^i)\right]\right],
\end{eqnarray*}
with $\textcolor{black}{\mathcal{G}_{K_{N_i},M}} = \{(x_1, \ldots,x_n) \mapsto \frac{1}{n}\sum_{k=1}^n\left|\bar{h}(x_k) -\bar{b}_i(x_k)\right|^2, \bar{h} \in \textcolor{black}{\mathcal{H}_{K_{N_i},M}}\}$. 
For each $g \in \textcolor{black}{\mathcal{G}_{K_{N_i},M}}$ and $x \in \mathbb{R}$, we have \textcolor{black}{on the event $\{N_i > 1\}$,} 
\begin{equation*}
0 \leq g(x) \leq 4 \textcolor{black}{A^{2}_{N_i}\log(N)}.     
\end{equation*}
Furthermore, we have that \citep[see][]{denis2020ridge}
\begin{equation*}
\textcolor{black}{\mathcal{N}_{\infty}\left(\varepsilon, \mathcal{G}_{K_{N_i},M}\right) \leq \left(\dfrac{12(K_{N_i}+M)A^{2}_{N_i}\log(N)}{\varepsilon}\right)^{K_{N_i}+M} \leq 
\left(\dfrac{12(K_N+M)\log^{3}(N)}{\varepsilon}\right)^{K_N+M}}.
\end{equation*}
Therefore, we deduce from Lemma~A.2 in~\cite{denis2020ridge} with $\varepsilon= \dfrac{12(K_N+M)\log^3(N)}{N_i}$, Equation~\eqref{eq:eqb1}, and Equation~\eqref{eq:eqb2},
that on the event \textcolor{black}{$\{N_i>1\}$ with $A_{N_i} = \log(N_i)$}
\begin{equation}
\label{eq:eqb3}
\E_i\left[\left\|\widehat{b}_i-\bar{b}_i\right\|^2_{n,i}\right] \leq 
 C\left(\textcolor{black}{\dfrac{\log^2(N_i)}{K^{2}_{N_i}}} + \sqrt{\dfrac{\textcolor{black}{K_{N_i}\log^{2}(N_i)\log(N)}}{N_i}}+    \dfrac{\textcolor{black}{\log^{2}(N_i)\log^2(N)K_{N_i}}}{N_i}+\Delta\right).
\end{equation}
\textcolor{black}{Thus, choosing $K_{N_i} \propto (N_i\log(N_i))^{1/5}$ and for $\log(N_i) \leq \log(N)$, we obtain from Equation~\eqref{eq:eqb3} that}
\begin{equation}
    \label{eq:eqb3-bis}
    \textcolor{black}{\E_i\left[\left\|\widehat{b}_i-\bar{b}_i\right\|^2_{n,i}\right] \leq C\left(\frac{\log^{8/5}(N)}{N^{2/5}_{i}} + \frac{\log^{21/5}(N)}{N^{4/5}_{i}} + \Delta\right) }
\end{equation}
Using Jensen's inequality, we have
\begin{eqnarray*}
\E\left[\left\|\w{b}_i-b_i^{*}\right\|_{n,i}\textcolor{black}{\one_{\{N_i > 1\}}}\right] \leq \sqrt{\E\left[\left\|\w{b}_i-{\bar{b}}_i\right\|^2_{n,i}\textcolor{black}{\one_{\{N_i > 1\}}}\right]
+ \E\left[\left\|\bar{b}_i-b_i^*\right\|^2_{n,i}\textcolor{black}{\one_{\{N_i > 1\}}}\right]}.
\end{eqnarray*}

%%%
Finally, let us study then the error $\left\|\bar{b}_i-b_i^*\right\|^2_{n,i}$. \textcolor{black}{On the event $\{N_i > 1\}$}, 
we observe with the Cauchy-Schwarz inequality
\begin{eqnarray*}
\left\|\bar{b}_i-b_i^*\right\|^2_{n,i} &=& \E_{X|Y=i}\left[\frac{1}{n}\sum_{k=1}^n \left(b_i^{*}(X_{k\Delta})\right)^2\one_{\{|X_{k\Delta}|> \textcolor{black}{A_{N_i}}\}}|~\textcolor{black}{\one_{Y_1=i}, \ldots, \one_{Y_N=i}}\right] \\
&\leq& 
C\sqrt{\sup_{t \in [0,1]}\mathbb{P}_{X|Y=i}\left(|X_t| \geq \textcolor{black}{A_{N_i}} |~\textcolor{black}{\one_{Y_1=i}, \ldots, \one_{Y_N=i}}\right)},
\end{eqnarray*}
since $\sup_{t \in [0,1]}\E\left[b_i^*(X_t)^4\right] \leq C$.
\textcolor{black}{For $A_{N_i} = \log(N_i)$ and} from Lemma~\ref{lem:controleSortiCompact}, we obtain \textcolor{black}{on the event $\{N_i > 1\}$}
\begin{equation*}
\left\|\bar{b}_i-b_i^*\right\|^2_{n,i} \leq C\exp\left(-\frac{C_2}{2}\textcolor{black}{\log^2(N_i)}\right),   
\end{equation*}
which, for \textcolor{black}{$N_i$ a.s.} large enough yields
\begin{equation*}
\left\|\bar{b}_i-b_i^*\right\|_{n,i} \leq \textcolor{black}{CN^{-1/2}_{i}}.  
\end{equation*}
This result leads us to obtain, from Equation \eqref{eq:eqb3-bis}, that
\begin{equation*}
\E\left[\left\|\w{b}_i-b_i^{*}\right\|_{n,i}\textcolor{black}{\one_{\{N_i > 1\}}}\right] \leq
C\left(\textcolor{black}{\E\left[\left(\frac{\log^{4/5}(N)}{N^{1/5}_{i}} + \frac{\log^{21/10}(N)}{N^{2/5}_{i}}\right)\one_{\{N_i>1\}}\right]} 
+\sqrt{\Delta}\right).
\end{equation*}
Using Jensen's inequality, we obtain
\begin{multline*}
\textcolor{black}{\E\left[\left\|\w{b}_i-b_i^{*}\right\|_{n,i}\one_{\{N_i > 1\}}\right] \leq
C\left(\log^{4/5}(N)\left(\E\left[\frac{\one_{N_i > 1}}{N_i}\right]\right)^{1/5} + \log^{21/10}(N)\left(\E\left[\frac{\one_{N_i > 1}}{N_i}\right]\right)^{2/5}  +\sqrt{\Delta}\right)}.
\end{multline*}
%%% Ni=0
To finish the proof, since for all $i\in\cY, \ N_i\sim\mathcal{B}(N,\mfp^{*}_i)$ we use Lemma~4.1 in~\citep{gyorfi2006distribution} to deduce that
\begin{equation*}
    \textcolor{black}{\E\left[\frac{\one_{N_i>1}}{N_i}\right] \leq} \E\left[\frac{\one_{N_i>0}}{N_i}\right] \leq \frac{2}{\mfp^{*}_i N} \leq \frac{2}{\mfp^{*}_0 N}
\end{equation*}
and finally, there exists a constant $C>0$ such that
\begin{equation}
\label{eq:eqb4}
\textcolor{black}{\E\left[\left\|\w{b}_i-b_i^{*}\right\|_{n,i}\one_{\{N_i > 1\}}\right]  \leq C\left( \left(\dfrac{\log^4(N)}{N\mfp^{*}_0} \right)^{1/5} + \sqrt{\Delta}\right)}.   
\end{equation}
To conclude the proof for the rates of convergence of the drift coefficient, we observe that since \textcolor{black}{$\widehat{b}_i$} is bounded by $\log^{3/2}(N)$ and $\sup_{t \in[0,1]} \E\left[b_i^*(X_t)^2\right] <+\infty$, we have \textcolor{black}{for $N$ large enough},
\begin{equation}
\label{eq:eqb5}
\E\left[\left\|\w{b}_i-b_i^*\right\|_{n,i}\textcolor{black}{\one_{\{N_i \leq 1\}}}\right] \leq \textcolor{black}{2\log^{3/2}(N)\left(\mathbb{P}\left(N_i=0\right)+ \mathbb{P}\left(N_i=1\right)\right)}.
\end{equation}
Since $N_i$ is distributed according to a Binomial distribution with parameters $(N,\mfp^{*}_i)$. We deduce that
\begin{equation}\label{eq:eqb6}
\P\left(N_i=0\right) = \exp\left(N\log(1-\mfp^{*}_i)\right), ~~ \textcolor{black}{\P\left(N_i=1\right) = \frac{\mfp^{*}_{i}}{1-\mfp^{*}_{i}}\exp\left(N\log(1-\mfp^{*}_i)\right)}.    
\end{equation}
\textcolor{black}{
Hence, gathering Equation~\eqref{eq:eqb4}, Equation~\eqref{eq:eqb5} and Equation~\eqref{eq:eqb6}, and choosing $\Delta= O(1/N)$, we obtain for each label $i \in \mathcal{Y}$,
$$
\E\left[\left\|\w{b}_i-b_i^*\right\|_{n,i}\right] = \E\left[\left\|\w{b}_i-b_i^*\right\|_{n,i}\one_{N_i > 1}\right] + \E\left[\left\|\w{b}_i-b_i^*\right\|_{n,i}\one_{N_i \leq 1}\right] = \mathrm{O}\left(\left(\frac{\log^{4}(N)}{N\mfp^{*}_{0}}\right)^{1/5}\right).
$$}

\paragraph*{Diffusion coefficient: rates of convergence.}

We estimate the square $\sigma^{*2}$ of the diffusion coefficient as solution of the following regression model
\begin{equation}
\label{eq:RegressionModel-sigma}
    \frac{(X^{j}_{(k+1)\Delta}-X^{j}_{k\Delta})^2}{\Delta}=\sigma^{*2}(X^{j}_{k\Delta})+\zeta^{j}_{k\Delta}+R^{j}_{k\Delta}
\end{equation}
where $\zeta^{j}_{k\Delta}:=\zeta^{j,1}_{k\Delta}+\zeta^{j,2}_{k\Delta}+\zeta^{j,3}_{k\Delta}$ with
\begin{align*}
   \zeta^{j,1}_{k\Delta}&:=\frac{1}{\Delta}\left[\left(\int_{k\Delta}^{(k+1)\Delta}{\sigma^{*}\left(X^{j}_{s}\right)dW^{j}_{s}}\right)^2-\int_{k\Delta}^{(k+1)\Delta}{\sigma^{*2}\left(X^{j}_{s}\right)ds}\right]\\
   \zeta^{j,2}_{k\Delta}&:=\frac{2}{\Delta}\int_{k\Delta}^{(k+1)\Delta}{((k+1)\Delta-s)\sigma^{*\prime}\left(X^{j}_{s}\right)\sigma^{*2}\left(X^{j}_{s}\right)dW^{j}_{s}}\\
   \zeta^{j,3}_{k\Delta}&:=2b^{*}_{Y}\left(X^{j}_{k\Delta}\right)\int_{k\Delta}^{(k+1)\Delta}{\sigma^{*}\left(X^{j}_{s}\right)dW^{j}_{s}},
\end{align*}

and $R^{j}_{k\Delta}:=R^{j,1}_{k\Delta}+R^{j,2}_{k\Delta}+R^{j,3}_{k\Delta}$ with,
\begin{equation}
    \label{eq:Residuals-R1R2}
    R^{j,1}_{k\Delta}:=\frac{1}{\Delta}\left(\int_{k\Delta}^{(k+1)\Delta}{b^{*}_{Y}\left(X^{j}_{s}\right)ds}\right)^2, \ \ \ R^{j,2}_{k\Delta}:=\frac{1}{\Delta}\int_{k\Delta}^{(k+1)\Delta}{((k+1)\Delta-s)\phi_Y\left(X^{j}_{s}\right)ds}
\end{equation}
\begin{equation}
\label{eq:Residuals-R3}
    R^{j,3}_{k\Delta}:=\frac{2}{\Delta}\left(\int_{k\Delta}^{(k+1)\Delta}{\left(b^{*}_{Y}\left(X^{j}_{s}\right)-b^{*}_{Y}\left(X_{k\Delta}\right)\right)ds}\right)\left(\int_{k\Delta}^{(k+1)\Delta}{\sigma^{*}\left(X^{j}_{s}\right)dW^{j}_{s}}\right)
\end{equation}
where $\phi_Y:=b^{*}_{Y}\sigma^{*\prime}\sigma^{*}+\left[\sigma^{*\prime\prime}\sigma^{*}+(\sigma^{*\prime})^2\right]\sigma^{*2}$. We prove in the sequel that $\zeta^{j,1}_{k\Delta}$ is the error term, and all the other terms are negligible residuals. We remind the reader that the estimator $\w{\sigma}^{2}$ of $\sigma^{*2}$ is given in \eqref{eq:boundedsigma}. 
%by
%\begin{equation*}
%\w{\sigma}^{2}(x) = \widetilde{\sigma}^{2}(x) \one_{\{\frac{1}{\log(N)} \leq %\widetilde{\sigma}^{2}(x)\leq \log^{3/2}(N)\}} + \log^{3/2}(N)  %\one_{\{\widetilde{\sigma}^{2}(x) > \log^{3/2}(N)\}} + \frac{1}{\log(N)} %\one_{\{\widetilde{\sigma}^{2}(x) \leq \frac{1}{\log(N)} \}}.
%\end{equation*}
%
%where $\widetilde{\sigma}^{2}$ satisfies:
%\begin{equation}
%    \label{eq:least squares contrast - sigma}
%    \widetilde{\sigma}^{2}={\arg}{\underset{h\in\mathcal{S}_{K_N,M}}{\min}%{\gamma_{n,N}(h)}}
%\end{equation}
%
%with $\gamma_{n,N}$ the least square contrast.
We rely on the following result:
\begin{lemme}
\label{lm:lm:loss error-PseudoNorm}
Under Assumption~\ref{subsec:ass} \textcolor{black}{and for $\tilde{A}_N = \log(N)$}, the following holds
\begin{equation*}
    \E\left\|\w{\sigma}^{2}-\sigma^{*2}\right\|^{2}_{n,N}\leq 3\underset{h\in\mathcal{S}_{\textcolor{black}{\tilde{K}_N},M}}{\inf}{\|h-\sigma^{*2}\|^{2}_{n}}+C\left(\sqrt{\frac{\textcolor{black}{\tilde{K}_N}\log^{3}(N)}{Nn}}+\Delta_n^{2}\right)
\end{equation*}
where $C>0$ is a constant depending on $\sigma_1$, and where 
\[\|h\|^2_{n,N}=\frac{1}{nN}\sum_{j=1}^N \sum_{k=0}^{n-1} h^2(X^j_{k\Delta}).\]
\end{lemme}
The empirical error of the estimator $\w{\sigma}^{2}$ is given by
\begin{align*}
    \E\left\|\w{\sigma}^{2}-\sigma^{*2}\right\|^{2}_{n}&=2\E\left\|\w{\sigma}^{2}-\sigma^{*2}\right\|^{2}_{n,N}+\left[\E\left\|\w{\sigma}^{2}-\sigma^{*2}\right\|^{2}_{n}-2\E\left\|\w{\sigma}^{2}-\sigma^{*2}\right\|^{2}_{n,N}\right]
    \end{align*}
    %\\
\textcolor{black}{Since $\tilde{A}_N = \log(N)$}, let us define $\mathcal{H}^{\sigma}$
as the set of functions $\bar{h}$ such that there exists a function $h\in\mathcal{S}_{K_N,M}$ satisfying
\begin{equation*}
\bar{h} = h(x) \one_{\{\frac{1}{\log(N)} \leq h(x)\leq \log^{3/2}(N)\}} + \log^{3/2}(N)  \one_{h(x) > \log^{3/2}(N)\}} + \frac{1}{\log(N)} \one_{\{h(x) \leq \frac{1}{\log(N)} \}}.
\end{equation*}
Using then an $\varepsilon-$net $\mathcal{H}^{\sigma,\varepsilon}$ of $\mathcal{H}^{\sigma}$ with $\varepsilon=\frac{12(\textcolor{black}{\tilde{K}_N}+M)\log^{3}(N)}{N}$, we finally obtain (see \cite{denis2020ridge}, Lemma A.2)
\begin{align*}
    \E\left\|\w{\sigma}^{2}-\sigma^{*2}\right\|^{2}_{n}-2\E\left\|\w{\sigma}^{2}-\sigma^{*2}\right\|^{2}_{n,N}
       &\leq \E\left[\underset{\bar{h}\in \mathcal{H}^{\sigma}}{\sup}{\left\{\E\left\|\bar{h}-\sigma^{*2}\right\|^{2}_{n}-2\E\left\|\bar{h}-\sigma^{*2}\right\|^{2}_{n,N}\right\}}\right]
    \\
    &\leq C\frac{\textcolor{black}{\tilde{K}_N}\log^{4}(N)}{N}.
\end{align*}
Thus, as $\Delta_n=\mathrm{O}(1/N)$, 
 \begin{align*}
     \E\left\|\w{\sigma}^{2}-\sigma^{*2}\right\|^{2}_{n}\leq&~ 3\underset{h\in\mathcal{S}_{\textcolor{black}{\tilde{K}_N},M}}{\inf}{\|h-\sigma^{*2}\|^{2}_{n}}+C\left(\frac{\sqrt{\textcolor{black}{\tilde{K}_N}\log^{3}(N)}}{N}+\frac{\textcolor{black}{\tilde{K}_N}\log^{4}(N)}{N}+\frac{1}{N^2}\right)\\
     \leq&~3\underset{h\in\mathcal{S}_{\textcolor{black}{\tilde{K}_N}M}}{\inf}{\|h-\sigma^{*2}\|^{2}_{n}}+C\frac{\textcolor{black}{\tilde{K}_N}\log^{4}(N)}{N},
 \end{align*}for $N$ large enough.
According to Proposition~\ref{prop:approx}, the bias term satisfies
$$\underset{h\in\mathcal{S}_{\textcolor{black}{\tilde{K}_N},M}}{\inf}{\|h-\sigma^{*2}\|^{2}_{n}}\leq C\frac{\log^{2}(N)}{\textcolor{black}{\tilde{K}^{2}_N}}.$$
Taking $\textcolor{black}{\tilde{K}_N}=(N\log(N))^{1/5}$  leads to
\begin{equation*}
    %\E\left\|\w{\sigma}^{2}-\sigma^{*2}\right\|^{2}_{n}\leq C^{2}_2\left(\frac{\log^{4}(N)}{N}\right)^{2/5}, \ \ \mathrm{and}  
    \E\left\|\w{\sigma}^{2}-\sigma^{*2}\right\|_{n}\leq C_2\left(\frac{\log^{4}(N)}{N}\right)^{1/5}.
\end{equation*}This concludes the proof of Theorem~\ref{thm:cveDriftSig}.
%where $C_2>0$ is a constant.
\end{proof}

\begin{proof}[\textbf{Proof of Lemma~\ref{lm:lm:loss error-PseudoNorm}~}]
Denote by
\begin{equation*}
\gamma_{N,n}(h) = \dfrac{1}{nN}\sum_{j=1}^N \sum_{k=0}^{n-1} \left(U_{k\Delta}^j-h(X^j_{k\Delta})\right)^2, 
\end{equation*}
the least square contrast appearing in \eqref{eq:least squares contrast - sigma}. For all $h\in\mathcal{S}_{K_N,M}$, we deduce that
\begin{equation}
\label{eq:property-least squares contrast-sigma}
    \gamma_{n,N}(\w{\sigma}^{2})-\gamma_{n,N}(\sigma^{*2})\leq\gamma_{n,N}(h)-\gamma_{n,N}(\sigma^{*2}).
\end{equation}
Using \eqref{eq:RegressionModel-sigma}, we have for all $h\in\mathcal{S}_{\textcolor{black}{\tilde{K}_N},M}$,
\begin{equation}
\label{eq:diff-gama}
    \gamma_{n,N}(h)-\gamma_{n,N}(\sigma^{*2})=\left\|h-\sigma^{*2}\right\|^{2}_{n,N}+2\nu_1(\sigma^{*2}-h)+2\nu_2(\sigma^{*2}-h)+2\nu_3(\sigma^{*2}-h)+2\mu(\sigma^{*2}-h)
\end{equation}
where
\begin{equation}
\label{eq:functions nu1 nu2 nu2 and mu}
    \nu_i(h)=\frac{1}{nN}\sum_{j=1}^{N}{\sum_{k=0}^{n-1}{h(X^{j}_{k\Delta})\zeta^{j,i}_{k\Delta}}}, \ \ i\in\{1,2,3\}, \ \ \ \mu(h)=\frac{1}{nN}\sum_{j=1}^{N}{\sum_{k=0}^{n-1}{h(X^{j}_{k\Delta})R^{j}_{k\Delta}}},
\end{equation}
we derive from Equations~\eqref{eq:property-least squares contrast-sigma}~and~\eqref{eq:diff-gama}~ that for all $h\in\mathcal{S}_{\textcolor{black}{\tilde{K}_N},M}$,
\begin{equation}
    \label{eq:first risk bound}
    \E\left\|\w{\sigma}^{2}-\sigma^{*2}\right\|^{2}_{n,N}\leq\underset{h\in\mathcal{S}_{\textcolor{black}{\tilde{K}_N},M}}{\inf}{\|h-\sigma^{*2}\|^{2}_{n}}+2\sum_{i=1}^{3}{\E\left[\nu_i(\w{\sigma}^{2}-h)\right]}+2\E\left[\mu(\w{\sigma}^{2}-h)\right].
\end{equation}
For all $i\in\{1,2,3\}$ and for all $h\in\mathcal{S}_{\textcolor{black}{\tilde{K}_N},M}$, taking the constraints \eqref{eq:SKNM} into account, one has
\begin{equation}
\label{eq:upper-bound Enu_i}
    \E\left[\nu_i\left(\w{\sigma}^{2}-h\right)\right]\leq\sqrt{2(\textcolor{black}{\tilde{K}_N}+M)\log^{3}(N)}\sqrt{\sum_{\ell=-M}^{\textcolor{black}{\tilde{K}_N}-1}{\E\left[\nu^{2}_{i}(B_{\ell,M,\mathbf{u}})\right]}}.
\end{equation}
\begin{enumerate}
    \item Upper bound of $\sum_{\ell=-M}^{\textcolor{black}{\tilde{K}_N}-1}{\E\left[\nu^{2}_{1}(B_{\ell,M,\mathbf{u}})\right]}$. According to Equation~\eqref{eq:functions nu1 nu2 nu2 and mu}, we have
    \begin{equation*}
        \forall \ell\in[\![-M,\textcolor{black}{\tilde{K}_N}-1]\!], \ \nu_1(B_{\ell,M,\mathbf{u}})=\frac{1}{nN}\sum_{j=1}^{N}{\sum_{k=0}^{n-1}{B_{\ell,M,\mathbf{u}}(X^{j}_{k\Delta})\zeta^{j,1}_{k\Delta}}}
    \end{equation*}
    where $\zeta^{j,1}_{k\Delta}=\frac{1}{\Delta}\left[\left(\int_{k\Delta}^{(k+1)\Delta}{\sigma^{*}(X^{j}_{s})dW^{j}_s}\right)^2-\int_{k\Delta}^{(k+1)\Delta}{\sigma^{*2}(X^{j}_{s})ds}\right]$ is a martingale satisfying
    \begin{equation*}
        \E\left[\zeta^{1,1}_{k\Delta}|\mathcal{F}^1_{k\Delta}\right]=0 \ \ \mathrm{and} \ \ \E\left[\left(\zeta^{1,1}_{k\Delta}\right)^2|\mathcal{F}^1_{k\Delta}\right]\leq\frac{1}{\Delta^2}\E\left[\left(\int_{k\Delta}^{(k+1)\Delta}{\sigma^{*2}(X^{1}_{s})ds}\right)^2\right]\leq C\sigma^{*4}_{1}
    \end{equation*}
    with $(\mathcal{F}^1_t)_{t\geq 0}$ the natural filtration associated with the Brownian motion $W^1$. We derive that
    \begin{align*}
        \sum_{\ell=-M}^{\textcolor{black}{\tilde{K}_N}-1}{\E\left[\nu^{2}_{1}(B_{\ell,M,\mathbf{u}})\right]}=&\frac{1}{Nn^2}\sum_{\ell=-M}^{\textcolor{black}{\tilde{K}_N}-1}{\E\left[\left(\sum_{k=0}^{n-1}{B_{\ell,M,\mathbf{u}}(X^{j}_{k\Delta})\zeta^{1,1}_{k\Delta}}\right)^2\right]}\\
        =&\frac{1}{Nn^2}\E\left[\sum_{k=0}^{n-1}{\sum_{\ell=-M}^{\textcolor{black}{\tilde{K}_N}-1}{B^{2}_{\ell,M,\mathbf{u}}(X^{1}_{k\Delta})\left(\zeta^{1,1}_{k\Delta}\right)^2}}\right]\\
        \leq&\frac{C}{Nn}
    \end{align*}
    where $C$ is a constant depending on $\sigma^{*}$, for each $k\in[\![0,n-1]\!]$, $\sum_{\ell=-M}^{\textcolor{black}{\tilde{K}_N}-1}{B^{2}_{\ell,M,\mathbf{u}}(X^{1}_{k\Delta})}\leq 1$ since $\sum_{\ell=-M}^{\textcolor{black}{\tilde{K}_N}-1}{B_{\ell,M,\mathbf{u}}(X^{1}_{k\Delta})}=1$ and $B_{\ell,M,\mathbf{u}}(X^{1}_{k\Delta})\leq 1$ for all $\ell=-M,\cdots,\textcolor{black}{\tilde{K}_N}-1$.
    \item Upper bound of $\sum_{\ell=-M}^{\textcolor{black}{\tilde{K}_N}-1}{\E\left[\nu^{2}_{2}(B_{\ell,M,\mathbf{u}})\right]}$.
    For all $k\in[\![0,n-1]\!]$ and for all $s\in[0,1]$, set $\xi(s)=k\Delta$ if $s\in[k\Delta,(k+1)\Delta)$. We have:
    \begin{multline*}
        \sum_{\ell=-M}^{\textcolor{black}{\tilde{K}_N}-1} \E\left[\nu^{2}_{2}(B_{\ell,M,\mathbf{u}})\right]\\
        \begin{aligned}
        &=\frac{4}{N n^2} \sum_{\ell=-M}^{\textcolor{black}{\tilde{K}_N}-1}\E\left[\left(\sum_{k=0}^{n-1}{\int_{k\Delta}^{(k+1)\Delta}{B_{\ell,M,\mathbf{u}}(X^{1}_{k\Delta})((k+1)\Delta-s)\sigma^{*\prime}(X^{1}_{s})\sigma^{*2}(X^{1}_{s})dW_s}}\right)^2\right]\\
        &=\frac{4}{Nn^2}\sum_{\ell=-M}^{\textcolor{black}{\tilde{K}_N}-1}\E\left[\left(\int_{0}^{1}{B_{\ell,M,\mathbf{u}}(X^{1}_{\xi(s)})(\xi(s)+\Delta-s)\sigma^{*\prime}(X^{1}_{s})\sigma^{*2}(X^{1}_{s})dW_s}\right)^2\right]\\
        &\leq\frac{C}{Nn^2}
    \end{aligned}
    \end{multline*}
     where the constant $C>0$ depends on the diffusion coefficient.
    \item Upper bound of $\sum_{\ell=-M}^{\textcolor{black}{\tilde{K}_N}-1}{\E\left[\nu^{2}_{3}(B_{\ell,M,\mathbf{u}})\right]}$. We have:
    \begin{align*}
        \sum_{\ell=-M}^{\textcolor{black}{\tilde{K}_N}-1}{\E\left[\nu^{2}_{3}(B_{\ell,M,\mathbf{u}})\right]}&=\frac{4}{Nn^2}\sum_{\ell=-M}^{\textcolor{black}{\tilde{K}_N}-1}{\E\left[\left(\sum_{k=0}^{n-1}{\int_{k\Delta}^{(k+1)\Delta}{B_{\ell,M,\mathbf{u}}(X^{1}_{k\Delta})b^{*}_{Y}(X^{1}_{k\Delta})\sigma^{*}(X^{1}_{s})dW_s}}\right)^2\right]}\\
        &=\frac{4}{Nn^2}\sum_{\ell = -M}^{\textcolor{black}{\tilde{K}_N}-1}{\E\left[\left(\int_{0}^{1}{B_{\ell,M,\mathbf{u}}(X^{1}_{\eta(s)})b^{*}_{Y}(X^{1}_{\eta(s)})\sigma^{*}(X^{1}_{s})dW_s}\right)^2\right]}\\
        &\leq\frac{4}{Nn^2}\E\left[\int_{0}^{1}{\sum_{\ell=-M}^{\textcolor{black}{\tilde{K}_N}-1}{B^{2}_{\ell,M,\mathbf{u}}(X^{1}_{\eta(s)})b^{*2}_{Y}(X^{1}_{\eta(s)})\sigma^{*2}(X^{1}_{s})ds}}\right].
    \end{align*}
    Since for all $x\in\mathbb{R}, \ b^{*2}_{Y}(x)\leq C_0(1+x^2), \ \sigma^{*2}(x)\leq\sigma^{*2}_{1}$ and $\sup_{t\in[0,1]}\E\left(|X_t|^2\right)<\infty$, there exists a constant $C>0$ depending on the upper bound $\sigma^{*}_1$ of the diffusion coefficient such that
    \begin{align*}
        \sum_{\ell=-M}^{\textcolor{black}{\tilde{K}_N}-1}{\E\left[\nu^{2}_{3}(B_{\ell,M,\mathbf{u}})\right]}\leq \frac{C}{Nn^2}.
    \end{align*}
\end{enumerate}
We finally deduce from Equations~\eqref{eq:first risk bound}~and~\eqref{eq:upper-bound Enu_i}~ that for all $h\in\mathcal{S}_{\textcolor{black}{\tilde{K}_N},M}$,
\begin{equation}
    \label{eq:second risk bound}
    \E\left\|\w{\sigma}^{2}-\sigma^{*2}\right\|^{2}_{n,N}\leq\underset{h\in\mathcal{S}_{\textcolor{black}{\tilde{K}_N},M}}{\inf}{\|h-\sigma^{*2}\|^{2}_{n}}+C\sqrt{\frac{(\textcolor{black}{\tilde{K}_N}+M)\log^{3}(N)}{Nn}}+2\E\left[\mu(\w{\sigma}^{2}-h)\right].
\end{equation}
It remains to obtain an upper bound of the term $\mu(\w{\sigma}^{2}-h)$. {Notice that for $a>0$, $x$ and $y\in \R$, 
\[2xy=2\frac{x}{\sqrt{a}} \times \sqrt{a}y\leq \frac{x^2}{a}+ a y^2.\]
}Then, for all $h\in\mathcal{S}_{\textcolor{black}{\tilde{K}_N},M}$ and $a>0$,
\begin{align*}
    2\mu\left(\w{\sigma}^{2}-h\right)&\leq\frac{2}{a}\left\|\w{\sigma}^{2}-\sigma^{*2}\right\|^{2}_{n,N}+\frac{2}{a}\left\|h-\sigma^{*2}\right\|^{2}_{n,N}+\frac{a}{Nn}\sum_{j=1}^{N}{\sum_{k=0}^{n-1}{\left(R^{j}_{k\Delta}\right)^2}}.
\end{align*}
We set $a=4$ and from Equation \eqref{eq:second risk bound} we deduce that,
\begin{equation}
\label{eq:third risk bound}
    \E\left\|\w{\sigma}^{2}-\sigma^{*2}\right\|^{2}_{n,N}\leq 3\underset{h\in\mathcal{S}_{\textcolor{black}{\tilde{K}_N},M}}{\inf}{\|h-\sigma^{*2}\|^{2}_{n}}+C\sqrt{\frac{(\textcolor{black}{\tilde{K}_N}+M)\log^{3}(N)}{Nn}}+\frac{4}{Nn}\sum_{j=1}^{N}{\sum_{k=0}^{n-1}{\E\left[\left(R^{j}_{k\Delta}\right)^2\right]}}.
\end{equation}
We have 
\begin{align*}
    \E\left[\left(R^{j}_{k\Delta}\right)^2\right]\leq 3\left(\E\left[\left(R^{j,1}_{k\Delta}\right)^2\right]+\E\left[\left(R^{j,2}_{k\Delta}\right)^2\right]+\E\left[\left(R^{j,3}_{k\Delta}\right)^2\right]\right)
\end{align*}
where for all $j\in[\![1,N]\!]$ and $k\in[\![0,n-1]\!]$, $R^{j,1}_{k\Delta}, R^{j,2}_{k\Delta}$ and $R^{j,3}_{k\Delta}$ are given in Equations~\eqref{eq:Residuals-R1R2} and \eqref{eq:Residuals-R3}. There exist constants $C_1,C_2,C_3>0$ such that

\begin{align*} 
\E\left[\left(R^{j,1}_{k\Delta}\right)^2\right] &\leq \E\left[\left(\int_{k\Delta}^{(k+1)\Delta}{b^{*2}_{Y}\left(X^{j}_{k\Delta}\right)ds}\right)^2\right]\leq\Delta\E\left[\int_{k\Delta}^{(k+1)\Delta}{b^{*4}_{Y}\left(X^{j}_{k\Delta}\right)ds}\right]\leq C_1\Delta^2 \\ 
\E\left[\left(R^{j,2}_{k\Delta}\right)^2\right] &\leq \frac{1}{\Delta^2}\int_{k\Delta}^{(k+1)\Delta}{((k+1)\Delta-s)^2ds}\int_{k\Delta}^{(k+1)\Delta}{\E\left[\phi^{2}_{Y}\left(X^{j}_s\right)\right]ds}\leq C_2\Delta^2 \\
\E\left[\left(R^{j,3}_{k\Delta}\right)^2\right] &\leq \frac{4}{\Delta^2}\E\left[\Delta\int_{k\Delta}^{(k+1)\Delta}{L^{2}_{0}\left|X^{j}_s-X^{j}_{k\Delta}\right|^2ds}\left(\int_{k\Delta}^{(k+1)\Delta}{\sigma^{*}(X^{j}_s)dW_s}\right)^2\right]\leq C_{3}\Delta^2.
\end{align*}
%where $\Delta=\Delta_n\tvc{=1/n}$ is the time-step depending on $n$. 
We deduce from Equation~\eqref{eq:third risk bound}~ that there exists a constant $C>0$ depending on $\sigma^{*}_{1}$ and $M$ such that,
\begin{equation*}
    \E\left\|\w{\sigma}^{2}-\sigma^{*2}\right\|^{2}_{n,N}\leq 3\underset{h\in\mathcal{S}_{\textcolor{black}{\tilde{K}_N},M}}{\inf}{\|h-\sigma^{*2}\|^{2}_{n}}+C\left(\sqrt{\frac{\textcolor{black}{\tilde{K}_N}\log^{3}(N)}{Nn}}+\Delta_n^2\right).
\end{equation*}This is the announced result.
\end{proof}

\begin{proof}[\textbf{Proof of Theorem~\ref{thm:consistency}}]
\textcolor{black}{$i \in \mathcal{Y}$, define once again $\E_i = \E[. | \one_{Y_1}, \ldots, \one_{Y_N = i}]$. On the event $\{N_i > 1\}$}, we have \textcolor{black}{for all $A_i > 0$}
\begin{eqnarray}
\E_{\textcolor{black}{i}}\left[\left\|\w{b}_i-b_i^*\right\|^2_{n}\right]
 &=&  \E_{\textcolor{black}{i}}\left[\frac{1}{n}\sum_{k=0}^{n-1} \left(\w{b}_i(X_{k\Delta})-b_i^*(X_{k\Delta})\right)^2\right]\nonumber\\
 &=&  \E_{\textcolor{black}{i}}\left[\frac{1}{n}\sum_{k=0}^{n-1} \left(\w{b}_i(X_{k\Delta})-b_i^*(X_{k\Delta})\right)^2\one_{\{\left|X_{k\Delta}\right|\leq \textcolor{black}{A_{i}}\}}\right]  \nonumber \\
 &&+ \E_{\textcolor{black}{i}}\left[\frac{1}{n}\sum_{k=0}^{n-1} \left(\w{b}_i(X_{k\Delta})-b_i^*(X_{k\Delta})\right)^2\one_{\{\left|X_{k\Delta}\right|> \textcolor{black}{A_{i}}\}}\right].\label{eq:eqConsist1}
\end{eqnarray}
We bound each term of the {\it r.h.s.} of the above inequality.
From Lemma~\ref{lem:controleSortiCompact}, and Cauchy-Schwarz Inequality, under Assumption~\ref{ass:RegEll}, we have for the second term of \eqref{eq:eqConsist1},
\begin{equation}
\label{eq:eqConsist2}
\E_{\textcolor{black}{i}}\left[\frac{1}{n}\sum_{k=0}^{n-1} \left(\w{b}_i(X_{k\Delta})-b_i^*(X_{k\Delta})\right)^2\one_{\{\left|X_{k\Delta}\right|> \textcolor{black}{A_{i}}\}}\right] \leq C \sqrt{\exp\left(-C \textcolor{black}{A^{2}_{i}}\right)}.    
\end{equation}
For the first term of \eqref{eq:eqConsist1}, we observe that
\begin{eqnarray*}
\E_{\textcolor{black}{i}}\left[\frac{1}{n}\sum_{k=0}^{n-1} \left(\w{b}_i(X_{k\Delta})-b_i^*(X_{k\Delta})\right)^2\one_{\{\left|X_{k\Delta}\right|\leq \textcolor{black}{A_{i}}\}} | \mathcal{D}_N\right] &=& \int_{-\textcolor{black}{A_{i}}}^{\textcolor{black}{A_{i}}} \left(\w{b}_i(x)-b_i^*(x)\right)^2 
\left(\frac{1}{n}\sum_{k=1}^{n-1}p(k \Delta,x)\right) \mathrm{d}x \\
&&+ {\frac{1}{n}}\left(\w{b}_i(0)-b_i^*(0)\right)^2.
\end{eqnarray*}
\textcolor{black}{For $A_{i} = (\log(N_i))^{1/4}$ and from} Lemma~\ref{lem:boundDensity}, we then deduce that
\begin{multline*}
 \E\left[\frac{1}{n}\sum_{k=0}^{n-1} \left(\w{b}_i(X_{k\Delta})-b_i^*(X_{k\Delta})\right)^2\textcolor{black}{\one_{N_i > 1}}\one_{\{\left|X_{k\Delta}\right|\leq \textcolor{black}{A_{i}}\}} | \mathcal{D}_N\right] \\
 \leq   C_1 e^{C_2 \textcolor{black}{\sqrt{\log(N)}}} \E\left[\frac{1}{n}\sum_{k=0}^{n-1} \left(\w{b}_i(X_{k\Delta})-b_i^*(X_{k\Delta})\right)^2\one_{\{\left|X_{k\Delta}\right|\leq \textcolor{black}{A_{i}}\}} | \mathcal{D}_N, Y = i\right] \\ 
 \leq  
C_1e^{C_2 \textcolor{red}{\sqrt{\log(N)}}} \E\left[\left\|\w{b}_i-b_i^*\right\|^2_{n,i}\right].   
\end{multline*}
From the above key equation, Equation~\eqref{eq:eqConsist1}, Equation~\eqref{eq:eqConsist2}, and Theorem~\ref{thm:cveDriftSig}, we deduce,
\begin{equation*}
\E\left[\left\|\w{b}_i-b_i^*\right\|^2_{n}\right] \leq 
C\left(\exp\left(C_2 \textcolor{black}{\sqrt{\log(N)}}\right)\left(\frac{\log(N)^4}{N}\right)^{1/5} + \textcolor{black}{\E\left[\exp(-C_2 \sqrt{\log(N_i)})\one_{N_i > 1}\right]}\right).
\end{equation*}
\textcolor{black}{Since $\exp(-C_2 \sqrt{\log(N_i)})\one_{N_i > 1} \longrightarrow 0 ~ a.s.$ as $N \rightarrow \infty$, and $\exp(-C_2 \sqrt{\log(N_i)}) \leq 1$ for almost all $N_i > 1$, the theorem of dominated convergence implies}
\begin{equation*}
    \textcolor{black}{\E\left[\exp(-C_2 \sqrt{\log(N_i)})\one_{N_i > 1}\right] \longrightarrow 0 ~~ \mathrm{as} ~~ N \rightarrow \infty.}
\end{equation*}
Besides from Theorem~\ref{thm:cveDriftSig}, we also have
\begin{equation*}
\E\left[\left\|\w{\sigma}^2-\sigma^{2*}\right\|_{n}\right] \leq 
\left(\frac{\log(N)^4}{N}\right)^{1/5}. 
\end{equation*}
Therefore, applying Theorem~\ref{thm:comparisonInequality} with ${b}_{{\rm \max}} = \log(N)^{3/2}$, ${\sigma}^{-2}_0 = \log(N)$, we get the desired result. 
\end{proof}

\begin{proof}[\textbf{Proof of Proposition~\ref{prop:eqNorm}~}]
    For all $i\in\cY$, let $\P_i=\P(.|Y=i)$ and denote by $\P_0$ the probability measure under which the diffusion process $X=(X_t)_{t\geq 0}$ is solution of $dX_t=d\widetilde{W}_t$ where $\widetilde{W}$ is a Brownian motion under $\P_0$. We deduce from the Girsanov's Theorem (see e.g. \cite{jacod2013limit}, Chapter III) that
\begin{equation*}
    \forall i\in\cY, \ \forall t\in[0,1], \ \ \frac{d{\P_i}}{d{\P_0}}|_{\mathcal{F}^{X}_{t}}=\exp\left(\int_{0}^{t}{b^{*}_{i}(X_s)dX_s}-\frac{1}{2}\int_{0}^{t}{b^{*2}_{i}(X_s)ds}\right),
\end{equation*}where $(\mathcal{F}^X_t)_{t\in [0,1]}$ is the natural filtration of $X$. Then, for all $i,j\in\cY$ such that $i\neq j$,
\begin{equation}
\label{eq:RatioProba}
     \forall t\in[0,1], \ \ \frac{d{\P_i}}{d{\P_j}}|_{\mathcal{F}^{X}_{t}}=\exp\left(\int_{0}^{t}{(b^{*}_{i}-b^{*}_{j})(X_s)dX_s}-\frac{1}{2}\int_{0}^{t}{(b^{*2}_{i}-b^{*2}_{j})(X_s)ds}\right)\leq C\exp\left(M^{i,j}_t\right)
\end{equation}
where the constant $C$ depends on $C_{b^{*}}$ given in Assumption \ref{ass:boundedDrift} and 
$$ \forall i,j\in\cY:i\neq j, \ \ M^{i,j}_{t}=\int_{0}^{t}{(b^{*}_{i}-b^{*}_{j})(X_s)dW_s}, \ \ t\in[0,1].$$ 
%
%Recall that $Z\in[0,\log^{\alpha}(N)]$ is a random variable measurable with respect to the natural filtration of the diffusion process $X=(X_t)_{t\leq 1}$. 
Then, for all $i,j\in\cY$ such that $i\neq j$ and for all $a>0$, \textcolor{black}{since $\left\|\w{b}_i - b^{*}_{i}\right\|^{2}_{\infty} \leq 2A^2\log(N)$, and} using Equation~\eqref{eq:RatioProba} we have
\begin{align*}
    \textcolor{black}{\left\|\w{b}_i - b^{*}_{i}\right\|^{2}_{n,i}} & \textcolor{black}{ = \frac{1}{n}\sum_{k=0}^{n-1}{\E_{X|Y=i}\left[\left(\w{b}_i - b^{*}_{i}\right)^{2}(X_{k\Delta})\right]} = \frac{1}{n}\sum_{k=0}^{n-1}{\E_{X|Y=j}\left[\left(\w{b}_i - b^{*}_{i}\right)^{2}(X_{k\Delta})\frac{d\P_i}{dP_j}|\mathcal{F}^{X}_{k\Delta}\right]}} \\
    & \textcolor{black}{\leq \frac{C}{n}\sum_{k=0}^{n-1}{\E_{X|Y=j}\left[\left(\w{b}_i - b^{*}_{i}\right)^{2}(X_{k\Delta})\exp\left(M^{i,j}_{k\Delta}\right)\right]}} \\
    & \textcolor{black}{ \leq C\exp(a)\left\|\w{b}_i - b^{*}_{i}\right\|^{2}_{n,j} + CA^2\log(N)\E_{X|Y=j}\left[\exp\left(M^{i,j}_{k\Delta}\right)\one_{M^{i,j}_{k\Delta} > a}\right]}
\end{align*}
Using the Cauchy-Schwarz inequality and \textit{Lemma 2.1 in \cite{van1995exponential}}, there exist constants $C>0$ and $c>0$ depending on $C_{\mathbf{b}^{*}}$ such that,
\begin{align*}
    \E\left[\exp\left(M^{i,j}_{t}\right)\one_{M^{i,j}_{t}>a}\right] \leq &~ %\sqrt{\E\left[\one_{M^{i,j}_{t} > a}\right]} \sqrt{\E\left[\exp\left(2M^{i,j}_{t}\right)\right]} \\
    %= &~ 
    \sqrt{\P\left(M^{i,j}_{t} > a\right)} \sqrt{\E\left[\exp\left(2M^{i,j}_{t}-2\left<M^{i,j},M^{i,j}\right>_t\right) \exp\left(2\left<M^{i,j},M^{i,j}\right>_t\right)\right]} \\
    \leq &~ C\textcolor{black}{\exp(-a^2/2c)}\sqrt{\E\left[\exp\left(2M^{i,j}_{t}-2\left<M^{i,j},M^{i,j}\right>_t\right)\right]}
\end{align*}
where $\P\left(M^{i,j}_{t} > a\right) \leq \exp(-a^2/c)$ (\cite{van1995exponential}) and   $\exp\left(2\left<M^{i,j},M^{i,j}\right>_t\right) < \infty \ a.s$ since the drift functions are bounded. Moreover, since $(M^{i,j}_{t})_{t\leq 1}$ is a martingale and $$\E\left[\exp\left(\left<M^{i,j},M^{i,j}\right>_{1}\right)\right]  
<\infty,$$
according to the Novikov assumption, thus
$\mathcal{E}(M^{i,j}):=\left\{\exp\left(2M^{i,j}_{t}-2\left<M^{i,j},M^{i,j}\right>_t\right)\right\}_{t\leq 1}$ is a martingale with respect to the natural filtration $\mathcal{F}^{M}$ of $M^{i,j}$ (see \cite{le2013mouvement}, Proposition 5.8 and Theorem 5.9).
We deduce that for all $t\in[0,1]$,  
\begin{align*}
    \E\left[\exp\left(2M^{i,j}_{t}-2 \left<M^{i,j},M^{i,j}\right>_t\right)\right]=&~\E\left[\E\left(\mathcal{E}(M^{i,j})_t|\mathcal{F}^{M}_{0}\right)\right]=\E\left[\exp\left(2M^{i,j}_{0}-2\left<M^{i,j},M^{i,j}\right>_0\right)\right]=1.
\end{align*}
Thus, for all $a>0$, we obtain $\E\left[\exp\left(M^{i,j}_{t}\right)\one_{M^{i,j}_{t}>a}\right]\leq C\exp(-a^2/c)$. Finally, set $a=\sqrt{c\log(N)}$, it follows that for all $i,j\in\cY$ such that $i\neq j$, there exists a constant $C>0$ such that
\begin{equation*}
    \textcolor{black}{\left\|\w{b}_i - b^{*}_{i}\right\|^{2}_{n,j} \leq C\exp\left(\sqrt{c\log(N)}\right)\left\|\w{b}_i - b^{*}_{i}\right\|^{2}_{n,i}+C\frac{A^2\log(N)}{N}.}
\end{equation*}
\end{proof}

\begin{proof}[\textbf{Proof of Theorem~\ref{thm:boundedDrift}~}]
   From Theorem~\ref{thm:comparisonInequality}, and its assumptions, we have
\begin{equation*}
   \E\left[\mathcal{R}(\w{g}) - \mathcal{R}(g^*)\right] \leq C \left(\sqrt{\Delta} + \frac{1}{\mfp^{*}_0 \sqrt{N}} 
    + \E\left[{b}_{\rm max}{\sigma}^{-2}_0  \sum_{i =1}^K \|\w{b}_i -b_i^*\|_n\right] + \E\left[{\sigma}^{-2}_0 \|\w{\sigma}^2-\sigma^{2*}\|_n\right]\right).
\end{equation*}
For all $i\in\cY$
%and for $N$ large enough, $\|\w{b}_i-b^{*}_{i}\|_n\leq\sqrt{2\log(N)}$. Then 
we obtain from Proposition~\ref{prop:eqNorm} \textcolor{black}{with $A_{N_i} = \log(N_i) \leq \log(N)$} that there exist constants $C_1,c>0$ such that
\begin{equation*}
    \E\left[\left\|\w{b}_i-b^{*}_{i}\right\|_n\right] = \sum_{j=1}^{K}{\mfp^{*}_{j}\E\left[\left\|\w{b}_i-b^{*}_{i}\right\|_{n,j}\right]}\leq~C_1\exp\left(\sqrt{c\log(N)}\right)\E\left[\left\|\w{b}_i-b^{*}_{i}\right\|_{n,i}\right]+C_1\frac{\textcolor{black}{\log^{3}(N)}}{N}.
\end{equation*}
Then, from Theorem~\ref{thm:cveDriftSig} \textcolor{black}{with $A_{N_i} = \log(N_i), ~ K_{N_i} \propto (N_i\log(N_i))^{1/5}$ on the event $\{N_i > 1\}$ for each $i \in \mathcal{Y}$, and $\tilde{A}_N = \log(N)$, $\tilde{K}_N \propto (N\log(N))^{1/5}$ and $\Delta = \mathrm{O}(1/N)$}, there exist constants $C_2,C_3>0$ such that
\begin{equation*}
    \forall i\in\cY, \ \E\left\|\w{b}_i-b^{*}_{i}\right\|_{n,i}\leq~C_2\left(\frac{\log^{4}(N)}{N}\right)^{1/5}, \ \ \mathrm{and} \ \ \E\left\|\w{\sigma}^{2}-\sigma^{*2}\right\|_n\leq~C_3\left(\frac{\log^{4}(N)}{N}\right)^{1/5}.
\end{equation*}
Finally, by \eqref{eq:ordre-exp-sqrt-log}, we deduce that there exist constants $C, c>0$ such that
\begin{equation*}
   \E\left[\mathcal{R}(\w{g}) - \mathcal{R}(g^*)\right] \leq C \exp\left(\sqrt{c\log(N)}\right)N^{-1/5}.
\end{equation*}
\end{proof}

Let us now turn to the proof of Theorem~\ref{thm:LossErrorDrift}. 
%For fixed $n$ and $N_i$ in $\mathbb{N}^{*}$, let us denote 
%\begin{equation*}
    %\label{equivalence-set}
%\Omega_{n,N_i,K_{N_i}}:=\underset{h\in\mathcal{S}_{K_{N_i},M}\setminus\{0\}}{\bigcap}{\left\{\left|\frac{\|h\|^{2}_{n,N_i}}{\|h\|^{2}_{n,i}}-1\right|\leq\frac{1}{2}\right\}}.
%\end{equation*}
%
%As we can see, the empirical norms $\|h\|_{n,N_i}$ and $\|h\|_{n,i}$ of any function $h\in\mathcal{S}_{K_{N_i},M}\setminus\{0\}$ are equivalent on $\Omega_{n,N_i,K_{N_i}}$. More precisely, on the set $\Omega_{n,N_i,K_{N_i}}$, for all $h\in\mathcal{S}_{K_{N_i},M}\setminus\{0\}$, we have
%$$\frac{1}{2}\|h\|^{2}_{n,i}\leq\|h\|^{2}_{n,N_i}\leq\frac{3}{2}\|h\|^{2}_{n,i}.$$ 
We have the following lemma.
\begin{lemme}
\label{lm:proba-complementary-omega}
Let $\beta\geq 1$ be a real number and suppose that $K_{N_i}=\mathrm{O}\left(\log^{-5/2}(N_i)N^{1/(2\beta+1)}_{i}\right)$ with $N_i \ a.s$ large enough, and $A_{N_i}=\sqrt{\frac{3\beta}{2\beta+1}\log(N_i)}$. Under Assumption~\ref{ass:RegEll}, the following holds:
\begin{align*}
\P_i\left(\Omega^{c}_{n,N_i,K_{N_i}}\right)\leq c\frac{K_{N_i}}{N_i}
\end{align*}
where $c>0$ is a constant.
\end{lemme}

\subsection{Proofs of Section~\ref{sec:ratesKnownSigma} }

\begin{proof}[\textbf{Proof of Theorem \ref{thm:LossErrorDrift}}]
Note that throughout the proof we work conditional on the random variables $\one_{Y_1=i},\cdots,\one_{Y_N=i}$ and on the event $\{N_i>1\}$, so that $N_i$ can be viewed as a deterministic variable. Then, to alleviate the notations, let use denote
$$\P_i:=\P(.|\one_{Y_1=i},\cdots,\one_{Y_N=i}) \ \ \mathrm{and} \ \ \E_i=\E[.|\one_{Y_1=i},\cdots,\one_{Y_N=i}].$$
For each class $i\in\mathcal{Y}$, the drift function $b^{*}_{i}$ is the solution of the following regression model
\begin{equation*}
Z^{j}_{k\Delta}=b^{*}_{i}(X^{j}_{k\Delta})+\xi^{j}_{k\Delta}+R^{j}_{k\Delta}, \ \ j\in \mathcal{I}_i, \ \ k\in[\![0,n-1]\!]
\end{equation*}
where we recall that $\mathcal{I}_i$ is the set of indices $j$ such that $Y_j=i$, and 
\begin{equation}
\label{eq:ErrorDrift}
    \xi^{j}_{k\Delta}:=\frac{1}{\Delta}\int_{k\Delta}^{(k+1)\Delta}{\sigma^{*}(X^{j}_{s})dW^{j}_{s}}, \ \ \ \ R^{j}_{k\Delta}:=\frac{1}{\Delta}\int_{k\Delta}^{(k+1)\Delta}{(b^{*}_{i}(X^{j}_{s})-b^{*}_{i}(X^{j}_{k\Delta}))ds}.
\end{equation}
We first focus on the error $\E_i\left[\left\|\w{b}_i-b^{*}_{A_{N_i},i}\right\|^{2}_{n,N_i}\right]$ for each label $i\in\cY$. Therefore, we consider the following decomposition:
\begin{equation}
\label{eq:DecompRisk}
\E_i\left[\left\|\w{b}_i-b^{*}_{A_{N_i},i}\right\|^{2}_{n,N_i}\right]=\E_i\left[\left\|\w{b}_i-b^{*}_{A_{N_i},i}\right\|^{2}_{n,N_i}\one_{\Lambda_i}\right]+\E_i\left[\left\|\w{b}_i-b^{*}_{A_{N_i},i}\right\|^{2}_{n,N_i}\one_{\Lambda^{\prime}_i}\right]   
\end{equation}
where 
\begin{equation*}
\Lambda_i=\Omega_{n,N_i,K_{N_i}} \ \ \mathrm{and} \ \  \Lambda^{\prime}_i=\Omega^{c}_{n,N_i,K_{N_i}}.
\end{equation*}

\paragraph{Upper bound of $\E_i\left[\left\|\w{b}_i-b^{*}_{A_{N_i},i}\right\|^{2}_{n,N_i}\one_{\Lambda_i}\right]$.}
From the proof of Proposition 4.4 in \cite{denis2020ridge}, Equation (D.5), we see that for all $h\in\mathcal{S}_{K_{N_i},M}$ and for all $a,d>0$, we have on the event $\Lambda_i=\Omega_{n,N_i,K_{N_i}}$,
\begin{align*}
    \left(1-\dfrac{2}{a}-\dfrac{4}{d}\right)\left\|\w{b}_i-b^{*}_{A_{N_i},i}\right\|^{2}_{n,N_i}\leq&~\left(1+\dfrac{2}{a}+\dfrac{4}{d}\right)\left\|h-b^{*}_{A_{N_i},i}\right\|^{2}_{n,N_i}+d\underset{\left\{h\in\mathcal{S}_{K_{N_i},M}, \|h\|_{n,i}=1\right\}}{\sup}{\nu^{2}\left(t\right)}+aC\Delta
\end{align*}
where $C>0$ is a constant and where for all $h\in\mathcal{S}_{K_{N_i},M}$,
\begin{equation}
\label{eq:definition-nu-mu}
    \nu(h)=\frac{1}{N_in}\sum_{j \in I_i}{\sum_{k=0}^{n-1}{h(X^{j}_{k\Delta})\xi^{j}_{k\Delta}}}.
\end{equation}
We set $a=d=8$, and we obtain,
\begin{eqnarray*}
\E_i\left[\left\|\w{b}_{i}-b^{*}_{A_{N_i},i}\right\|^{2}_{n,N_i}\one_{\Lambda_i}\right]&\leq & ~7\underset{h\in\mathcal{S}_{K_{N_i},M}}{\inf}{\left\|h-b^{*}_{A_{N_i},i}\right\|^{2}_{n,i}}+32\E_i\left[\underset{\left\{h\in\mathcal{S}_{K_{N_i},M}, \|h\|_{n,i}=1\right\}}{\sup}{\nu^{2}\left(h\right)}\right]\\
   && +32C\Delta.
\end{eqnarray*}
For $h\in\mathcal{S}_{K_{N_i},M}$, $h=\sum_{\ell=-M}^{K_{N_i}-1}{w_{\ell}B_{\ell,M,\mathbf{u}}}$ and $\|h\|^{2}_{n,i}=w^{\prime}\Psi^i_{K_{N_i}}w$ equals to one here, then $w=\Psi^{-1/2}_{K_{N_i}}u$ where the vector $u$ satisfies $\|u\|_{2,K_{N_i}+M}=1$. Finally, one obtains, 
\begin{equation}
    \label{relation t-psi matrix}
    h=\sum_{\ell=-M}^{K_{N_i}-1}{w_{\ell}B_{\ell,M,\mathbf{u}}}=\sum_{\ell=-M}^{K_{N_i}-1}{u_{\ell}\left(\sum_{\ell^{\prime}=-M}^{K_{N_i}-1}{\left[\Psi^{-1/2}_{K_{N_i}}\right]_{\ell^{\prime},\ell}B_{\ell^{\prime},M,\mathbf{u}}}\right)}.
\end{equation}
For all $h\in\mathcal{S}_{K_{N_i},M}$ such that $\|h\|_{n,i}=1$, using Equation~\eqref{eq:definition-nu-mu}~and~\eqref{relation t-psi matrix}, gives
\begin{align*}
\nu^{2}(h)&=\left(\sum_{\ell=-M}^{K_{N_i}-1}{u_{\ell}\frac{1}{N_in}\sum_{j=1}^{N_i}{\sum_{k=0}^{n-1}{\sum_{\ell^{\prime}=-M}^{K_{N_i}-1}{\left[\Psi^{-1/2}_{K_{N_i}}\right]_{\ell^{\prime},\ell}B_{\ell^{\prime},M,\mathbf{u}}(X^{i_j}_{k\Delta})\xi^{i_j}_{k\Delta}}}}}\right)^2.
\end{align*}
Cauchy-Schwarz inequality together with $\|u\|_2=1$, produce
\begin{align*}
\nu^{2}(h)&\leq\sum_{\ell=-M}^{K_{N_i}-1}{\left(\frac{1}{N_in}\sum_{j=1}^{N_i}{\sum_{k=0}^{n-1}{\sum_{\ell^{\prime}=-M}^{K_{N_i}-1}{\left[\Psi^{-1/2}_{K_{N_i}}\right]_{\ell^{\prime},\ell}B_{\ell^{\prime},M,\mathbf{u}}(X^{i_j}_{k\Delta})\xi^{i_j}_{k\Delta}}}}\right)^2}.
\end{align*}
Finally we obtain,
\begin{eqnarray*}
\E_i\left[\underset{h\in\mathcal{S}_{K_{N_i},M}, \|h\|_{n,i}=1}{\sup}{\nu^{2}(h)}\right]
%&=&
%\E_i\left[\underset{t\in\mathcal{S}_{K_{N_i},M}, \|t\|_{n,i}=1}{\sup}{\nu^{2}(t)}\right]\\
&\leq& \frac{1}{N_i}\E_i\left[\frac{1}{n^2}\sum_{\ell=-M}^{K_{N_i}-1}{\left(\sum_{k=0}^{n-1}{\sum_{\ell^{\prime}=-M}^{K_{N_i}-1}{\left[\Psi^{-1/2}_{K_{N_i}}\right]_{\ell^{\prime},\ell}B_{\ell^{\prime},M,\mathbf{u}}(X^{i_1}_{k\Delta})\xi^{i_1}_{k\Delta}}}\right)^2}\right]\\
&=&\frac{1}{N_i}\E_i\left[\frac{1}{n^2}\sum_{\ell=-M}^{K_{N_i}-1}{\sum_{k=0}^{n-1}{\left(\sum_{\ell^{\prime}=-M}^{K_{N_i}-1}{\left[\Psi^{-1/2}_{K_{N_i}}\right]_{\ell^{\prime},\ell}B_{\ell^{\prime},M,\mathbf{u}}(X^{i_1}_{k\Delta})}\right)^2\left(\xi^{i_1}_{k\Delta}\right)^2}}\right].
\end{eqnarray*}
According to Equation~\eqref{eq:ErrorDrift}~ and considering the natural filtration $\left(\mathcal{F}_{t}\right)_{t\geq 0}$ of the Brownian motion, for all $k\in[\![0,n-1]\!]$, we have $\E_i\left(\xi^{i_1}_{k\Delta}|\mathcal{F}_{k\Delta}\right)=0$ and
\begin{align*}
\E_i\left[\left(\xi^{i_1}_{k\Delta}\right)^2|\mathcal{F}_{k\Delta}\right]=\frac{1}{\Delta^2}\E\left[\sigma^{*2}\left(X^{i_1}_{k\Delta}\right)\E\left(\left(\int_{k\Delta}^{(k+1)\Delta}{\sigma^{*}(X^{i_1}_{s})}\right)^2|\mathcal{F}_{k\Delta}\right)\right]\leq\frac{\sigma^{*2}_{1}}{\Delta}.
\end{align*}
By definition of the Gram matrix $\Psi_{K_{N_i}}$, we deduce that
\begin{align*}
    \E_i\left[\underset{h\in\mathcal{S}_{K_{N_i},M}, \|h\|_{n,i}=1}{\sup}{\nu^{2}(h)}\right]\leq & \frac{\sigma^{*2}_{1}}{N_i}\E_i\left[\frac{1}{n}\sum_{\ell=-M}^{K_{N_i}-1}{\sum_{k=0}^{n-1}{\left(\sum_{\ell^{\prime}=-M}^{K_{N_i}-1}{\left[\Psi^{-1/2}_{K_{N_i}}\right]_{\ell^{\prime},\ell}B_{\ell^{\prime},M,\mathbf{u}}(X^{1,i}_{k\Delta})}\right)^2}}\right]\\
    \leq & \frac{\sigma^{*2}_{1}}{N_i}\E_i\left(\sum_{\ell,\ell^{\prime},\ell^{\prime\prime}=-M}^{K_{N_i}-1}{\left[\Psi^{-1/2}_{K_{N_i}}\right]_{\ell^{\prime},\ell}\left[\Psi^{-1/2}_{K_{N_i}}\right]_{\ell^{\prime\prime},\ell}\left[\Psi_{K_{N_i}}\right]_{\ell^{\prime},\ell^{\prime\prime}}}\right)\\
    = & \frac{\sigma^{*2}_{1}}{N_i}\E_i\left(\mathrm{Tr}\left(\Psi^{-1}_{K_{N_i}}\Psi_{K_{N_i}}\right)\right).
\end{align*}
Besides,
$$
\mathrm{Tr}\left(\Psi^{-1}_{K_{N_i}}\Psi_{K_{N_i}}\right)=K_{N_i}+M.
$$
Thus, finally, there exists a constant $C_1>0$ depending on $\sigma^{*}_{1}$ and $M$ such that
\begin{equation*}
\E_i\left[\underset{h\in\mathcal{S}_{K_{N_i},M}, \|h\|_{n,i}=1}{\sup}{\nu^{2}(h)}\right]\leq  C_1\frac{K_{N_i}}{N_i}.
\end{equation*}
Thus, there exists a constant $C>0$ such that,
\begin{equation}
\label{eq:ErrorTerm1}
    \E_i\left[\left\|\w{b}_{i}-b^{*}_{A_{N_i},i}\right\|^{2}_{n,N_i}\one_{\Lambda_i}\right]\leq~7\underset{h\in\mathcal{S}_{K_{N_i},M}}{\inf}{\left\|h-b^{*}_{A_{N_i},i}\right\|^{2}_{n,i}}+C\left(\frac{K_{N_i}}{N_i}+\Delta\right).
\end{equation}

\paragraph{Upper bound of $\E\left[\left\|\w{b}_{i}-b^{*}_{A_{N_i},i}\right\|^{2}_{n,N_i}\one_{\Lambda^{\prime}_i}\right]$.}

Using the Cauchy-Schwarz inequality, we have
\begin{equation*}
    \E_i\left[\left\|\w{b}_{i}-b^{*}_{A_{N_i},i}\right\|^{2}_{n,N_i}\one_{\Lambda^{\prime}_{i}}\right] \leq C_0\log^{2}(N_i)\P_i\left(\Omega^{c}_{n,N_i,K_{N_i}}\right)
\end{equation*}
since for $N$ large enough, using \eqref{eq:boundedbi}, we have,
\begin{align*}
    \left\|\w{b}_{i}-b^{*}_{A_{N_i},i}\right\|^{2}_{n,N_i}\leq &  2\|\w{b}_i\|^{2}_{\infty} + 2\|b^{*}_{A_{N_i},i}\|^{2}_{\infty} \leq 4A^{2}_{N_i}\log(N_i) \leq C_0\log^{2}(N_i)
\end{align*}
where $C_0>0$ is a constant. Using Lemma~\ref{lm:proba-complementary-omega}, we have
\begin{equation}
    \label{eq:ProbaGammaPrime}
    \P_i(\Lambda^{\prime}_{i})=\mathbb{P}_i\left(\Omega^{c}_{n,N_i,K_{N_i}}\right)\leq c\frac{K_{N_i}}{N_i}.
\end{equation}
Then, from Equation~\eqref{eq:ProbaGammaPrime}, there exists a constant $C>0$ such that
\begin{equation}
\label{eq:ErrorTerm2}
    \E_i\left[\left\|\w{b}_{i}-b^{*}_{A_N,i}\right\|^{2}_{n,N_i}\one_{\Lambda^{\prime}_i}\right] \leq C\log^{2}(N_i)\frac{K_{N_i}}{N_i}.
\end{equation}

\paragraph{Upper bound of $\E_i\left[\left\|\w{b}_{i}-b^{*}_{A_{N_i},i}\right\|^{2}_{n,N_i}\right]$.}

From Equations~\eqref{eq:DecompRisk},~\eqref{eq:ErrorTerm1}~and~\eqref{eq:ErrorTerm2}, there exists a constant $C>0$ such that
\begin{equation}
    \label{eq:EmpLossError}
    \E_i\left[\left\|\w{b}_{i}-b^{*}_{A_{N_i},i}\right\|^{2}_{n,N_i}\right]\leq 7\underset{h\in\mathcal{S}_{K_{N_i},M}}{\inf}{\left\|h-b^{*}_{A_{N_i},i}\right\|^{2}_{n,i}}+C\left(\frac{\log^{2}(N_i)K_{N_i}}{N_i}+\Delta\right).
\end{equation}

\paragraph{Upper bound of $\E_i\left[\left\|\w{b}_{i}-b^{*}_{A_{N_i},i}\right\|^{2}_{n,i}\right]$.}

Using Equation~\eqref{eq:EmpLossError}, we have
\begin{align*}
    \E_i\left[\left\|\w{b}_{i}-b^{*}_{A_{N_i},i}\right\|^{2}_{n,i}\right] =&~ \E_i\left[\left\|\w{b}_{i}-b^{*}_{A_{N_i},i}\right\|^{2}_{n,i}\right] - 2\E_i\left[\left\|\w{b}_{i}-b^{*}_{A_{N_i},i}\right\|^{2}_{n,N_i}\right]\\
    &+2\E_i\left[\left\|\w{b}_{i}-b^{*}_{A_{N_i},i}\right\|^{2}_{n,N_i}\right]\\
    \leq&~ \E_i\left[\left\|\w{b}_{i}-b^{*}_{A_{N_i},i}\right\|^{2}_{n,i}\right]-2\E_i\left[\left\|\w{b}_{i}-b^{*}_{A_{N_i},i}\right\|^{2}_{n,N_i}\right]\\
    &+7\underset{h\in\mathcal{S}_{K_{N_i},M}}{\inf}{\left\|h-b^{*}_{A_{N_i},i}\right\|^{2}_{n,i}}+C\left(\frac{\log^{2}(N_i)K_{N_i}}{N_i}+\Delta\right).
\end{align*}
From the proof of Theorem~\ref{thm:cveDriftSig}, we deduce that 
$$\E_i\left[\left\|\w{b}_{i}-b^{*}_{A_{N_i},i}\right\|^{2}_{n,i}\right]-2\E_i\left[\left\|\w{b}_{i}-b^{*}_{A_{N_i},i}\right\|^{2}_{n,N_i}\right]\leq C\log^{3}(N_i)K_{N_i}/N_i$$ 
with $C>0$ a constant depending on $\mfp_0=\underset{i\in\cY}{\min}{\mfp^{*}_{i}}$. Besides, since $b^{*}_{i}\in\Sigma(\beta,R)$, we have 
$$\underset{h\in\mathcal{S}_{K_{N_i},M}}{\inf}{\|h-b^{*}_{A_{N_i},i}\|^{2}_{n,i}}\leq C\left(\frac{A_{N_i}}{K_{N_i}}\right)^{2\beta}$$
where $C>0$ is a constant (see \cite{denis2020ridge}, Lemma D.2). Then it comes that
$$\E_i\left[\left\|\w{b}_{i}-b^{*}_{A_{N_i},i}\right\|^{2}_{n,i}\right] \leq C\left(\left(\frac{A_{N_i}}{K_{N_i}}\right)^{2\beta}+\frac{K_{N_i}\log^{3}(N_i)}{N_i}+\Delta\right)
    $$
where $C>0$ is a constant depending on $\beta$, $\Delta=\mathrm{O}(1/N)$. Since $$K_{N_i}=\mathrm{O}\left(\log^{-5/2}(N_i)N^{1/(2\beta+1)}_{i}\right),$$
we obtain
\begin{equation*}
    \E_i\left[\left\|\w{b}_{i}-b^{*}_{A_{N_i},i}\right\|^{2}_{n,i}\right] \leq  C\log^{6\beta}(N_i)N^{-\frac{2\beta}{2\beta+1}}_{i} \leq C\log^{6\beta}(N)N^{-\frac{2\beta}{2\beta+1}}_{i}.
\end{equation*}
Using the Jensen's inequality,
\begin{equation*}
  \E\left[\one_{N_i>1}\left\|\w{b}_{i}-b^{*}_{A_{N_i},i}\right\|^{2}_{n,i}\right] \leq C\log^{6\beta}(N)\E\left[\one_{N_i>1}N^{-\frac{2\beta}{2\beta+1}}_{i}\right] \leq C\log^{6\beta}(N)\left(\E\left[\frac{\one_{N_i>1}}{N_i}\right]\right)^{\frac{2\beta}{2\beta+1}}.  
\end{equation*}
Using again \textit{Lemma 4.1} from \cite{gyorfi2006distribution}, we obtain
\begin{equation*}
    \E\left[\one_{N_i>1}\left\|\w{b}_{i}-b^{*}_{A_{N_i},i}\right\|^{2}_{n,i}\right] \leq C\log^{6\beta}(N)\left(\E\left[\frac{\one_{N_i>1}}{N_i}\right]\right)^{\frac{2\beta}{2\beta+1}} \leq C\log^{6\beta}(N)N^{-\frac{2\beta}{2\beta+1}}.
\end{equation*}
\end{proof}

\begin{proof}[\textbf{Proof of Theorem~\ref{thm:AlmostOptimalRate}~}]
For all $i\in\cY$, recall that $b^{*}_{A_{N_i},i}=b^{*}_{i}\one_{[-A_{N_i},A_{N_i}]}$. Furthermore, set 
\begin{equation}
\label{eq:MinA_Ni}
    N_0:=\underset{i\in\cY}{\min}{\ N_i}, \ \ \mathrm{then} \ \ A_{N_0}:=\underset{i\in\cY}{\min}{A_{N_i}}.
\end{equation}
We have 
$$ \E\left[\cR(\w{g})-\cR(g^{*})\right] =  \E\left[\left(1-\cR(g^{*})\right)\one_{N_0 \leq 1}\right] +  \E\left[(\cR(\w{g})-\cR(g^{*}))\one_{N_0>1}\right]. $$
Then, from Proposition~\ref{prop:excessRiskClass}, we deduce that
\begin{eqnarray*}
    \E\left[\cR(\w{g})-\cR(g^{*})\right] 
    &\leq & \sum_{i=1}^{K}{\P(N_i \leq 1)} + 2\sum_{i=1}^{K}{\E\left[\left|\w{\pi}_{i}(X)-\pi^{*}_{i}(X)\right|\one_{N_0>1}\right]} \\
    &\leq & 2KN(1-\mfp^{*}_{0})^{N-1} + 2\sum_{i=1}^{K}{\E\left[\left|\w{\pi}_{i}(X)-\pi^{*}_{i}(X)\right|\one_{N_0>1}\right]}
\end{eqnarray*}
since $\w{g}=1$ on the event $\{N_0 \leq 1\}$. For all $i\in\cY$ and on the event $\{N_0>1\}$,
\begin{equation*}
    \left|\w{\pi}_{i}(X)-\pi^{*}_{i}(X)\right|\leq\left|\w{\pi}_{i}(X)-\bar{\pi}^{A_{N_0}}_{i}(X)\right|+\left|\bar{\pi}^{A_{N_0}}_{i}(X)-\bar{\pi}^{*}_{i}(X)\right|+\left|\bar{\pi}^{*}_{i}(X)-\pi^{*}_{i}(X)\right|
\end{equation*}
where $\bar{\pi}^{A_{N_0}}_{i}(X):=\phi_i\left(\bar{\textbf{F}}^{A_{N_0}}\right)$ and $\bar{\textbf{F}}^{A_{N_0}}=\left(\bar{F}^{A_{N_0}}_{1},\cdots,\bar{F}^{A_{N_0}}_{K}\right)$ with
\begin{equation*}
    \forall i\in\cY, \ \ \bar{F}^{A_{N_0}}_{i}=\sum_{k=0}^{n-1}{b^{*}_{A_{N_0},i}(X_{k\Delta})(X_{(k+1)\Delta}-X_{k\Delta})-\frac{\Delta}{2}b^{*2}_{A_{N_0},i}(X_{k\Delta})}.
\end{equation*}
Then, there exists a constant $c>0$ such that
\begin{eqnarray*}
    \E\left[\cR(\w{g})-\cR(g^{*})\right]&\leq & 2\left(\sum_{i=1}^{K}{\E\left(\left|\w{\pi}_{i}(X)-\bar{\pi}^{A_{N_0}}_{i}(X)\right|\one_{N_0>1}\right)}+\sum_{i=1}^{K}{\E\left(\left|\bar{\pi}^{A_{N_0}}_{i}(X)-\bar{\pi}^{*}_{i}(X)\right|\one_{N_0>1}\right)}\right)\\
    &&+ c(1-\mfp^{*}_0)^{N/2}  +2\sum_{i=1}^{K}{\E\left|\bar{\pi}^{*}_{i}(X)-\pi^{*}_{i}(X)\right|}.
\end{eqnarray*}
From the proof of Theorem~\ref{thm:comparisonInequality}, there exists a constant $C_1>0$ depending on $K, \mfp^{*}_0$ and $C_{\bf b^{*}}$ and a constant $C_2>0$ depending on $K$ such that  
\begin{align*}
    & \sum_{i=1}^{K}{\E\left|\w{\pi}_{i}(X)-\bar{\pi}^{A_{N_0}}_{i}(X)\right|}\leq C_1\left(\frac{1}{\sqrt{N}}+\sum_{i=1}^{K}{\E\left[\one_{N_0>1}\left\|\w{b}_i-b^{*}_{A_{N_0},i}\right\|_n\right]}\right),\\
    & \sum_{i=1}^{K}{\E\left|\bar{\pi}^{*}_{i}(X)-\pi^{*}_{i}(X)\right|}\leq C_2\sqrt{\Delta}.
\end{align*}
Thus, we have
\begin{eqnarray*}
    \E\left[\cR(\w{g})-\cR(g^{*})\right]&\leq& 2C_1\left(\frac{1}{\sqrt{N}}+\sum_{i=1}^{K}{\E\left[\one_{N_0>1}\left\|\w{b}_i-b^{*}_{A_{N_0},i}\right\|_n\right]}\right)+2C_2\sqrt{\Delta} + c(1-\mfp^{*}_0)^{N/2} \\
    &&+2K\sum_{i=1}^{K}{\E\left[\left|\bar{F}^{A_{N_0}}_{i}(X)-\bar{F}_{i}(X)\right|\one_{N_0>1}\right]}.
\end{eqnarray*}
For all $i\in\cY$,
\begin{eqnarray*}
\E\left[\left|\bar{F}^{A_{N_0}}_{i}(X)-\bar{F}_{i}(X)\right|\one_{N_0>1}\right]&\leq& \E\left[\left|\sum_{k=0}^{n-1}{b^{*}_{i}(X_{k\Delta})\one_{|X_{k\Delta}|>A_{N_0}}\int_{k\Delta}^{(k+1)\Delta}{b^{*}_{i}(X_s)ds}}\right|\one_{N_0>1}\right]\\
    &&+\frac{\Delta}{2}\sum_{k=0}^{n-1}{\E\left[\one_{N_0>1}b^{*2}_{i}(X_{k\Delta})\one_{|X_{k\Delta}|>A_{N_0}}\right]}\\
    &&+\E\left|\sum_{k=0}^{n-1}{b^{*}_{i}(X_{k\Delta})\one_{N_0>1}\one_{|X_{k\Delta}|>A_{N_0}}(W_{(k+1)\Delta}-W_{k\Delta})}\right|.
\end{eqnarray*}
Under Assumption~\ref{ass:boundedDrift}, we easily obtain that
$$
    \E\left|\sum_{k=0}^{n-1}{b^{*}_{i}(X_{k\Delta})\one_{N_0>1}\one_{|X_{k\Delta}|>A_{N_0}}\int_{k\Delta}^{(k+1)\Delta}{b^{*}_{i}(X_s)ds}}\right|\leq  C^{2}_{\mathbf{b}^{*}}\underset{t\in[0,1]}{\sup}{\mathbb{P}\left(\{N_0>1\}\cap\{|X_t|>A_{N_0}\}\right)}, $$
and
$$
    \frac{\Delta}{2}\sum_{k=0}^{n-1}{\E\left[b^{*2}_{i}(X_{k\Delta})\one_{|X_{k\Delta}|>A_{N_0}}\right]}\leq \frac{C^{2}_{\mathbf{b}^{*}}}{2}\underset{t\in[0,1]}{\sup}{\mathbb{P}\left(\{N_0>1\}\cap\{|X_t|>A_{N_0}\}\right)}.
$$
For the last term, consider the natural filtration \textcolor{black}{$(\mathcal{F}_t)_{t \geq 0}$} of the Brownian motion $(W_t)_{t\geq 0}$. For all $k\in[\![0,n-1]\!]$, $X_{k\Delta}$ is measurable with respect to $\mathcal{F}_{k\Delta}$ and $W_{(k+1)\Delta}-W_{k\Delta}$ is independent of $\mathcal{F}_{k\Delta}$ since the Brownian motion is an independently increasing process. Consequently,  \textcolor{black}{setting,}
$$\textcolor{black}{\mathcal{Z} = \E\left|\sum_{k=0}^{n-1}{b^{*}_{i}(X_{k\Delta})\one_{N_0>1}\one_{|X_{k\Delta}|>A_{N_0}}(W_{(k+1)\Delta}-W_{k\Delta})}\right|,}$$
\textcolor{black}{and using the Cauchy Schwarz inequality, we obtain}
\begin{align*}
  \textcolor{black}{\mathcal{Z} \leq} &  \textcolor{black}{\left\{\mathbb{E}\left[\left(\sum_{k=0}^{n-1}{b^{*}_{i}(X_{k\Delta})\mathds{1}_{N_0 > 1}\mathds{1}_{|X_{k\Delta}|>A_{N_0}}(W_{(k+1)\Delta} - W_{k\Delta})}\right)^2\right]\right\}^{1/2}}\\
    \textcolor{black}{\leq} & \textcolor{black}{\left\{\mathbb{E}\left[\sum_{k,\ell = 0}^{n-1}{b^{*}_{i}(X_{k\Delta})b^{*}_{i}(X_{\ell\Delta})\mathds{1}_{N_0 > 1}\mathds{1}_{|X_{k\Delta}|>A_{N_0}}\mathds{1}_{|X_{\ell\Delta}|>A_{N_0}}(W_{(k+1)\Delta} - W_{k\Delta})(W_{(\ell+1)\Delta} - W_{\ell\Delta})}\right]\right\}^{1/2}}\\
    \textcolor{black}{\leq} & \textcolor{black}{\left\{2\mathbb{E}\left[\sum_{k > \ell}{b^{*}_{i}(X_{k\Delta})b^{*}_{i}(X_{\ell\Delta})\mathds{1}_{N_0 > 1}\mathds{1}_{|X_{k\Delta}|>A_{N_0}}\mathds{1}_{|X_{\ell\Delta}|>A_{N_0}}(W_{(k+1)\Delta} - W_{k\Delta})(W_{(\ell+1)\Delta} - W_{\ell\Delta})}\right]\right\}^{1/2}}\\
    & \textcolor{black}{+ \left\{\mathbb{E}\left[\sum_{k = 0}^{n-1}{b^{*2}_{i}(X_{k\Delta})\mathds{1}_{N_0 > 1}\mathds{1}_{|X_{k\Delta}|>A_{N_0}}(W_{(k+1)\Delta} - W_{k\Delta})^2}\right]\right\}^{1/2}} \\
    \textcolor{black}{\leq} & \textcolor{black}{T_1 + T_2.}
\end{align*}
\textcolor{black}{We recall that $(\mathcal{F}_{t})_{t \geq 0}$ is the natural filtration of the Brownian motion $(W_t)_{t \geq 0}$. Since for all $k \in [\![0,n-1]\!]$, $X_{k\Delta}$ is $\mathcal{F}_{k\Delta}-$measurable, we have}
\begin{align*}
    \textcolor{black}{T_2^2 \leq} & \textcolor{black}{\mathbb{E}\left[\sum_{k = 0}^{n-1}{b^{*2}_{i}(X_{k\Delta})\mathds{1}_{N_0 > 1}\mathds{1}_{|X_{k\Delta}|>A_{N_0}}\mathbb{E}\left[(W_{(k+1)\Delta} - W_{k\Delta})^2|\mathcal{F}_{k\Delta}\right]}\right]} \\
    \textcolor{black}{\leq} & \textcolor{black}{C^{2}_{\mathbf{b}^{*}}\underset{t \in [0,1]}{\sup}{\mathbb{P}(N_0 > 1, |X_t| > A_{N_0})}.}
\end{align*}
\textcolor{black}{On the other hand, for all $k,\ell \in [\![0,n-1]\!]$ such that $k > \ell$, we remark that}
\begin{align*}
    \textcolor{black}{T_1^2 \leq} & \textcolor{black}{2\mathbb{E}\left[\sum_{k > \ell}{b^{*}_{i}(X_{k\Delta})b^{*}_{i}(X_{\ell\Delta})\mathds{1}_{N_0 > 1}\mathds{1}_{|X_{k\Delta}|>A_{N_0}}\mathds{1}_{|X_{\ell\Delta}|>A_{N_0}}(W_{(\ell+1)\Delta} - W_{\ell\Delta})\mathbb{E}\left[W_{(k+1)\Delta} - W_{k\Delta}|\mathcal{F}_{k\Delta}\right]}\right]} \\
    \textcolor{black}{=} & \textcolor{black}{0.}
\end{align*}
\textcolor{black}{Thus, we deduce that}
\begin{align*}
    \textcolor{black}{\mathcal{Z} = \E\left|\sum_{k=0}^{n-1}{b^{*}_{i}(X_{k\Delta})\one_{N_0>1}\one_{|X_{k\Delta}|>A_{N_0}}(W_{(k+1)\Delta}-W_{k\Delta})}\right| \leq C_{\mathbf{b}^{*}}\sqrt{\underset{t \in [0,1]}{\sup}{\mathbb{P}(N_0 > 1, |X_t| > A_{N_0})}}.}
\end{align*}
Finally, there exists a constant $C>0$ such that
\begin{equation}
\label{eq:UB-ExcessRisk1}
    \E\left[\cR(\w{g})-\cR(g^{*})\right]\leq C\left(\frac{1}{\sqrt{N}}+\sum_{i=1}^{K}{\sum_{j=1}^{K}{\mfp^{*}_j\E\left[\left\|\w{b}_i-b^{*}_{A_{N_0},i}\right\|_{n,j}\one_{N_0>1}\right]}}+\textcolor{black}{\sqrt{\underset{t \in [0,1]}{\sup}{\mathbb{P}(N_0 > 1, |X_t| > A_{N_0})}}}\right).
\end{equation}
From Proposition~\ref{prop:eqNorm}~ with $\alpha=1$, for all $i,j\in\cY$ such that $i\neq j$, we have
\begin{equation}
\label{eq:equivNorms}
    \E\left[\left\|\w{b}_i-b^{*}_{A_{N_0},i}\right\|_{n,j}\one_{N_0>1}\right]\leq~C\exp(\sqrt{c\log(N)})\E\left[\left\|\w{b}_i-b^{*}_{A_{N_0},i}\right\|_{n,i}\one_{N_0>1}\right]+C\frac{\log(N)}{N}.
\end{equation}
Furthermore, for all $i\in\cY$, we have
\begin{multline*}
    \E\left[\left\|\w{b}_i-b^{*}_{A_{N_0},i}\right\|_{n,i}\one_{N_0>1}\right] \leq \E\left[\left\|\w{b}_i-b^{*}_{A_{N_i},i}\right\|_{n,i}\one_{N_i>1}\right] + \E\left[\left\|b^{*}_{A_{N_i},i}-b^{*}_{A_{N_0},i}\right\|_{n,i}\one_{N_0>1}\right]\\
    \leq \E\left[\left\|\w{b}_i-b^{*}_{A_{N_i},i}\right\|_{n,i}\one_{N_i>1}\right] + \|b^{*}_{i}\|_{\infty}\underset{t\in[0,1]}{\sup}{\P\left(\{A_{N_i}\geq|X_t|>A_{N_0}\}\cap\{N_0>1\}\right)}\\
    \leq \E\left[\left\|\w{b}_i-b^{*}_{A_{N_i},i}\right\|_{n,i}\one_{N_i>1}\right] + C_{{\bf b}^{*}}\underset{t\in[0,1]}{\sup}{\sum_{j\neq i}{\P\left(\{|X_t|>A_{N_j}\}\cap\{N_j>1\}\right)}}.
\end{multline*}
We deduce from Equations~\eqref{eq:UB-ExcessRisk1}~and~\eqref{eq:equivNorms}~ that there exists a constant $C>0$ depending on $C_{{\bf b}^{*}}, K$ and $p_0$ such that
\begin{multline*}
    \E\left[\cR(\w{g})-\cR(g^{*})\right]\leq C\left(\frac{1}{\sqrt{N}}+\exp\left(\sqrt{c\log(N)}\right)\sum_{i=1}^{K}{\E\left[\left\|\w{b}_i-b^{*}_{A_{N_i},i}\right\|_{n,i}\one_{N_i>1}\right]}\right)\\
    +C\exp\left(\sqrt{c\log(N)}\right)\textcolor{black}{\sqrt{\underset{t\in[0,1]}{\sup}{\sum_{i=1}^{K}{\P\left(\{|X_t|>A_{N_i}\}\cap\{N_i>1\}\right)}}}}.
\end{multline*}
Under the Assumptions of the Proposition and according to Theorem~\ref{thm:LossErrorDrift}, there exist two constants $C_1,C_2>0$ such that $\forall i\in\cY$,
\begin{equation*}
 \E\left[\left\|\w{b}_i-b^{*}_{A_{N_i},i}\right\|_{n,i}\one_{N_i>1}\right]\leq C_1\log^{3\beta}(N)N^{-\beta/(2\beta+1)}  
\end{equation*}
and we deduce from Lemma~\ref{lm:DensityConstSigma}~ with $q=3/2$, for all $i\in\cY$, and for all $t\in[0,1]$,
\begin{align*}
    \mathbb{P}(\{|X_t|>A_{N_i}\}\cap\{N_0>1\})= &~ \E\left[\mathbb{P}(\{|X_t|>A_{N_i}\}\cap\{N_i>1\}|\one_{Y_1=i},\cdots,\one_{Y_N=i})\right]\\
    \leq &~ C_2\E\left[\frac{\one_{N_i>1}}{A_{N_i}}\exp\left(-\frac{A^{2}_{N_i}}{3}\right)\right].
\end{align*}
Thus, we obtain
\begin{equation*}
    \E\left[\cR(\w{g})-\cR(g^{*})\right]\leq C\left(\exp\left(2\sqrt{c\log(N)}\right)N^{-\beta/(2\beta+1)}+\textcolor{black}{\sqrt{\sum_{i=1}^{K}{\E\left[\one_{N_i>1}\exp\left(-\frac{A^{2}_{N_i}}{3}\right)\right]}}}\right)
\end{equation*}
where $C>0$ is a constant depending on $\beta, C_{\bf b^{*}}, K, \mfp^{*}_0$. Finally, choosing \textcolor{black}{$A_{N_i}=\sqrt{\frac{6\beta}{2\beta+1}\log(N_i)}$} for each $i\in\cY$ leads to the attended result applying the Jensen's inequality together with Lemma 4.1 in \cite{gyorfi2006distribution}.
\end{proof}

\begin{proof}[\textbf{Proof of Theorem~\ref{thm:RateClass.Re-entrant.Drift}~}]
From Theorem~\ref{thm:comparisonInequality}, as we assumed $\sigma^{*}(.)=1$, the excess risk of $\w{g}$ satisfies
\begin{equation}
\label{eq:ExRiskUp}
    \E\left[\cR(\w{g})-\cR(g^{*})\right]\leq C\left(\sqrt{\Delta}+\frac{1}{\mfp^{*}_0\sqrt{N}}+\sum_{i=1}^{K}{\E\left[\left\|\w{b}_i-b^{*}_{i}\right\|_n\one_{N_i>1}\right]} + \sum_{i=1}^{K}{\P(N_i\leq 1)}\right)
\end{equation}
where the constant $C>0$ depends on $b^{*}=\left(b^{*}_{1},\cdots,b^{*}_{K}\right)$ and $K$. For each $i\in\cY$, we have 
\begin{align*}
    \P(N_i \leq 1) \leq 2N(1-\mfp^{*}_{0})^{N-1}
\end{align*}
and
\begin{align*}
    \E\left[\left\|\w{b}_i-b^{*}_{i}\right\|_n\one_{N_i>1}\right] \leq & \sqrt{\E\left[\left\|\w{b}_i-b^{*}_{A_{N_i},i}\right\|^{2}_{n}\one_{N_i>1}\right]+\E\left[\left\|b^{*}_{i}\one_{[-A_{N_i},A_{N_i}]^{c}}\right\|^{2}_{n}\one_{N_i>1}\right]}
\end{align*}
Using the Cauchy-Schwarz inequality and Assumption~\ref{ass:RegEll}, there exists a constant $C^{\prime}>0$ such that
\begin{equation*}
    \E\left[\left\|b^{*}_{i}\one_{[-A_{N_i},A_{N_i}]^{c}}\right\|^{2}_{n}\one_{N_i>1}\right] \leq C^{\prime}\sqrt{\underset{t\in[0,1]}{\sup}{\P(\{|X_t|>A_{N_i}\}\cap\{N_i>1\})}}.
\end{equation*}
Thus, for all $i\in\cY$, we obtain
\begin{equation}
\label{eq:FirstInq}
    \E\left\|\w{b}_i-b^{*}_{i}\right\|_n\leq\sqrt{\E\left[\left\|\w{b}_i-b^{*}_{A_{N_i},i}\right\|^{2}_{n}\one_{N_i>1}\right]+C^{\prime}\sqrt{\underset{t\in[0,1]}{\sup}{\P(\{|X_t|>A_{N_i}\}\cap\{N_i>1\})}}}.
\end{equation}
For each label $i\in\cY$,
\begin{align*}
    \E\left[\left\|\w{b}_i-b^{*}_{A_{N_i},i}\right\|^{2}_{n}\one_{N_i>1}\right] = & E\left(\one_{N_i>1}\int_{-A_{N_i}}^{A_{N_i}}{\left(\w{b}_i-b^{*}_{A_{N_i},i}\right)^{2}(x)f_{n,Y}(x)dx}\right)+\frac{2\log^{3}(N)}{n}
\end{align*}
where $$f_{n,Y}(x):=\frac{1}{n}\sum_{k=1}^{n-1}{\textcolor{black}{p_{Y,X}(k\Delta,x)}}.$$
From the proof of Lemma~\ref{lm:MinEigenValue}, under Assumption~\ref{ass:RegEll}, there exist constants $C_1,C_2>0$ such that on the event $\{N_i>1\}$,
\begin{equation*}
    \forall x\in[-A_{N_i},A_{N_i}], \  f_{n,Y}(x)\geq\frac{C_1}{\log(N)}\exp\left(-\frac{2A^{2}_{N_i}}{3(1-\log^{-1}(N))}\right)\geq\frac{C_2}{\log(N)}\exp\left(-\frac{2}{3}A^{2}_{N_i}\right) \ a.s
\end{equation*}
and from Lemma~\ref{lem:boundDensity}~ there exists another constant $C_0>0$ such that $f_{n,Y}(x)\leq C_0$ for all $x\in\R$. Then we have
\begin{equation*}
    \forall i\in\cY, \ \forall x\in[-A_{N_i},A_{N_i}], \ \frac{f_{n,Y}(x)}{f_{n,i}(x)}\leq\frac{C_0}{C_2}\log(N)\exp\left(\frac{2}{3}A^{2}_{N}\right).
\end{equation*}
Then, for all $i\in\cY$, we obtain
\begin{align*}
     \E\left[\left\|\w{b}_i-b^{*}_{A_{N_i},i}\right\|^{2}_{n}\one_{N_i>1}\right] \leq &~ \E\left[\one_{N_i>1}\int_{-A_{N_i}}^{A_{N_i}}{\left(\w{b}_i-b^{*}_{A_{N_i},i}\right)^{2}(x)f_{n,i}(x)\frac{f_{n,Y}(x)}{f_{n,i}(x)}}\right]+\frac{2\log^{3}(N)}{n}\\
    \leq&~\frac{C_0}{C_2}\log(N)\exp\left(\frac{2}{3}A^{2}_{N}\right)\E\left[\left\|\w{b}_i-b^{*}_{A_{N_i},i}\right\|^{2}_{n,i}\one_{N_i>1}\right]+\frac{2\log^{3}(N)}{n}.
\end{align*}
From Theorem~\ref{thm:LossErrorDrift}, Equation~\eqref{eq:FirstInq}~ and for $n\propto N$, there exists a constant $C_3>0$ such that
\begin{equation*}
    \E\left[\left\|\w{b}_i-b^{*}_{i}\right\|_n\one_{N_i>1}\right] \leq C_3\sqrt{\exp\left(\frac{2}{3}A^{2}_{N}\right)\log^{6\beta+1}(N)N^{-\frac{2\beta}{2\beta+1}}+\sqrt{\underset{t\in[0,1]}{\sup}{\P(\{|X_t|>A_{N_i}\}\cap\{N_i>1\})}}}.
\end{equation*}
Using the Markov inequality, for all $t\in[0,1]$, we have 
\begin{align*}
\P(\{|X_t|>A_{N_i}\}\cap\{N_i>1\}) = & \E\left[\P\left(\{\exp(4|X_t|^2)>\exp(4A^{2}_{N_i})\}\cap\{N_i>1\}|\one_{Y_1=i},\cdots,\one_{Y_N=i}\right)\right]\\
\leq & \E\left[\exp(4|X_t|^2)\right]\E\left[\exp(-4A^{2}_{N_i})\one_{N_i>1}\right]
\end{align*}
and since $\sigma^{*}(.)=1$ and under Assumption~\ref{ass:Re-entrant.Drift}, there exists a constant $C_{*}>0$ such that $\E\left[\exp(4|X_t|^2)\right]\leq C_{*}$ (according to \cite{gobet2002lan}, Proposition 1.1). Thus, there exists a constant $C>0$ such that
\begin{equation}
\label{eq:SecondInq}
    \E\left[\left\|\w{b}_i-b^{*}_{i}\right\|_n\one_{N_i>1}\right] \leq C\left(\exp\left(\frac{1}{3}A^{2}_{N}\right)\log^{3\beta+1}(N)N^{-\beta/(2\beta+1)}\right)+C\E\left[\exp(-4A^{2}_{N_i})\one_{N_i>1}\right].
\end{equation}
From Equations~\eqref{eq:SecondInq}~and~\eqref{eq:ExRiskUp}, we finally obtain
\begin{align*}
    \E\left[\cR(\w{g})-\cR(g^{*})\right] \leq C\log^{3\beta+1}(N)N^{-3\beta/4(2\beta+1)}
\end{align*}
with $A_{N_i} \textcolor{black}{\leq} \sqrt{\frac{3\beta}{4(2\beta+1)}\log(N_i)}$ and $C>1$ a new constant.
\end{proof}

\bibliographystyle{ScandJStat}
\bibliography{mabiblio.bib}

\newpage 

\section*{Appendix}

\begin{proof}[\textbf{Proof of Lemma \ref{lem:discrete}}]
Let $s,t\in[0,1]$ with $s<t$, and $q \geq 1$. By convexity of $x \mapsto |x|^{2q}$, we have
\begin{equation*}
|X_t-X_s|^{2q} \leq   2^{2q-1} \left(\left|\int_{s}^{t}b^{*}_{Y}(X_u)du\right|^{2q}+\left|\int_{s}^{t}{\sigma(X_u)dW_u}\right|^{2q}\right) 
\end{equation*}
Then, from Jensen's inequality, we have
\begin{equation*}
 \left|\int_{s}^{t}b^{*}_{Y}(X_u)du\right|^{2q}\leq (t-s)^{2q-1} \int_{s}^t \left|b^{*}_{Y}(X_u)\right|^{2q} du,
\end{equation*}
Hence, under Assumption~\ref{ass:RegEll} on function $b^*_Y$, we deduce that
\begin{equation*}
 \E\left[\left|\int_{s}^{t}b^{*}_{Y}(X_u)du\right|^{2q}\right]
 \leq C_q(t-s)^{2q}\left(1+\E\left[\sup_{t \in [0,1]}\left|X_s\right|^{2q}\right]\right),
\end{equation*}
and 
%Now, under Assumption~\ref{ass:RegEll}, $\E\left(\int_{s}^{t}{\sigma^{2m}(X_u)du}\right)<+\infty$ for  $0\leq s<t\leq 1$, and $m>0$.
using Burkholder-Davis-Gundy inequality, we obtain
\begin{align*}
	\forall m>0, \ \ \ \E\left[\left(\int_{s}^{t}{\sigma(X_u)dW_u}\right)^{2m}\right]&\leq C_m\E\left[\left(\int_{s}^{t}{\sigma^{2}(X_u)du}\right)^m\right]\leq C_m\sigma^{2m}_{1}(t-s)^{m}.
\end{align*}
From the above equalities, we get
%	\begin{align*}
 %|X_t-X_s|^{2q}&\leq 2^{2q-1}\left((t-s)^{2q}\int_{s}^{t}{\left|b^{*}_{Y}(X_u)\right|^{2q}du}+\left|\int_{s}^{t}{\sigma(X_u)dW_u}\right|^{2q}\right)\\
 %&\leq C(q,C_0)\left((t-s)^{2q+1}\left(1+\underset{s\in[0,1]}{\sup}{|X_s|^{2q}}\right)+\left|\int_{s}^{t}{\sigma(X_u)dW_u}\right|^{2q}\right).
%\end{align*}
%Hence, from the above inequality, we deduce  
%\begin{equation}
%\label{eq:eqlemDiscrete}
%\E\left|X_t-X_s\right|^{2q} \leq C(q)\left((t-s)^{2q}\left(1+\E\left[\underset{s\in[0,1]}{\sup}{|X_s|^{2q}}\right]\right)+\E\left[\left|\int_{s}^{t}{\sigma(X_u)dW_u}\right|^{2q}\right]\right).
%\end{equation}

%(cf ~\cite{karatzas2014brownian}, %~\cite{revuz2013continuous}, %~\cite{le2016brownian}~). 

%Assumption~\ref{ass:RegEll}, $\E\left(\int_{s}^{t}{\sigma^{2m}(X_u)du}\right)<+\infty$ for  $0\leq s<t\leq 1$, and $m>0$.

%which implies, with Equation~\eqref{eq:eqlemDiscrete}, that 
Finally, as the process has finite moments, we obtain that 
\begin{equation*}
\E\left|X_t-X_s\right|^{2q}\leq C(t-s)^q
\end{equation*}
where $C$ is a constant depending on $q,L_{0}$, and $\sigma_{1}$.
\end{proof}

\begin{proof}[\textbf{Proof of Lemma~\ref{lm:MinEigenValue}~}]
For all $i\in\cY$ and on the event $\{N_i>1\}$, let us consider a vector 

$\left(x_{-M},\cdots,x_{K_{N_i}-1}\right)\in\mathbb{R}^{K_{N_i}+M}$ such that $x_j\in[u_{j+M},u_{j+M+1})$ and $B_{j,M,\mathbf{u}}(x_j)\neq 0$. Since $[u_{j+M},u_{j+M+1})\cap[u_{j^{\prime}+M},u_{j^{\prime}+M+1})=\emptyset$ for all $j,j^{\prime}\in\{-M,\cdots,K_{N_i}-1\}$ such that $j\neq j^{\prime}$, then for all $j,j^{\prime}\in\{-M,\cdots,K_{N_i}-1\}$ such that $j\neq j^{\prime}, B_{j,M,\mathbf(u)}(x_{j^{\prime}})=0$. Consequently, we obtain:
\begin{align*}
    \det\left(\left(B_{\ell,M,\mathbf{u}}(x_{\ell^{\prime}})\right)_{-M\leq\ell,\ell^{\prime}\leq K_{N_i}-1}\right)&=\det\left(\mathrm{diag}\left(B_{-M,M,\mathbf{u}}(x_M),\cdots,B_{K_{N_i}-1,M,\mathbf{u}}(x_{K_{N_i}-1})\right)\right)\\
    &=\prod_{\ell=-M}^{K_{N_i}-1}{B_{\ell,M,\mathbf{u}}(x_{\ell})}\neq 0.
\end{align*}
Then, we deduce from \cite{comte2020nonparametric}, \textit{Lemma 1} that the matrix $\Psi_{K_{N_i}}$ is invertible for all $K_{N_i}\in\mathcal{K}_{N_i}$, where the interval $[-A_{N_i},A_{N_i}]$ and the function $f_T$ is replaced by $f_n : x\mapsto\frac{1}{n}\sum_{k=0}^{n-1}{p(k\Delta,x)}$ with $\lambda([-A_{N_i},A_{N_i}]\cap\mathrm{supp}(f_n))>0$, $\lambda$ being the Lebesgue measure. \\

For all $w\in\mathbb{R}^{K_{N_i}+M}$ such that $\|w\|_{2,K_{N_i}+M}=1$, we have:
\begin{equation*}
w^{\prime}\Psi_{K_{N_i}}w=\|h_{w}\|^{2}_{n}=\int_{-A_{N_i}}^{A_{N_i}}{h^{2}_{w}(x)f_n(x)dx}+\frac{h^{2}_{w}(x_0)}{n} \ \quad  \mathrm{with} \ \ h_w=\sum_{\ell=-M}^{K_{N_i}-1}{w_{\ell}B_{\ell,M,\mathbf{u}}}.
\end{equation*}

Since $\sigma^{*}=1$, according to Lemma~\ref{lm:DensityConstSigma}, under Assumption~\ref{ass:RegEll}, the transition density satisfies:
\begin{equation*}
    \forall (t,x)\in(0,1]\times\mathbb{R}, \ \ \frac{1}{K_q\sqrt{t}}\exp\left(-\frac{(2q-1)x^2}{2qt}\right)\leq \textcolor{black}{p_X(t,x)} \leq\frac{K_q}{\sqrt{t}}\exp\left(-\frac{x^2}{2qt}\right) \ \ \mathrm{where} \ \ K_q>1 \ \ \mathrm{and} \ \ q>1.
\end{equation*}

We set $q=3/2$, thus, since $s\mapsto\exp\left(-(2q-1)x^2/2qs\right)$ is an increasing function, we have on the event $\{N_i>1\}$ and for all $x\in[-A_{N_i},A_{N_i}]$,
\begin{align*}
    f_n(x)&\geq\frac{1}{Cn}\sum_{k=1}^{n-1}{\exp\left(-\frac{2x^2}{3k\Delta}\right)}\geq\frac{1}{C}\int_{0}^{(n-1)\Delta}{\exp\left(-\frac{2x^2}{3s}\right)ds}\\
    &\geq\frac{1}{C}\int_{1-\log^{-1}(N_{i})}^{1-2^{-1}\log^{-1}(N_i)}{\exp\left(-\frac{2x^2}{3s}\right)ds}\\
    &\geq\frac{1}{2C\log(N_i)}\exp\left(-\frac{2A^{2}_{N_i}}{3(1-\log^{-1}(N_i))}\right).
\end{align*}

Finally, since there exists a constant $C_1>0$ such that $\|h_w\|^{2}\geq C_1A_{N_i}K^{-1}_{N_i}$ (see \cite{denis2020ridge}, Lemma 2.6), for all $w\in\mathbb{R}^{K_{N_i}+M}$ such that $\|w\|_{2,K_{N_i}+M}=1$, there exists constants $C^{\prime},C>0$ such that,
\begin{equation*}
    w^{\prime}\Psi_{K_{N_i}}w\geq\frac{C^{\prime}A_{N_i}}{K_{N_i}\log(N_i)}\exp\left(-\frac{2A^{2}_{N_i}}{3(1-\log^{-1}(N_i))}\right)\geq\frac{CA_{N_i}}{K_{N_i}\log(N_i)}\exp\left(-\frac{2}{3}A^{2}_{N_i}\right).
\end{equation*}

Furthermore, we set $w_0=e_{K_{N_i}-1}\in\mathbb{R}^{K_{N_i}+M}$ where for all $\ell\in[\![-M,K_{N_i}-1]\!]$,
$$\left[e_{K_{N_i}-1}\right]_{\ell}:=\delta_{\ell,K_{N_i}-1}=\begin{cases} 0 \ \ \mathrm{if} \ \ \ell\neq K_{N_i}-1 \\ 1 \ \ \mathrm{else}.\end{cases}$$

We have,
\begin{align*}
    w^{\prime}_{0}\Psi_{K_{N_i}}w_0&=\int_{-A_{N_i}}^{A_{N_i}}{B^{2}_{K_{N_i}-1,M,\mathbf{u}}(x)f_n(x)}+\frac{B_{K_{N_i}-1,M,\mathbf{u}}(0)}{n}\\
    &\leq\frac{C}{n}\sum_{k=1}^{n-1}{\frac{1}{\sqrt{k\Delta}}\exp\left(-\frac{u^{2}_{K_{N_i}-1}}{3k\Delta}\right)}\left\|B_{K_{N_i}-1,M,\mathbf{u}}\right\|^{2}+\frac{1}{n}\\
    &\leq\frac{CC_1A_{N_i}K^{-1}_{N_i}}{n}\sum_{k=1}^{n-1}{\frac{1}{\sqrt{k\Delta}}\exp\left(-\frac{\alpha^{2}_{N_i}}{3k\Delta}\right)}+\frac{1}{n}
\end{align*}
where $\alpha_{N_i}=A_{N_i}(K_{N_i}-2)/K_{N_i}, \left\|B_{K_{N_i}-1,M,\mathbf{u}}\right\|^{2}\leq C_1A_{N_i}K^{-1}_{N_i}$ (see \cite{denis2020ridge}, Lemma 2.6) and $C_1>0$ is a constant. Since the function $s\mapsto\exp\left(-\alpha^{2}_{N_i}/3s\right)/\sqrt{s}$ is increasing, we deduce that $$n^{-1}\sum_{k=1}^{n-1}{\frac{1}{\sqrt{k\Delta}}\exp\left(-\alpha^{2}_{N_i}/3k\Delta\right)}\leq n^{-1}\sum_{k=1}^{n-1}{\exp\left(-\alpha^{2}_{N_i}/3\right)},$$
and for $N$ large enough,
\begin{align*}
    w^{\prime}_{0}\Psi_{K_{N_i}}w_0\leq \frac{CA_{N_i}}{K_{N_i}}\exp\left(-\frac{A^{2}_{N_i}}{3}\left(\frac{K_{N_i}-2}{K_{N_i}}\right)^{2}\right)+\frac{1}{n}\leq \frac{C^{\prime}A_{N_i}}{K_{N_i}}\exp\left(-\frac{A^{2}_{N_i}}{3}\left(\frac{K_{N_i}-2}{K_{N_i}}\right)^{2}\right)
\end{align*}
where $C^{\prime}>0$ is a constant and $n\geq N\geq N_i$.
\end{proof}

\begin{proof}[\textbf{Proof of Lemma~\ref{lm:proba-complementary-omega}~}]\label{proof:lemmaprobacom}
Let us remind the reader of the Gram matrix $\Psi_{K_{N_i}}$ given in Equation~\eqref{eq:psimatrix}  for $i\in\cY$,
\begin{equation*}\Psi_{K_{N_i}}=\E\left[\frac{1}{N_in}\mathbf{B}^{\prime}_{K_{N_i}}\mathbf{B}_{K_{N_i}}\right]=\E\left(\w{\Psi}_{K_{N_i}}\right)
\end{equation*}
where, on the event $\{N_i>1\}$, and denoting by $
\mathcal{I}_i:= \{i_1, \ldots, i_{N_i}\}$ the indices $j$ such that $Y_j=i$, 
\begin{equation}
\label{eq:Bmatrix}
\mathbf{B}_{K_{N_i}}:=\begin{pmatrix}
        B_{-M}\left(X^{i_1}_{0}\right) & \dots & \dots & B_{K_{N_i}-1}\left(X^{i_1}_{0}\right) \\
        \vdots & & & \vdots\\
        B_{-M}\left(X^{i_1}_{(n-1)\Delta}\right) & \dots & \dots & B_{K_{N_i}-1}\left(X^{i_1}_{(n-1)\Delta}\right) \\
        \vdots & & & \vdots\\
        B_{-M}\left(X^{i_{N_i}}_{0}\right) & \dots & \dots & B_{K_{N_i}-1}\left(X^{i_{N_i}}_{0}\right) \\
        \vdots & & & \vdots\\
        B_{-M}\left(X^{i_{N_i}}_{(n-1)\Delta}\right) & \dots & \dots & B_{K_{N_i}-1}\left(X^{i_{N_i}}_{(n-1)\Delta}\right)
    \end{pmatrix}\in\mathbb{R}^{N_in\times (K_{N_i}+M)}.
\end{equation}
The empirical counterpart $\w{\Psi}$ is the random matrix given by 
$\w{\Psi}_{K_{N_i}}$ of size $(K_{N_i}+M) \times (K_{N_i}+M)$ is given by
\begin{equation}
\label{eq:psihatmatrix}
\w{\Psi}_{K_{N_i}}:=\frac{1}{N_in}\mathbf{B}^{\prime}_{K_{N_i}}\mathbf{B}_{K_{N_i}}=\left(\frac{1}{N_in}\sum_{j=1}^{N_i}{\sum_{k=0}^{n-1}{B_{\ell}(X^{i_j}_{k\Delta})B_{\ell^{\prime}}(X^{i_j}_{k\Delta})}}\right)_{\ell,\ell^{\prime}\in[-M,K_{N_i}-1]}.
\end{equation}
We build an orthonormal basis $\theta=(\theta_{-M},\cdots,\theta_{K_{N_i}-1})$ of the subspace $\mathcal{S}_{K_{N_i},M}$ with respect to the $\mathbb{L}^{2}$ inner product $\left<.,.\right>$ through the Gram-Schmidt orthogonalization of the spline basis $(B_{-M},\cdots,B_{K_{N_i}-1})$. Then, we have 
$$\mathrm{Span}(B_{-M},\cdots,B_{K_{N_i}-1})=\mathrm{Span}(\theta_{-M},\cdots,\theta_{K_{N_i}-1})=\mathcal{S}_{K_{N_i},M}$$ 
and the matrix given in Equation~\eqref{eq:Bmatrix}~ is factorized as follows
\begin{equation}
    \label{eq:Bmatrix-factorization}
    \mathbf{B}_{K_{N_i}}=\mathbf{\Theta}_{K_{N_i}}\mathbf{R}_{K_{N_i}}
\end{equation}
where
\begin{align*}
    \mathbf{\Theta}_{K_{N_i}}=\left(\left(\theta_{\ell}(X^{i_j}_{0}),\theta_{\ell}(X^{i_j}_{\Delta}),\cdots,\theta_{\ell}(X^{i_j}_{n\Delta})\right)^{\prime}\right)_{\underset{-M \leq \ell \leq K_{N_i}-1}{1 \leq j \leq N_i}} \in \R^{N_in \times (K_{N_i}+M)}
\end{align*}
and $\mathbf{R}_{K_{N_i}}$ is an upper triangular matrix of size $(K_{N_i}+M)\times(K_{N_i}+M)$ see~\cite{leon2013gram}). Let $\Phi_{K_{N_i}}$ be the Gram matrix under the orthonormal basis $\theta=\left(\theta_{-M},\cdots,\theta_{K_{N_i}-1}\right)$ and given by
\begin{equation*}
    \Phi_{K_{N_i}}=\E\left[\frac{1}{N_in}\mathbf{\Theta}^{\prime}_{K_{N_i}}\mathbf{\Theta_{K_{N_i}}}\right]=\E\left(\w{\Phi}_{K_{N_i}}\right)
\end{equation*}
where,
\begin{equation}
\label{eq:phihatmatrix}
\w{\Phi}_{K_{N_i}}:=\frac{1}{N_in}\mathbf{\Theta}^{\prime}_{K_{N_i}}\mathbf{\Theta}_{K_{N_i}}=\left(\frac{1}{N_in}\sum_{j=1}^{N_i}{\sum_{k=0}^{n-1}{\theta_{\ell}(X^{i_j}_{k\Delta})\theta_{\ell^{\prime}}(X^{i_j}_{k\Delta})}}\right)_{\ell,\ell^{\prime}\in[-M,K_{N_i}-1]}.
\end{equation}
The matrices $\Psi_{K_{N_i}}$ and $\w{\Psi}_{K_{N_i}}$ are respectively linked to the matrices $\Phi_{K_{N_i}}$ and $\w{\Phi}_{K_{N_i}}$ through the following relations
\begin{equation*}
    \Psi_{K_{N_i}}= {\bf R}^{\prime}_{K_{N_i}}\Phi_{K_{N_i}} {\bf R}_{K_{N_i}} \ \ \mathrm{and} \ \ \w{\Psi}_{K_{N_i}}= {\bf R}^{\prime}_{K_{N_i}}\w{\Phi}_{K_{N_i}} {\bf R}_{K_{N_i}}
\end{equation*}
Since for all $h=\sum_{\ell=-M}^{K_{N_i}-1} a_{\ell} B_{\ell,M,{\bf u}}\in S_{K_{N_i}, M}$ one has
$$\|h\|_{n,N_i}^2 = a^{\prime} \w{\Psi}_{K_{N_i}} a \ \ \mathrm{and} \ \ \|h\|_{n,i}^2 = a^{\prime} \Psi_{K_{N_i}} a, \ \ \mathrm{with} \ \ a=\left(a_{-M},\cdots,a_{K_{N_i}-1}\right)^{\prime},$$
we deduce that
$$\|h\|_{n,N_i}^2 = w^{\prime} \w{\Phi}_{K_{N_i}} w \ \ \mathrm{and} \ \ \|h\|_{n,i}^2 = w^{\prime} \Phi_{K_{N_i}} w, \ \ \mathrm{with} \ \ w={\bf R}_{K_{N_i}}a.$$
Under Assumption~\ref{ass:RegEll}, we follow the lines of ~\cite{comte2020regression} \textit{Proposition 2.3} and \textit{Lemma 6.2}. Then,
\begin{eqnarray*}
\sup _{h \in S_{K_{N_i},M},\|h\|_{n,i}=1}\left|\|h\|_{n,N_i}^2-\|h\|_{n,i}^2\right|&=&\sup _{w \in \R^{K_{N_i}+M},\left\|\Phi_{K_{N_i}}^{1 / 2} w\right\|_{2, K_{N_i}+M}=1}\left|w^{\prime}\left(\w{\Phi}_{K_{N_i}}-\Phi_{K_{N_i}}\right) w\right| \\
&=&\sup _{u \in \mathbb{R}^{K_{N_i}+M},\|u\|_{2, K_{N_i}+M}=1}\left|u^{\prime} \Phi_{K_{N_i}}^{-1 / 2}\left(\w{\Phi}_{K_{N_i}}-\Phi_{K_{N_i}}\right) \Phi_{K_{N_i}}^{-1 / 2} u\right| \\
&=&\left\|\Phi_{K_{N_i}}^{-1 / 2} \w{\Phi}_{K_{N_i}} \Phi_{K_{N_i}}^{-1 / 2}-\operatorname{Id}_{K_{N_i}+M}\right\|_{\mathrm{op}}.
\end{eqnarray*}
Therefore,  
$$
\Omega_{n, N_i, K_{N_i}}^c=\left\{\left\|\Phi_{K_{N_i}}^{-1 / 2} \w{\Phi}_{K_{N_i}} \Phi_{K_{N_i}}^{-1 / 2}-\operatorname{Id}_{K_{N_i}+M}\right\|_{\mathrm{op}} > 1 / 2\right\}.
$$
%%%%%%%%%%
Then, we apply here Theorem 1 of \cite{cohen2013}, it yields  
\begin{equation}
\label{eq:OmegaComp-bound1}
    \P_i\left(\Omega^{c}_{n,N_i,K_{N_i}}\right)\leq 2(K_{N_i}+M) \exp\left(- c_{1/2} \frac{N_i}{\mathcal{L}(K_{N_i}+M)(\|\Phi^{ -1}_{K_{N_i}}\|_{\mathrm{op}} \vee 1)}\right)
\end{equation}
with $c_{1/2}=(3\log(3/2)-1)/2$ and
$\mathcal{L}(K_{N_i}+M):=\underset{x\in[-A_{N_i},A_{N_i}]}{\sup}{\sum_{\ell=-M}^{K_{N_i}-1}{\theta^{2}_{\ell}(x)}}$
(from application of Lemma 6.2 from \cite{comte2020regression}).
For all $h=\sum_{\ell=-M}^{K_{N_i}-1}{w_{\ell}\theta_{\ell}}\in\mathrm{Span}\left(\theta_{-M},\cdots,\theta_{K_{N_i}-1}\right)=\mathcal{S}_{K_{N_i},M}$, we have
\begin{align*}
    \|h\|^{2}=\|w\|^{2}_{2,K_{N_i}+M} \ \ \mathrm{and} \ \ \|h\|^{2}_{n,i}=1 \ \ \mathrm{implies} \ \ w=\Phi^{-1/2}_{K_{N_i}}u \ \ \mathrm{where} \ \ u\in\R^{K_{N_i}+M}: \|u\|_{2,K_{N_i}+M}=1.
\end{align*}
We deduce that
\begin{equation*}
    \underset{h\in\mathcal{S}_{K_{N_i}+M}, \ \|h\|^{2}_{n,i}=1}{\sup}{\|h\|^2}=\underset{u\in\R^{K_{N_i}+M}, \|u\|_{2,K_{N_i}+M}=1}{\sup}{\ u^{\prime}\Phi^{-1}_{K_{N_i}}u}=\left\|\Phi^{-1}_{K_{N_i}}\right\|_{\mathrm{op}}.
\end{equation*}
Furthermore, for all $h=\sum_{\ell=-M}^{K_{N_i}-1}{a_{\ell}B_{\ell}}\in\mathrm{Span}\left(B_{-M},\cdots,B_{K_{N_i}-1}\right)=\mathcal{S}_{K_{N_i},M}$, we have on one side
$$\|h\|^{2}_{n,i}=1 \ \ \mathrm{implies} \ \ a=\Psi^{-1/2}_{K_{N_i}}u \ \ \mathrm{where} \ \ u\in\R^{K_{N_i}+M}: \|u\|_{2,K_{N_i}+M}=1$$
and on the other side, for all $h\in\mathcal{S}_{K_{N_i}+M}$ such that $\|h\|^{2}_{n,i}=1$, from \cite{denis2020ridge} \textit{Lemma 2.6}, there exists a constant $C>0$ such that,
$$\|h\|^{2}\leq C A_{N_i}K^{-1}_{N_i} \|a\|^{2}_{2,K_{N_i}+M}=C A_{N_i}K^{-1}_{N_i} u^{\prime}\Psi^{-1}_{K_{N_i}}u.$$
Then we have $a.s$
\begin{equation}
\label{eq:Norms:phimatrix-psimatrix}
\left\|\Phi^{-1}_{K_{N_i}}\right\|_{\mathrm{op}}=\underset{h\in\mathcal{S}_{K_{N_i}+M}, \ \|h\|^{2}_{n,i}=1}{\sup}{\|h\|^2}\leq \frac{CA_{N_i}}{K_{N_i}} \underset{u\in\R^{K_{N_i}+M}, \|u\|_{2,K_{N_i}+M}=1}{\sup}{\ u^{\prime}\Psi^{-1}_{K_{N_i}}u} = \frac{CA_{N_i}}{K_{N_i}}  \left\|\Psi^{-1}_{K_{N_i}}\right\|_{\mathrm{op}}.
\end{equation}
\textcolor{black}{From Equations~\eqref{eq:OmegaComp-bound1}~and~\eqref{eq:Norms:phimatrix-psimatrix}, there exists a constant $C>0$ such that}
\begin{equation}
\label{eq:OmegaComp-bound2}
    \textcolor{black}{\P_i\left(\Omega^{c}_{n,N_i,K_{N_i}}\right)\leq 2(K_{N_i}+M)\exp\left(-C \frac{N_iK_{N_i}}{A_{N_i}\mathcal{L}(K_{N_i}+M)\left\|\Psi^{-1}_{K_{N_i}}\right\|_{\mathrm{op}}}\right).}
\end{equation}
\textcolor{black}{We have $\mathcal{L}(K_{N_i}+M):=\underset{x\in[-A_{N_i},A_{N_i}]}{\sup}{\sum_{\ell=-M}^{K_{N_i}-1}{\theta^{2}_{\ell}(x)}}$ and the functions $\theta_{\ell}, ~ \ell = -M,\ldots,K_{N_i}-1$ are given by}
\begin{align*}
    \textcolor{black}{\theta_{-M} =} &~ \textcolor{black}{\frac{f_{-M}}{\|f_{-M}\|} ~~ \mathrm{with} ~~ f_{-M} = B_{-M} } \\
    \textcolor{black}{\theta_{\ell} = } &~ \textcolor{black}{\frac{f_{\ell}}{\|f_{\ell}\|} ~~ \mathrm{with} ~~ f_{\ell} = B_{\ell} - \sum_{k=-M}^{\ell-1}{\left<B_{\ell},\theta_{k}\right>\theta_{k}}, ~~ \ell = -M+1,\ldots,K_{N_i}-1. }
\end{align*}
\textcolor{black}{Note that for all $x\in[-A_{N_i},A_{N_i}]$, there exists $\ell \in [\![-M,K_{N_i}-1]\!]$ such that $x \in [u_{\ell},u_{\ell+1})$. Then, $x \in [u_{\ell^{\prime}},u_{\ell^{\prime}+M+1})$ for all $\ell^{\prime} \in [\![\ell-M,\ell]\!]$ if $\ell \geq 0$ and $\ell^{\prime} \in [\![-M,\ell]\!]$ for $\ell \leq -1$. Thus, for each $x\in[-A_{N_i}, A_{N_i}]$, there exists at most $M+1$ spline functions that don't vanish at $x$. As a result, we have on one side,}
\begin{equation}
\label{eq:Reduction-SumOfTheta}
    \textcolor{black}{\forall~x\in[-A_{N_i},A_{N_i}],~~ \sum_{\ell = -M}^{K_{N_i}-1}{\theta^{2}_{\ell}(x)} = \sum_{j=1}^{M+1}{\theta^{2}_{\ell_{j}}(x)}}
\end{equation}
\textcolor{black}{where for all $ x\in[-A_{N_i},A_{N_i}]$, there exists integers $\ell_{j},~ j \in [\![1,M+1]\!]$ such that 
$$x \in \bigcap_{j=1}^{M+1}{[u_{\ell_j},u_{\ell_j+M+1})} ~~ \mathrm{and} ~~ x \notin [-A_{N_i},A_{N_i}] \setminus \bigcap_{j=1}^{M+1}{[u_{\ell_j},u_{\ell_j+M+1})}.$$ 
One the other side, for all $\ell \in [\![1,K_{N_i}-1]\!]$ and for all $x\in[u_{\ell},u_{\ell+M+1})$ there exists at most $M+1$ integers $\ell_1,\ldots,\ell_{M+1}$ such that}
\begin{equation*}
    \textcolor{black}{\theta_{\ell}(x) = \frac{f_{\ell}(x)}{\|f_{\ell}\|} ~~ \mathrm{and} ~~ f_{\ell}(x) = B_{\ell}(x) - \sum_{j=1}^{M+1}{\left<B_{\ell},\theta_{\ell_j}\right>\theta_{\ell_j}(x)}.}
\end{equation*}
\textcolor{black}{Now we focus on the supremum norm of each basis function $\theta_{\ell}, ~ \ell=-M, \ldots, K_{N_i}-1$. For all each $\ell \in [\![-M,K_{N_i}-1]\!]$, since the spline function $B_{\ell}$ is non-zero, positive and continuous on the interval $[u_{\ell},u_{\ell+M+1})$, there exists an interval $[\alpha_{\ell},\beta_{\ell}] \subset [u_{\ell},u_{\ell+M+1})$ such that $c_{\ell} = \underset{x\in[\alpha_{\ell},\beta_{\ell}]}{\inf}{B_{\ell}(x)} > 0$ where $(\alpha_{\ell} - \beta_{\ell}) \propto A_{N_i}/K_{N_i}$ since $\int_{\alpha_{\ell}}^{\beta_{\ell}}{B_{\ell}(x)dx} \propto A_{N_i}/K_{N_i}$. Then we have
}
\begin{equation}
\label{eq:L2Norm-Bl}
   \textcolor{black}{\forall~\ell\in[\![-M,K_{N_i}-1]\!], ~~ \left\|B_{\ell}\right\|^{2} = \int_{u_{\ell}}^{u_{\ell+M+1}}{B^{2}_{\ell}(x)dx} \geq c_{\ell}\int_{\alpha_{\ell}}^{\beta_{\ell}}{B_{\ell}(x)dx} = C_{\ell}\frac{A_{N_i}}{K_{N_i}}}
\end{equation}
\textcolor{black}{where the constant $C_{\ell}>0$ depends on $c_{\ell} = \underset{x\in[\alpha_{\ell},\beta_{\ell}]}{\inf}{B_{\ell}(x)} > 0$. Then, for $\ell = -M$, there exists a constant $C_{-M}$ such that $\theta^{2}_{-M}(x) \leq C_{-M}K_{N_i}$ and for each $\ell \geq -M+1$, since the function $f_{\ell}$ depends on splines functions $B_{-M},\ldots,B_{\ell}$ and only $B_{\ell}$ does not vanish on the interval $[u_{\ell+M},u_{\ell+M+1})$, we obtain that} 
\begin{equation*}
    \textcolor{black}{\left\|f_{\ell}\right\|^{2} = \int_{-A_{N_i}}^{A_{N_i}}{f^{2}_{\ell}(x)} \geq \int_{u_{\ell+M}}^{u_{\ell+m+1}}{B^{2}_{\ell}(x)dx}.}
\end{equation*}
\textcolor{black}{Moreover, since $B_{\ell}$ is non-zero, positive and continue on the interval $[u_{\ell+M},u_{\ell+M+1})$, there exists an interval $[\alpha_{\ell},\beta_{\ell}] \subset [u_{\ell+M},u_{\ell+M+1})$ with $(\alpha_{\ell} - \beta_{\ell}) \propto A_{N_i}/K_{N_i}$ such that $c_{\ell} = \underset{x\in[\alpha_{\ell},\beta_{\ell}]}{\inf}{B_{\ell}(x)} > 0$. Then we obtain}
\begin{equation}
\label{eq:L2Norm-fl}
     \textcolor{black}{\left\|f_{\ell}\right\|^{2} \geq c_{\ell}\int_{\alpha_{\ell}}^{\beta_{\ell}}{B_{\ell}(x)dx} = C\frac{A_{N_i}}{K_{N_i}}, ~~ \ell \in [\![-M+1,\ldots,K_{N_i}-1]\!]}
\end{equation}
\textcolor{black}{where $C >0$ is a constant depending on $\underset{\ell=-M+1,\ldots,K_{N_i}-1}{\min}{c_{\ell}} >0$. On the other side, for all $\ell \in [\![-M+1,K_{N_i}-1]\!]$ and for all $x \in [-A_{N_i},A_{N_i}]$,}
\begin{equation}
\label{eq:UpperBound-fl}
    \textcolor{black}{\left|f_{\ell}(x)\right| \leq \left|B_{\ell}(x)\right| + \sum_{j=-M}^{\ell-1}{\frac{\left<B_{\ell},f_j\right>}{\|f_j\|^{2}}|f_j(x)|} \leq 1 + C\sum_{j=-M}^{\ell-1}{|f_j(x)|}}
\end{equation}
\textcolor{black}{where the constant $C>0$ is the upper-bound of $\left<B_{\ell},f_j\right>/\|f_j\|^{2} \leq \|B_{\ell}\|/\|f_j\|$ according to Equations \eqref{eq:L2Norm-Bl}~and~\eqref{eq:L2Norm-fl}. For $\ell = -M$, we have $\|f_{-M}\|_{\infty} < \infty$. Let $\ell \in [\![-M+1,K_{N_i}-1]\!]$. Assume that the functions $f_{-M},\ldots,f_{\ell-1}$ are all bounded, then by recurrence hypothesis, we have from Equation~\eqref{eq:UpperBound-fl} that}
\begin{equation*}
    \textcolor{black}{\|f_{\ell}\|_{\infty} \leq 1 + C\sum_{j=-M}^{\ell-1}{\|f_{j}\|_{\infty}} < \infty.}
\end{equation*}
\textcolor{black}{Thus, we obtain by recurrence that the functions $f_{\ell},~ \ell=-M,\ldots,K_{N_i}-1$ are bounded and finally conclude from Equation~\eqref{eq:Reduction-SumOfTheta} that} 
\begin{equation*}
    \textcolor{black}{\mathcal{L}(K_{N_i}+M):=\underset{x\in[-A_{N_i},A_{N_i}]}{\sup}{\sum_{\ell=-M}^{K_{N_i}-1}{\theta^{2}_{\ell}(x)}} \leq CK_{N_i}}
\end{equation*}
\textcolor{black}{where the constant $C>0$ depends on the spline basis. We deduce from Equation~\eqref{eq:OmegaComp-bound2} that there exists a constant $C>0$ such that}  
\begin{equation}
    \label{eq:OmegaComp-bound2-bis}
    \textcolor{black}{\P_i\left(\Omega^{c}_{n,N_i,K_{N_i}}\right)\leq 2(K_{N_i}+M)\exp\left(-C \frac{N_i}{A_{N_i}\left\|\Psi^{-1}_{K_{N_i}}\right\|_{\mathrm{op}}}\right).}
\end{equation}
Furthermore, since $A_{N_i}\leq \sqrt{\frac{3\beta}{2\beta+1}\log(N_i)}, ~ \textcolor{black}{K_{N_i} \propto \log^{-5/2}(N_i)N^{1/(2\beta+1)}_{i}}$ \textcolor{black}{and from Lemma~\ref{lm:MinEigenValue}}, we obtain from \textcolor{black}{Equation~\eqref{eq:OmegaComp-bound2-bis}},
\begin{equation}
\label{eq:OmegaComp-bound3}
    \P_i\left(\Omega^{c}_{n,N_i,K_{N_i}}\right)\leq 2(K_{N_i}+M)\exp\left(-C \log^{3/2}(N_i)\right)
\end{equation}
where $C>0$ is a new constant depending on $C_{\theta}, \beta$ and $M$. Since $N_i\longrightarrow\infty \ a.s.$ as $N\longrightarrow\infty$, one has
\begin{equation*}
    \exp\left(\log(N_i)-C\log^{3/2}(N_i)\right)\longrightarrow 0 \ a.s. \ \ \mathrm{as} \ \ N\longrightarrow\infty.
\end{equation*}
Then, for $N$ large enough, $\exp\left(\log(N_i)-C\log^{3/2}(N_i)\right)\leq 1 \ \ a.s.$ and from Equation~\eqref{eq:OmegaComp-bound3},
\begin{equation*}
\P_i\left(\Omega^{c}_{n,N_i,K_{N_i}}\right)\leq\frac{2(K_{N_i}+M)}{N_{i}}\leq c\frac{K_{N_i}}{N_i}
\end{equation*}
where the constant $c>0$ depends on $M$.
\end{proof}

\end{document}